\documentclass[12pt,reqno]{amsart}
\usepackage{amsmath,amssymb,amsthm,mathtools,wasysym,calc,verbatim,enumitem,tikz,pgfplots,url,mathrsfs,fullpage}
\pgfplotsset{compat=1.18}
\SetSymbolFont{wasy}{bold}{U}{wasy}{m}{n}

\usepackage{comment}

\newtheorem{theorem}{Theorem}[section]
\newtheorem{definition}[theorem]{Definition}

\newtheorem{lemma}[theorem]{Lemma}
\newtheorem{corollary}[theorem]{Corollary}
\newtheorem{proposition}[theorem]{Proposition}

\theoremstyle{definition}

\theoremstyle{remark}

\usepackage[colorlinks = true,
            linkcolor  = blue,
            citecolor  = blue,
            urlcolor   = blue,
            bookmarks  = true]{hyperref}

\theoremstyle{remark}

\numberwithin{equation}{section}

\usepackage[backend=biber, style=numeric, citestyle=numeric-comp]{biblatex}
\DeclareFieldFormat{labelnumberwidth}{\mkbibbrackets{#1}}
\DeclareNameAlias{author}{family-given}

\DeclareFieldFormat*{title}{#1}

\AtEveryBibitem{%
  \clearfield{issn}
  \clearfield{url}
  \clearfield{eprint}
}

\renewbibmacro{in:}{}

\DeclareFieldFormat[article]{volume}{\textbf{#1}}
\DeclareFieldFormat[article]{number}{\textbf{#1}}
\renewbibmacro*{volume+number+eid}{%
  \printfield{volume}
  \iffieldundef{number}
    {}%
    {\printtext[parens]{\printfield{number}}}
  \setunit{\addcomma\space}%
  \printfield{eid}}

\DeclareFieldFormat[journal]{journaltitle}{\mkbibemph{#1}}

\DeclareNameAlias{author}{family-given}
\DeclareNameAlias{editor}{family-given}

\DeclareDelimFormat{namedelim}{\addcomma\space}
\DeclareFieldFormat[article,incollection,inproceedings,book]{giveninit}{\mkbibnamegiven{#1}}

\DeclareNameFormat{family-given}{%
  \usebibmacro{name:family-given}
    {\namepartfamily}
    {\namepartgiveni} 
    {\namepartprefix}
    {\namepartsuffix}%
  \usebibmacro{name:andothers}}

\DeclareFieldFormat{pages}{#1}

\allowdisplaybreaks[4]

\begin{document}

\title{The singularity probability of dense random regular graphs}

\author{Zeyan Song}
\address{Shandong University, Jinan, 250100, China.}
\email{zeyansong8@gmail.com}

\author{Hanchao Wang}
\address{Shandong University, Jinan, 250100, China.}
\email{wanghanchao@sdu.edu.cn}

\author{Shengyan Yin}
\address{Shandong University, Jinan, 250100, China.}
\email{202620348@mail.sdu.edu.cn}

\author{Kexin Yu}
\address{Shandong University, Jinan, 250100, China.}
\email{kexinyu01@gmail.com}


\begin{abstract}
Let $G_{n,d}$ be a uniformly random simple $d$-regular graph on
$n$ vertices, and let $A_n$ be its adjacency matrix. For every fixed $\lambda\in(0,1/2)$, we prove that for any $\lambda(n-1) \le d \le (1-\lambda)(n-1)$,
\begin{align*}
    \mathbb P(A_n\text{ is singular})\le e^{-cn}
\end{align*}
where $c>0$ depends only on $\lambda$, and $n$ is sufficiently large with $nd$ even.
        
\end{abstract}

\maketitle

\section{Introduction}\label{Sec:introduction}
Let $G_{n,d}$ be uniform on the simple $d$-regular graphs on $[n]$, and let $A_n$ be its adjacency matrix. 
\begin{equation}\label{cond: dense regime}
 \lambda(n-1)\le d\le(1-\lambda)(n-1),
 \quad 0<\lambda <1/2, \quad nd \in 2\mathbb{Z}.
\end{equation}

\begin{theorem}\label{thm:main}
For every $\lambda\in(0,1/2)$, there exist $c_{\ref{thm:main}}>0$ and $n_{\ref{thm:main}}>0$ depending only on $\lambda$ such that
\begin{align*}
 \mathbb{P}(A_n\text{ is singular})\le e^{-c_{\ref{thm:main}}n}
\end{align*}
for every $n\ge n_{\ref{thm:main}}$ and $d$ that satisfies \eqref{cond: dense regime}.
\end{theorem}
The singularity problem for discrete random matrices connects
anti-concentration with arithmetic structure. For independent entries,
Kahn, Koml\'os and Szemer\'edi \cite{KKS-jams} obtained an exponential
singularity bound, inverse Littlewood--Offord theory was developed by
Tao and Vu \cite{TV-aom}, and the LCD approach of Rudelson and
Vershynin \cite{RV-adv} gave quantitative invertibility estimates.
Tikhomirov's inversion of randomness method \cite{Tikhomirov-aom}
resolved the asymptotic Bernoulli singularity problem, and was further
developed by Hu, Song and Wu \cite{HSW-arxiv} to obtain sharp
singularity asymptotics for iid discrete random matrices. These results
motivate counting exceptional coefficient vectors through their small
ball probabilities rather than prescribing their arithmetic form in
advance.

Symmetry creates dependence between a row and the complementary minor.
Costello, Tao and Vu \cite{CTV-duke} proved asymptotic nonsingularity
for symmetric Bernoulli matrices. Campos, Jenssen, Michelen and
Sahasrabudhe \cite{CJMS-jams} obtained an exponential singularity
bound; their work on the least singular value \cite{CJMS-pi} develops
further quantitative control in the symmetric setting. We adopt their Fourier approach to the inversion of randomness argument, which serves as a useful bridge between the random model and its structural analysis. The conditional switching and
hybrid matching constructions below are the additional steps required
by the present fixed-degree model.

For regular directed graphs, switching and anti-concentration were
used by Cook \cite{Cook-ptrf} and by Litvak, Lytova, Tikhomirov,
Tomczak-Jaegermann and Youssef \cite{regular-jmaa,regular-ptrf}.
Jain, Sah and Sawhney \cite{JSS-imrn} proved a dense-digraph least
singular value estimate with an exponential remainder. Undirected
regular graph invertibility was studied by Huang \cite{Huang-duke};
M\'esz\'aros \cite{Meszaros-tams} and Nguyen and Wood \cite{NW-arxiv}
also relate nonsingularity to the distribution of finite abelian
groups associated with regular graph models. These models and
parameter regimes are distinct from the conditional symmetric model
used here. Our argument keeps the original simple-graph law throughout.

The proof combines these ideas with exact exposure identities. Dense
enumeration \cite{McKay-cpc,McKay-ejc} is used only for a bounded number
of incidence rows at a time. The slice Fourier estimate comes from
Roos \cite{Roos}, in the form used by Tran \cite{Tran-arxiv}; the latter
also supplies the CLCD small-ball estimate used in the final
classification. The Gaussian comparison, residual-to-hybrid second
moment, and monotone threshold reclassification are established below
for the operators constructed in this paper. Random rounding is used
in the sense of Livshyts \cite{Livshyts-jam}, with the precise
geometric and image approximations written out before application.

\subsection{Notation}\label{Sec:notation}

We use $\mathfrak c$ for a fixed nonempty atom of a specified
conditional exposure, called a fibre or certificate. Its exposed data
are stated before it is used in a conditional argument.
In graph arguments, $\mathbb P$ and $\mathbb E$ refer to the original
graph law, enlarged by the independent slice columns in
\eqref{eq:extension}. On a stated conditional fibre they refer to its
exact conditional law: the exposed data are fixed and only the
unrevealed graph or extension variables are integrated. In abstract
probability lemmas, they refer to the law specified in that lemma.
Whenever several independent layers coexist, we distinguish them:
$\mathbb P_\omega,\mathbb E_\omega$ integrate the auxiliary matching
and lazy-sign seeds on fixed supports;
$\mathbb P_Z,\mathbb E_Z$ integrate the uniform product measure on a
fixed coefficient box; and $\mathbb P_R,\mathbb E_R$ integrate only
the coordinate choices of the random rounding map defined where it
is used. These laws do not average over certificates. The matching
seed and the coefficient box are constructed in
Definitions~\ref{def:lazy matching walk} and~\ref{def:signed coefficient box};
all supports, coefficient data, and projections held fixed under a
particular expectation are specified in its local context.

For a random vector $X\in\mathbb R^k$ and $r\ge 0$, its L\'evy
concentration function is
\begin{align*}
    \mathcal L(X,r)
    :=
    \sup_{z\in\mathbb R^k}
    \mathbb P\left(\|X-z\|_2\le r\right).
\end{align*}
We write $[n]:=\{1,\dots,n\}$, and denote by $S^{k-1}$ and
$B_2^k$ the Euclidean unit sphere and unit ball in $\mathbb R^k$,
respectively. The standard coordinate vector at index $i$ is denoted
by $\mathbf e_i$, and $\mathbf 1_T$ denotes the indicator vector of a
finite set $T$; $\mathbf 1$ is the all-ones vector in the stated ambient
space. For a vector $x$ and a finite coordinate set $T$, write
$x_T:=x|_T\in\mathbb R^T$ and let $B_2^T$ be the Euclidean unit ball
in $\mathbb R^T$. For a matrix, restriction to $T$ means row restriction,
with all columns retained, unless two index sets are displayed.
Restrictions and indicator vectors are extended by zero when viewed in
a larger coordinate space.

We use the Gaussian measure
\begin{align*}
    d\gamma_k(\theta)
    :=
    e^{-\pi\|\theta\|_2^2}\,d\theta.
\end{align*}
Thus $\gamma_k$ is the centered Gaussian measure with covariance
$(2\pi)^{-1}I_k$. This normalization will be used throughout the
Fourier arguments and in the Gaussian cylinder lemma.

For a $p\times q$ matrix $B$, its singular values are ordered as
\begin{align*}
    \sigma_1(B)
    \ge \cdots
    \ge \sigma_{\min(p,q)}(B)
    \ge 0.
\end{align*}
For $K\subset\mathbb R^k$, we write $N(K,rB_2^k)$ for the least
cardinality of an $r$-net of $K$ in Euclidean distance.

For $t\in\mathbb R$, let
\begin{align*}
    \|t\|_{\mathbb R/\mathbb Z}
    :=
    \operatorname{dist}(t,\mathbb Z).
\end{align*}
If $T$ is finite and $0\le a\le |T|$, the Boolean $a$-slice on $T$ is
\begin{align*}
    \binom{T}{a} := \left\{ x\in\{0,1\}^T: \sum_{i\in T}x_i=a \right\}.
\end{align*}
We identify subsets of $T$ with their indicator vectors, and write
$\operatorname{Unif}\binom{T}{a}$ for the uniform distribution on this
slice.

Let $T$ be finite with $s:=|T|\ge 2$. For $x\in\mathbb R^T$, set
\begin{align*}
    \bar x_T
    &:=
    \frac1s\sum_{i\in T}x_i,\\
    P_{T,0}x
    &:=
    x-\bar x_T\mathbf 1_T,\\
    D_Tx
    &:=
    (x_i-x_j)_{\{i,j\}\in\binom{T}{2}}.
\end{align*}
The map $P_{T,0}$ is the orthogonal projection onto
$\mathbf 1_T^\perp$. For an ambient vector $x$, the notation
$P_{T,0}x$ means $P_{T,0}(x_T)$.
We further define the centered torus energy
\begin{align*}
    \|x_{T}\|_{\mathbb D}^2
    :=
    \frac1{s-1}
    \sum_{\{i,j\}\in\binom{T}{2}}
    \|x_i-x_j\|_{\mathbb R/\mathbb Z}^2.
\end{align*}
Then
\begin{align}
    \|D_Tx\|_2^2
    &=
    s\|P_{T,0}x\|_2^2,
    \notag\\
    \operatorname{dist}
    \left(D_Tx,\mathbb Z^{\binom{s}{2}}\right)^2
    &=
    (s-1)\|x_{T}\|_{\mathbb D}^2.
    \label{eq:difference-identities}
\end{align}

The combinatorial least common denominator introduced by Tran
\cite{Tran-arxiv} measures the arithmetic structure through these coordinate differences. For $\alpha>0$ and $\gamma\in(0,1)$, define
\begin{align*}
    \operatorname{CLCD}_{\alpha,\gamma}^{T}(x)
    :=
    \inf\left\{
        \theta>0:
        \operatorname{dist}
        \left(\theta D_Tx,\mathbb Z^{\binom{s}{2}}\right)
        <
        \min\{
            \gamma\theta\|D_Tx\|_2,\alpha
        \}
    \right\}.
\end{align*}
When the support $T$ is fixed from context, we simply write
$\operatorname{CLCD}_{\alpha,\gamma}(x)$.

For $\delta,\rho\in(0,1)$, define
\begin{align*}
    \mathrm{Comp}_{\delta,\rho}(m)
    :=
    \left\{
        x\in S^{m-1}:
        \operatorname{dist}
        \left(
            x,\{y:|\operatorname{supp}y|\le \delta m\}
        \right)
        \le \rho
    \right\},
\end{align*}
and
\begin{align*}
    \mathrm{Incomp}_{\delta,\rho}(m)
    :=
    S^{m-1}\setminus \mathrm{Comp}_{\delta,\rho}(m).
\end{align*}
We omit the ambient dimension $m$ when it is clear from context.

For a finite set $T$, we write $u\in\mathrm{Cons}_{\delta,\rho}(T)$ if
$u\in\mathbb R^T$, $\|u\|_2=1$, and there exists $a\in\mathbb R$
such that
\begin{align*}
    \left|
        \{i\in T:|u_i-a|\le \rho/\sqrt{|T|}\}
    \right|
    \ge
    (1-\delta)|T|.
\end{align*}

Finally, fix
\begin{align*}
    h:=
    \frac{\mathbf e_1-\mathbf e_2}{\sqrt2}, \
    H_m:=
    h^\perp, \
    P_h:=
    I-hh^T.
\end{align*}
For a symmetric $m\times m$ matrix $B$, we denote its compression to
$H_m$ by
\begin{align*}
    C_B
    :=
    P_hB|_{H_m}.
\end{align*}

Constants supplied by a labeled result are indexed by its number:
$c_{\ref{prop:dense slice Fourier}}$ or $C_{\ref{thm:Fourier capture}}$, for example.
When several constants of the same letter arise from one result, we use
superscripts $(1),(2),\ldots$, as in
$c_{\ref{thm:reduction}}^{(1)},c_{\ref{thm:reduction}}^{(2)}$.
Dependence on a parameter may be displayed as an argument, for example
$c_{\ref{thm:hybrid box inversion}}(L)$. These labeled constants retain
their meaning in subsequent applications. Locally we use
$c,c',c''$ and $C,C',C''$ for temporary constants, specifying when one
must remain fixed; otherwise their values may change from line to line. Floors and ceilings are suppressed
in entropy and dimension estimates whenever they affect the bounds only
by a fixed polynomial factor.

\section{Proof overview}

The argument first constructs a short independent slice extension
$A_m$ of $A_n$, then an auxiliary matching operator
$\widehat M_{\mathfrak c}$ on a fixed conditional certificate
$\mathfrak c$. Here $v$ denotes a candidate kernel witness and $T$
a fixed coordinate support. The sequence of reductions is
\begin{align*}
 A_n&\longrightarrow A_m
 \longrightarrow\text{exact conditional product slices}
 \longrightarrow\widehat M_{\mathfrak c},\\
 \widehat M_{\mathfrak c}&\longrightarrow\text{inversion of randomness}
 \longrightarrow\text{stopping chain},\\
 &\longrightarrow\text{classification of }v_T
 \longrightarrow\mathbb P(A_n\text{ is singular}).
\end{align*}
The extension and the auxiliary operator are constructed in
\eqref{eq:extension} and \eqref{eq:hybrid auxiliary operator}, respectively.
We now describe the purpose of each passage.

A short independent slice extension converts a nonzero kernel of $A_n$
into a one-dimensional kernel of some $A_m$, with only a polynomial loss.
Permutation symmetry then makes this kernel transverse to the fixed
hyperplane $H_m$. The compression $C_{A_m}$ is consequently invertible,
and a Schur complement expresses the terminal singularity event as a
quadratic equation in one balanced disagreement slice. Two applications
of Cauchy--Schwarz reduce that equation to a linear concentration
problem. For a concentration cutoff $\beta\in(0,1)$,
$q_m(\beta)$ denotes the witness probability in \eqref{eq:qm}.
Theorem~\ref{thm:reduction} records the output as a bound in terms of
this quantity, retaining a deterministic target and a supremum over
every sufficiently large fixed support.

The extended matrix is not a regular adjacency matrix. Section~\ref{Sec:quasirandomness}
therefore preserves its law and changes the exposure. After the pair
profile and the extension block are fixed, the residual original graph
has degree defect at most two. Local enumeration supplies quasirandom
switching supports, and a reconstruction bijection gives the exact
conditional product law in Theorem~\ref{thm:exact hybrid factorization}.
Only the good event measurable in this one exposure is used for
conditioning. This replaces the independence unavailable in the original
matrix by independent slices on an explicitly identified fibre.

The second change of model is analytic. In Section~\ref{Sec:fourier},
the absolute characteristic function of each true slice is dominated
by the nonnegative characteristic function of a lazy matching walk.
Their product defines the auxiliary operator $\widehat M_{\mathfrak c}$.
Gaussian smoothing turns this comparison into concentration uniformly
in the centre, with exponent exactly $\operatorname{rank}Q$ for an
output projection $Q$. The same identity removes the random translation
created by extension walks before the coefficient average in
Section~\ref{Sec:inversion}. Both operations preserve the exponentially
small probability scale.
The two-row exposure fixes the compression, whereas the hybrid exposure
leaves the product factors fresh. A witness fixed under the first need
not be fixed under the second; this is why deterministic coefficient nets
must precede the final conditional estimates.

Inversion of randomness is performed for this auxiliary operator.
Two independent residual copies test the same coefficient vector;
centered support energy supplies local soft singular directions, and
successive conditioning along the alternating path tensorizes their
small-ball savings. A fixed-centre second moment, followed by the
independent extension factors, bounds the fraction of coefficients in
each product box with excessive auxiliary concentration. Section~\ref{Sec:threshold control}
converts these fractions into lattice counts and image nets.

A rounded point need not have the same auxiliary threshold as the point
it approximates. The stopping chain handles this obstruction on the
same fibre: either the rounded point has a usable threshold and the
chain stops, or it is reclassified at a strictly larger scale. A
terminal point always lies in a predetermined net, so one union over
terminal nets gives uniform threshold control without conditioning on
the path selected by the true matrix.

Finally, Section~\ref{Sec:main proof} excludes compressible solutions
and places incompressible solutions in a fixed finite family of signed
geometric domains. On each large fixed support, the restriction is
constant, has an almost-constant centered direction, or belongs to a
CLCD level. The entropy saved on that support survives the two scalar
meshes and the output codimension. Multiplying the covers by Fourier
capture closes $q_m(\beta)$ and then Theorem~\ref{thm:reduction}.
The model transformations, Gaussian bridge, local-to-global inversion,
and stopping argument are kept in the main text; the appendices contain
only the routine auxiliary proofs indicated there.

\section{Reduction to \texorpdfstring{$q_m(\beta)$}{q-m(beta)}}
The reduction follows the two-row argument of \cite{CJMS-jams}.
The extension columns below are independent dense slices; their role is to
reduce the corank while retaining a witness from the original kernel.

For an even set $T\subseteq W_0:=\{3,\ldots,n\}$, set
\begin{align*}
 \Omega_T^0:=\{\eta\in\{-1,1\}^T:\langle\eta,\mathbf1_T\rangle=0\},
 \qquad
 \rho_T(v):=\sup_{a\in\mathbb R}
       \mathbb P_{\eta\sim\operatorname{Unif}(\Omega_T^0)}
                  (\langle\eta,v_T\rangle=a).
\end{align*}
Let $c_{\ref{lem:Sigma fields}}^{(2)}(\lambda)>0$ be the support constant
in Lemma~\ref{lem:Sigma fields}, and define
\begin{align}
 \rho_*(v):=\sup_{\substack{T\subseteq W_0,\ |T|\ge
           c_{\ref{lem:Sigma fields}}^{(2)}n\\ |T|\ \mathrm{even}}}\rho_T(v).\label{eq:rho-star}
\end{align}
The supremum is over all such fixed supports and does not depend on a partition.
Fix $\lambda_*:=\lambda/4$. For $n\le s\le2n-3$, let
$X_s\in\{0,1\}^s$ be independent uniform vectors of deterministic
weights $r_s\in[\lambda_*s,(1-\lambda_*)s]$, independent also of $A_n$.
Set
\begin{align}
 A_{s+1}:=\begin{pmatrix}A_s&X_s\\X_s^T&0\end{pmatrix}.\label{eq:extension}
\end{align}
Recall $h=(\mathbf e_1-\mathbf e_2)/\sqrt2$, $H_m=h^\perp$ and
$C_{A_m}=P_hA_m|_{H_m}$. For $\beta\in(0,1)$, put
\begin{align}
 q_m(\beta):=\sup_{b\in H_m\cap\mathbf1^\perp}
 \mathbb P\{\exists v\in H_m\setminus\{0\}:
                P_hA_mv=b,\ \rho_*(v)>\beta\}.\label{eq:qm}
\end{align}

\begin{theorem}\label{thm:reduction}
Fix $0 < \kappa < 1/4$ and let $k_n=\lfloor\kappa n\rfloor$.
There are positive constants
$c_{\ref{thm:reduction}}^{(i)}=c_{\ref{thm:reduction}}^{(i)}(\lambda,\kappa)$ and
$C_{\ref{thm:reduction}}^{(i)}=C_{\ref{thm:reduction}}^{(i)}(\lambda,\kappa)$
for $i=1,2$, and
$\beta_{\ref{thm:reduction}}=\beta_{\ref{thm:reduction}}(\lambda,\kappa)$,
such that, whenever
$e^{-{c_{\ref{thm:reduction}}^{(1)}}n}\le\beta\le
\beta_{\ref{thm:reduction}}$,
\begin{align}
 \mathbb P(A_n\text{ is singular})
 \le e^{-{c_{\ref{thm:reduction}}^{(1)}}n}+ n^{{C_{\ref{thm:reduction}}^{(1)}}}\sum_{m=n}^{n+k_n-1}
 \left(\beta^{-{C_{\ref{thm:reduction}}^{(2)}}}q_m(\beta)
 +\beta^{{c_{\ref{thm:reduction}}^{(2)}}}\right).\label{eq:reduction}
\end{align}
\end{theorem}

We prove it by passing to a simple kernel with invertible compression,
and then decoupling its Schur complement.

\subsection{Short extensions}
\begin{proposition}\label{prop:rank increment}
    Fix $0 < \kappa < 1/4$ and let $k_{n}=\lfloor \kappa n\rfloor$. There exist positive constants $n_{\ref{prop:rank increment}}=n_{\ref{prop:rank increment}}(\lambda,\kappa)$, $c_{\ref{prop:rank increment}}=c_{\ref{prop:rank increment}}(\lambda,\kappa)$ and $C_{\ref{prop:rank increment}}=C_{\ref{prop:rank increment}}(\lambda,\kappa)$ depending only on $\kappa$ and $\lambda$ such that for $n \ge n_{\ref{prop:rank increment}}$
    \begin{align*}
        \mathbb P\left(  A_{n} \text{ is singular} \right) \le n^{C_{\ref{prop:rank increment}}}\sum_{m=n}^{n+k_{n}-1}\mathbb P\left( \operatorname{corank}(A_{m})=1 , \  C_{A_{m}} \text{ is invertible} \right) + e^{-c_{\ref{prop:rank increment}}n}.
    \end{align*}
\end{proposition}

\begin{proof}
If $B=B^T$ and $x\notin\operatorname{Im}B$, then
\begin{align*}
 \ker\begin{pmatrix}B&x\\x^T&0\end{pmatrix}
   =\{(z,0):z\in\ker B,\ \langle x,z\rangle=0\}.
\end{align*}
Indeed, $Bz+tx=0$ forces $t=0$; the remaining constraint has
codimension one in $\ker B$. Thus a successful extension decreases
the corank by one.

For a dense slice vector $X\in\{0,1\}^s$ and a proper subspace $V$ of
codimension $j$, the elementary slice bounds are
\begin{align*}
 \mathbb P(X\in V)\le\min\{1-c,\ C\sqrt s(1-\lambda_*)^j\}.
\end{align*}
For the first bound, take $a\ne0$ with $V\subset a^\perp$.
If $a$ is constant the slice misses $a^\perp$; otherwise choose
$a_i\ne a_j$. The event $X_i\ne X_j$ has probability
$2r(s-r)/(s(s-1))\ge2\lambda_*^2$, and at most one of its two
equiprobable orientations lies in $a^\perp$.
For the second bound, condition independent Bernoulli$(r/s)$
coordinates on their sum $r$. The conditioning costs at most $C\sqrt s$
by Stirling's formula, while $j$ pivot coordinates for $V$ have
conditional mass at most $(1-\lambda_*)^j$.

Write $\mathcal H_s=\sigma(A_n,X_n,\ldots,X_{s-1})$.
On $\{\operatorname{corank}A_s=j\}$ these estimates apply to
$V=\operatorname{Im}A_s$ conditional on $\mathcal H_s$.
Choose $J=C_0\log n$ so that
$\sum_{j\ge J}C\sqrt{2n}(1-\lambda_*)^j\le1/4$.
Along successive successful extensions, the steps with corank at
least $J$ all succeed with probability at least $3/4$, and the
remaining at most $J$ steps cost at most a factor $c^J$.
Consequently, for $m=n+k-1$,
\begin{align*}
 &\mathbf1_{\{\operatorname{corank}A_n=k\}}
 \mathbb P\bigl(\operatorname{corank}A_m=1,
       \ \ker A_m\subseteq\ker A_n\times\{0\}^{m-n}\mid A_n\bigr)
 \\[-2pt]
 &\hspace{3em}\ge n^{-C}\mathbf1_{\{\operatorname{corank}A_n=k\}}.
\end{align*}

We also need a large-corank bound. Let $B$ be the adjacency matrix
of $G(n,p)$, where $p=d/(n-1)$, and put $\zeta=\max(p,1-p)$.
A symmetric matrix of rank $r$ has an invertible principal
$r\times r$ minor. Once this minor and its incident block are fixed,
the Schur complement prescribes the remaining block. Its independent
upper-triangular entries therefore give
\begin{align*}
 \mathbb P(\operatorname{corank}B\ge k)
 \le\sum_{r\le n-k}\binom nr\zeta^{\binom{n-r}{2}}
 \le2^n\zeta^{\binom k2}.
\end{align*}
Dense enumeration \cite[Theorem~3]{McKay-cpc} and Stirling's formula
bound the probability of all degrees being $d$ below by
$(c/\sqrt n)^n$. Hence
\begin{align*}
 \mathbb P(\operatorname{corank}A_n\ge k)
 \le\exp\bigl(\tfrac12n\log n+Cn-ck^2\bigr).
\end{align*}
In particular, corank exceeding $k_n$ has exponentially small probability.

On the successful-extension event above, let
$\ker A_m=\mathbb Rz$. Since $A_n\mathbf1=d\mathbf1$, the nonzero
restriction $z_{[n]}$ is orthogonal to $\mathbf1$.
At least $n-1$ pairs $i,j\in[n]$ therefore satisfy $z_i\ne z_j$.
For each such pair the compression to
$h_{ij}^\perp$, where $h_{ij}=(\mathbf e_i-\mathbf e_j)/\sqrt2$,
is invertible: if $x\perp h_{ij}$ and $A_mx\in\mathbb Rh_{ij}$,
its inner product with $z$ forces $A_mx=0$, and then $x=0$.
Averaging over pairs and using permutation symmetry of the original
$n$ coordinates gives
\begin{align*}
 \mathbb P(\operatorname{corank}A_n=k)
 \le n^{C+1}\mathbb P(\operatorname{corank}A_{n+k-1}=1,
                         \ C_{A_{n+k-1}}\text{ invertible}).
\end{align*}
Sum over $1\le k\le k_n$ and add the large-corank bound.
\end{proof}

\subsection{The two-row equation}
Put $\mathcal T=\{i\in W_0:A_{1i}\ne A_{2i}\}$ and expose
\begin{align*}
 \mathcal F_m:=\sigma\bigl(A_{12},(A_{ij})_{3\le i<j\le n},
          (A_{1i}+A_{2i})_{i\in W_0},
          (A_{ij})_{1\le i<j\le m,\ j>n}\bigr).
\end{align*}
\begin{lemma}\label{lem:Sigma fields}
    The sigma-field $\mathcal F_m$ determines $\mathcal T$ and $C_{A_m}$. Conditional on $\mathcal F_m$, the vector $\eta=(A_{1i}-A_{2i})_{i\in\mathcal T}$
    is uniform on the balanced sign slice $\Omega_{\mathcal T}^0$.
    Moreover, there exists
    ${c_{\ref{lem:Sigma fields}}^{(1)}=c_{\ref{lem:Sigma fields}}^{(1)}(\lambda)}>0$ depending only on $\lambda$ such that
    \begin{align*}
        \mathbb P\big(
        |\mathcal T|<{c_{\ref{lem:Sigma fields}}^{(2)}}n
        \big)
        \le
        e^{-{c_{\ref{lem:Sigma fields}}^{(1)}}n}
    \end{align*}
    for some
    ${c_{\ref{lem:Sigma fields}}^{(2)}}=c_{\ref{lem:Sigma fields}}^{(2)}(\lambda)>0$ depending only on $\lambda$.
\end{lemma}

\begin{proof}
The pair sums fix both entries off $\mathcal T$. On $\mathcal T$
the two choices are $(1,0)$ and $(0,1)$, and equality of the original
degrees forces their counts to agree. Every balanced choice preserves
all degrees and exposed entries, so uniformity gives the stated slice
law. The corresponding matrix variation is
$2^{-1/2}(h\eta^T+\eta h^T)$, which vanishes upon compression to $H_m$.
For the support bound, dense degree-sequence enumeration
\cite[Theorem~3]{McKay-cpc} compares the two incidence rows with
independent Bernoulli$(p)$ rows at cost $Cn$, where $p=d/(n-1)$;
the finite-row argument is recalled in Section~\ref{Sec:quasirandomness}.
Their disagreement count has mean $2p(1-p)(n-2)\ge c_\lambda n$,
so a Chernoff bound proves the claim.
\end{proof}

For $S\subseteq W_0$, define
$L_Sa=2^{-1/2}\sum_{i\in S}a_i\mathbf e_i$.
In $\mathbb R^m=H_m\oplus\mathbb Rh$, the exposed matrix is
\begin{align*}
 A_m=\begin{pmatrix}
 C_{A_m}&y_0+L_{\mathcal T}\eta\\
 (y_0+L_{\mathcal T}\eta)^T&\alpha
 \end{pmatrix},
\end{align*}
where $y_0,\alpha$ are $\mathcal F_m$-measurable.
When $C_{A_m}$ is invertible, the terminal event is exactly
\begin{align*}
 F(\eta):=\alpha-
 \langle y_0+L_{\mathcal T}\eta,
        C_{A_m}^{-1}(y_0+L_{\mathcal T}\eta)\rangle=0.
\end{align*}

\begin{proposition}\label{prop:two row terminal}
    There exist positive constants
    $c_{\ref{prop:two row terminal}}=c_{\ref{prop:two row terminal}}(\lambda)$ and
    $C_{\ref{prop:two row terminal}}=C_{\ref{prop:two row terminal}}(\lambda)$ depending only on $\lambda$ such that for any $m \in [n,n+k_{n}-1]$, $1 \le t \le c_{\ref{prop:two row terminal}}n$ and $0<\beta<1$, we have
    \begin{align*}
        &\mathbb P\big(
        \operatorname{corank}A_m=1,\,
        C_{A_m} \text{ is invertible}
        \big)\\
        &\le
        e^{-c_{\ref{prop:two row terminal}}n}
        +3^t q_m(\beta)
        +C_{\ref{prop:two row terminal}}
        \big(
        C_{\ref{prop:two row terminal}}^t\beta
        +e^{-c_{\ref{prop:two row terminal}}t}
        \big)^{1/4}
        +C_{\ref{prop:two row terminal}}
        e^{-c_{\ref{prop:two row terminal}}t}.
    \end{align*}
\end{proposition}

\begin{proof}
Let $\mathcal E_m=\{\operatorname{corank}A_m=1,
C_{A_m}\text{ invertible}\}$.
Independently choose a uniform $t$-subset $I\subset W_0$, and put
$S=I\cap\mathcal T$, $J=\mathcal T\setminus S$ and
$K=|\{i\in S:\eta_i=1\}|$.
Set $\mathcal H=\sigma(\mathcal F_m,I,K)$.
Conditional on $\mathcal H$, the restrictions $\eta_S,\eta_J$ are
independent uniform sign slices of their prescribed weights.
Hypergeometric concentration \cite[Section~6]{Hoeffding-jasa} gives,
outside probability $e^{-cn}+Ce^{-ct}$,
\begin{align*}
 |\mathcal T|\ge c_{\ref{lem:Sigma fields}}^{(2)}n,
 \qquad ct\le|S|\le|\mathcal T|/4,
 \qquad |S|/3\le K\le2|S|/3.
\end{align*}
Here and below the restriction $t\le cn$ uses a sufficiently small
constant depending only on $\lambda$.

For a fixed $I$ the deterministic set
\begin{align*}
 \mathcal W_I:=\{L_I\delta:\delta\in\{-2,0,2\}^I,
                                  \textstyle\sum_i\delta_i=0\}
 \subset H_m\cap\mathbf1^\perp
\end{align*}
has at most $3^t$ points. By \eqref{eq:qm},
\begin{align*}
 \mathbb P\{C_{A_m}\text{ invertible},\ 
       \exists w\in\mathcal W_I\setminus\{0\}:
             \rho_*(C_{A_m}^{-1}w)>\beta\mid I\}
 \le3^tq_m(\beta).
\end{align*}
Let $\mathcal G\in\mathcal H$ be the event that the compression is
invertible, the displayed support and balance bounds hold, and this last exceptional
event does not occur. Thus
\begin{align*}
 \mathbb P(\mathcal E_m)
 \le e^{-cn}+Ce^{-ct}+3^tq_m(\beta)
       +\mathbb E\bigl[\mathbf1_{\mathcal G}
                           \mathbb P(F(\eta)=0\mid\mathcal H)\bigr].
\end{align*}

Two applications of Cauchy--Schwarz, using independent conditional
copies $\eta'_S,\eta'_J$, cancel the single-block terms in $F$ and give
\begin{align*}
 &\mathbf1_{\mathcal G}\mathbb P(F(\eta)=0\mid\mathcal H)^4\\
 &\quad\le\mathbf1_{\mathcal G}\mathbb P\bigl(
 \langle L_S(\eta_S-\eta'_S),
       C_{A_m}^{-1}L_J(\eta_J-\eta'_J)\rangle=0\mid\mathcal H\bigr).
\end{align*}
The diagonal case $\eta_S=\eta'_S$ costs
$\binom{|S|}{K}^{-1}\le e^{-ct}$ on $\mathcal G$.
Otherwise set $w=L_S(\eta_S-\eta'_S)$ and $v=C_{A_m}^{-1}w$.
Then $w\in\mathcal W_I\setminus\{0\}$ and $\rho_*(v)\le\beta$.
After also fixing $\eta_S,\eta'_S,\eta'_J$, the remaining equation
is $\langle\eta_J,v_J\rangle=a$ for a fixed $a$.

For completeness, if $T$ is balanced and $s=|S|\le|T|/4$,
conditioning a prescribed sign pattern on $S$ gives
\begin{align*}
 \sup_a\mathbb P(\langle\eta_J,v_J\rangle=a\mid I,K)
 \le8^{|S|}\rho_T(v).
\end{align*}
Indeed, a prescribed admissible pattern on $S$ has conditional
mass at least $2^{-s}$, and its count $K$ has hypergeometric mass
at least $4^{-s}$. Combining it with a maximal atom on $J$
produces an atom of the original balanced slice. This argument
is pointwise in $v$, so it applies after the additional exposure.
Writing $\mathcal D=\sigma(\mathcal H,\eta_S,\eta'_S,\eta'_J)$, we obtain
\begin{align*}
 \mathbf1_{\mathcal G}\mathbf1_{\{\eta_S\ne\eta'_S\}}
   \mathbb P(\langle\eta_J,v_J\rangle=a\mid\mathcal D) \le\mathbf1_{\mathcal G}\mathbf1_{\{\eta_S\ne\eta'_S\}}
             8^t\rho_{\mathcal T}(v)
 \le\mathbf1_{\mathcal G}\mathbf1_{\{\eta_S\ne\eta'_S\}}8^t\beta.
\end{align*}
Average over the copies, take fourth roots and substitute into the event decomposition above. This proves the proposition.
\end{proof}

\begin{proof}[Proof of Theorem~\ref{thm:reduction}]
In Proposition~\ref{prop:two row terminal}, take
$t=\lfloor a\log(1/\beta)\rfloor$ with fixed $a>0$ small enough that
$C^t\beta\le\beta^{1/2}$.
After reducing the upper cutoff for $\beta$ and the lower exponential
range constant, $1\le t\le cn$, while
$3^t\le\beta^{-C_1}$ and $e^{-ct}\le C\beta^{c_1}$.
Thus the terminal probability is at most
$e^{-cn}+\beta^{-C_1}q_m(\beta)+C\beta^{c_1}$.
Proposition~\ref{prop:rank increment} now gives \eqref{eq:reduction}.
\end{proof}

\section{Conditional product structure}\label{Sec:quasirandomness}
The fixed degrees prevent independent row exposure. We instead expose
pair sums and partial row sums, leaving independent orientations on
explicit slices. Quasirandomness will ensure that these slices have
large supports and dense weights. The conditional law itself is exact.

\subsection{Alternating exposure}
Fix a vertex set $V=X\sqcup Y$, a balanced splitter $Y=Y_0\sqcup Y_1$,
and a prescribed degree sequence. Write $N=|V|$.
Write $C:=A_{X\times X}$, $D:=A_{Y\times Y}$, and
$H:=A_{X\times Y}$. Conditional on $C,D$, the cross-block is uniform on
\begin{align}
 \mathcal M(\mathbf r,\mathbf c)
 :=\{M\in\{0,1\}^{X\times Y}:M\mathbf 1=\mathbf r,
                  M^T\mathbf 1=\mathbf c\},\label{eq:rectangular margin fibre}
\end{align}
where $r_x=d_x-\sum_{x'\in X}C_{xx'}$ and
$c_y=d_y-\sum_{y'\in Y}D_{yy'}$.
Thus $H$ has the rectangular fixed-margin law underlying the
switching construction of \cite[Section~3.1]{JSS-imrn}; symmetry places
$H^T$ in the opposite block and imposes no further constraint on $H$.
This conditional law holds for every feasible degree sequence.

Let $\pi=(x_1,\ldots,x_\ell)$ order $X$, where $\ell:=|X|$.
The ordering is either deterministic or sampled independently of the
graph; in either case it is included in the exposure.
Set $R_i:=H_{x_i,\cdot}$ and
$r_i^{(b)}:=\sum_{y\in Y_b}H_{x_i y}$.
For $1\le j<\ell$, let $e_j:=(x_j,x_{j+1})$ and set
$b(j)=0$ for odd $j$, and $b(j)=1$ for even $j$.
Write $b(e_j):=b(j)$, $E:=\{e_j:1\le j<\ell\}$ and
$E_b:=\{e_j:b(j)=b\}$. Finally, let $P_\ell=Y_0$ if $\ell$ is
odd and $P_\ell=Y_1$ if $\ell$ is even. Define
\begin{align}
 \mathcal G_{X,Y}:=\sigma\bigl(&\pi,C,D,(r_i^{(b)})_{i,b},
 ((R_j+R_{j+1})_{Y_{b(j)}})_{j<\ell},(R_1)_{Y_1},(R_\ell)_{P_\ell}\bigr).\label{eq:alternating sigma field}
\end{align}
The last two restrictions are exactly the endpoint pieces not covered
by a pair. The notation suppresses the fixed splitter and ordering rule.

For $e_j=(x_j,x_{j+1})$, let
\begin{align*}
 T_{e_j}:=\{y\in Y_{b(j)}:H_{x_jy}+H_{x_{j+1}y}=1\}, a_{e_j}
 :=r_j^{(b(j))}
      -|\{y\in Y_{b(j)}:H_{x_jy}+H_{x_{j+1}y}=2\}|.
\end{align*}
Both are $\mathcal G_{X,Y}$-measurable. Write also
$T_j:=T_{e_j}$ and $a_j:=a_{e_j}$, and define
$\mathcal O_{e_j}:=\{y\in T_{e_j}:H_{x_jy}=1\}$.
Thus $|\mathcal O_e|=a_e$. We identify $\mathcal O_e$ with its indicator when it is
used as a vector.

The fresh rectangles follow an alternating path:
\begin{align*}
\begin{array}{c|cc}
 &Y_0&Y_1\\\hline
 x_1&\mathcal O_{e_1}&\text{exposed}\\
 x_2&T_{e_1}\setminus\mathcal O_{e_1}&\mathcal O_{e_2}\\
 x_3&\mathcal O_{e_3}&T_{e_2}\setminus\mathcal O_{e_2}\\
 x_4&T_{e_3}\setminus\mathcal O_{e_3}&\mathcal O_{e_4}\\
 x_5&\mathcal O_{e_5}&T_{e_4}\setminus\mathcal O_{e_4}\\
 \vdots&\vdots&\vdots
\end{array}
\end{align*}
Each displayed set represents the ones on its disagreement support;
entries with pair sum zero or two are already fixed. Odd pairs use $Y_0$,
even pairs use $Y_1$, and the final uncovered row piece is exposed.
Successive pairs share a row but never a matrix entry.

\begin{proposition}\label{prop:exact switching fibre}
Almost surely, on every nonempty exposure fibre,
\begin{align}
 \operatorname{Law}\left((\mathcal O_e)_{e\in E}
                      \mid\mathcal G_{X,Y}\right)
 =\bigotimes_{e\in E}
       \operatorname{Unif}\binom{T_e}{a_e}.\label{eq:exact switching fibre}
\end{align}
This identity is exact for every nonempty fibre
\eqref{eq:rectangular margin fibre}, regardless of its margins.
\end{proposition}

\begin{proof}
Fix a nonempty atom. On the rectangle for $e=(x,x')$, a pair sum
zero or two fixes both entries, while a pair sum one leaves the
choice $(1,0)$ or $(0,1)$. The exposed partial sum in row $x$
requires exactly $a_e$ first orientations. The other row sum follows
from the pair sums. Hence the choices are precisely $\binom{T_e}{a_e}$.

Within each half the row pairs are disjoint; opposite parities use
different halves. Each choice preserves the two partial row sums and
contributes the same exposed pair sum to every column. Together with
the exposed endpoint pieces, arbitrary choices in
$\prod_e\binom{T_e}{a_e}$ reconstruct exactly one compatible cross-block,
and every compatible block arises this way. Uniformity on the margin
fibre therefore assigns probability
$\prod_e\binom{|T_e|}{a_e}^{-1}$ to each configuration, proving the product law.
\end{proof}

For $b_e=\mathbf e_x-\mathbf e_{x'}$, set
$\xi_e=\mathbf1_{\mathcal O_e}-(a_e/|T_e|)\mathbf1_{T_e}$ when
$T_e\ne\varnothing$, and $\xi_e=0$ otherwise. Extend these vectors
by zero. The reconstruction reads
\begin{align}
 H=\bar H+\sum_{e\in E}(b_e|_X)(\xi_e|_Y)^T,
 \qquad
 H-\bar H=
 \begin{array}{c|cc}
 &Y_0&Y_1\\\hline
 x_1&\xi_{e_1}^T&0\\
 x_2&-\xi_{e_1}^T&\xi_{e_2}^T\\
 x_3&\xi_{e_3}^T&-\xi_{e_2}^T\\
 x_4&-\xi_{e_3}^T&\xi_{e_4}^T\\
 \vdots&\vdots&\vdots
 \end{array}\label{eq:conditional cross block diagram}
\end{align}
where $\bar H=\mathbb E(H\mid\mathcal G_{X,Y})$.
All entries on the diagonal blocks
$C,D$ are fixed, so the full residual fluctuation is obtained by
adding the transpose of this cross-block.

\subsection{Uniform support bounds}
Fix a finite collection of splits and balanced splitters before
sampling the graph, with $|X|,|Y|\ge\delta N$ and the number of choices
bounded by a fixed integer $K_0$. No vector-dependent selection is made
in the conditional law. The family used for incompressible vectors
will be chosen in Section~\ref{subsec:final signed geometry}.

For a degree sequence $\mathbf d=(d_i)_{i\in V}$, let
$\mathfrak G_V(\mathbf d)$ be the simple graphs on $V$ with those degrees,
$g_V(\mathbf d)=|\mathfrak G_V(\mathbf d)|$, and
$\bar d=N^{-1}\sum_i d_i$. We assume
\begin{align}
 \lambda(N-1)\le\bar d\le(1-\lambda)(N-1),
 \qquad \max_i|d_i-\bar d|\le B,\label{eq:bounded defect hypothesis}
\end{align}
where $B$ is fixed. For $V=[N]$ write $g$ for $g_V$.
We follow the finite-row comparison in the proof of
\cite[Theorem~2.17]{JSS-imrn}. For the present graph law, use the
enumeration formula \cite[Theorem~3]{McKay-cpc}: after fixing the
induced graph on $r=O(1)$ tested vertices, the remaining degree
sequences have a common sum and bounded defect, so their completion
counts differ by at most $e^{C(\lambda,B,r)}$. Comparison with
independent Bernoulli$(\bar d/(N-1))$ incidence rows then costs
$CN^{r/2}$, the polynomial factor coming from their prescribed
row sums. This transfers exponential concentration uniformly under
\eqref{eq:bounded defect hypothesis}. The comparison is made before
fixing $C,D$; the rectangular conditional law
\eqref{eq:rectangular margin fibre} is used for the subsequent switching.

For $e=(x,x')$ and $S\subseteq Y$, put
$S_e^Y=\{y\in Y:H_{xy}\ne H_{x'y}\}$ and
$\Delta_e(S)=\sum_{y\in S}(H_{xy}-H_{x'y})$.

\begin{proposition}\label{thm:fixed block quasirandomness}
Fix the preceding block data, a defect bound $B$ and
$\varepsilon\in(0,1/4)$. There are constants
$c_{\ref{thm:fixed block quasirandomness}}^{(i)}
=c_{\ref{thm:fixed block quasirandomness}}^{(i)}
(\lambda,B,\delta,\varepsilon,K_0)>0$, $1\le i\le4$, and an integer
$\ell_{\ref{thm:fixed block quasirandomness}}
=\ell_{\ref{thm:fixed block quasirandomness}}(\lambda,\varepsilon)$
such that, outside probability
$e^{-c_{\ref{thm:fixed block quasirandomness}}^{(1)}N}$,
the following hold simultaneously under \eqref{eq:bounded defect hypothesis}.
For every prescribed splitter half $S$, every row pair $e$, and
every two vertex-disjoint row pairs $e,e'$,
\begin{align*}
 |S_e^Y\cap S|&\ge c_{\ref{thm:fixed block quasirandomness}}^{(2)}N,
 &|S_e^Y\cap S_{e'}^Y\cap S|
     &\ge c_{\ref{thm:fixed block quasirandomness}}^{(3)}N,\\
 |\Delta_e(S)|&\le |S_e^Y\cap S|/10.
\end{align*}
Every cross-degree lies between
$c_{\ref{thm:fixed block quasirandomness}}^{(4)}$ and
$1-c_{\ref{thm:fixed block quasirandomness}}^{(4)}$
times the opposite block size. Moreover, for every collection
of $\ell=\ell_{\ref{thm:fixed block quasirandomness}}$
vertex-disjoint row pairs and each half $Y_b$,
\begin{align*}
 \left|Y_b\setminus\bigcup_{j=1}^{\ell}S_{e_j}^Y\right|
 \le\varepsilon|Y_b|.
\end{align*}
The same assertions hold with the blocks transposed.
One may take $c_{\ref{thm:fixed block quasirandomness}}^{(4)}
=\lambda/2$.
\end{proposition}

\begin{proof}
In the independent incidence model, disagreements have mean
$q=2p(1-p)$, intersections of disjoint pairs have mean $q^2$,
and row differences are centered. Chernoff--Hoeffding bounds
\cite{Hoeffding-jasa}, the preceding comparison with at most
four rows, and a polynomial union bound give the support,
intersection, balance and margin assertions. Choose the fixed
$\ell$ so that $(1-2\lambda(1-\lambda))^\ell\le\varepsilon/4$.
A column is then missed by $\ell$ disjoint pairs with probability
at most $\varepsilon/4$; concentration, comparison of $2\ell$
rows and a union over at most $N^{2\ell}$ collections prove coverage.
\end{proof}

\begin{lemma}\label{lem:pair coverage and energy}
Fix $\delta_0>0$. Suppose a deterministic cross-block has
$|X|,|Y|\ge\delta N$ and every column degree lies in
$[\delta_0|X|,(1-\delta_0)|X|]$.
Choose $\pi$ uniformly and independently of the graph.
There are positive constants
${c_{\ref{lem:pair coverage and energy}}^{(1)}=c_{\ref{lem:pair coverage and energy}}^{(1)}(\delta_0)}$ and
${c_{\ref{lem:pair coverage and energy}}^{(2)}=c_{\ref{lem:pair coverage and energy}}^{(2)}(\delta_0,\delta)}$, depending only on
$\delta_0,\delta$, with the first constant depending only on $\delta_0$,
such that, outside probability
$e^{-{c_{\ref{lem:pair coverage and energy}}^{(2)}}N}$ over $\pi$,
\begin{align}
 |\{e\in E_b:y,y'\in T_e\}|
 \ge {c_{\ref{lem:pair coverage and energy}}^{(1)}}|X|\label{eq:coordinate pair coverage}
\end{align}
for both $b=0,1$ and every two distinct $y,y'\in Y_b$.
Whenever \eqref{eq:coordinate pair coverage} holds, simultaneously
for every $z\in\mathbb R^{Y_b}$,
\begin{align}
 \sum_{e\in E_b}\|P_{T_e,0}z\|_2^2
 \ge {c_{\ref{lem:pair coverage and energy}}^{(1)}}|X|
       \|P_{Y_b,0}z\|_2^2.\label{eq:centered energy coverage}
\end{align}
The sum is indexed by the row pairs, so repeated supports retain their
multiplicities; terms with $|T_e|\le1$ are interpreted as zero.
\end{lemma}

\begin{proof}
For fixed columns $y,y'$, let $\alpha_{uv}$ be the proportions of row
types $(u,v)\in\{0,1\}^2$. Dense column margins imply
$\alpha_{00}\alpha_{11}+\alpha_{01}\alpha_{10}\ge c(\delta_0)>0$:
a zero would force one column to be constant. Thus a uniform row pair
disagrees in both columns with probability at least $c(\delta_0)$.
For each parity, the expected count in \eqref{eq:coordinate pair coverage}
is at least $c|X|$. A transposition of the ordering changes it by at
most two. Permutation bounded differences and a union over column pairs
prove the coverage bound: the exposure martingale has bounded
increments, so deviations of order $N$ have probability $e^{-cN}$.
For every realization with this property,
\begin{align*}
 \sum_{e\in E_b}\|P_{T_e,0}z\|_2^2
 &=\sum_{e\in E_b}\frac1{|T_e|}
               \sum_{\{i,j\}\subseteq T_e}(z_i-z_j)^2\\
 &\ge\frac{c|X|}{|Y_b|}\sum_{\{i,j\}\subseteq Y_b}(z_i-z_j)^2
 =c|X|\|P_{Y_b,0}z\|_2^2.
\end{align*}
This deterministic inequality holds for all $z$ simultaneously.
\end{proof}

Apply this lemma below with
$\delta_0={c_{\ref{thm:fixed block quasirandomness}}^{(4)}}$.
Fix also $\varepsilon\in(0,1/4)$ and
$\ell_{\ref{thm:fixed block quasirandomness}}$ as in
Proposition~\ref{thm:fixed block quasirandomness}.

\begin{definition}\label{def:good switching certificate}
For a fixed split and splitter, let $\mathcal Q_{X,Y}$ be the event
that the following conditions hold:
\begin{enumerate}[label=\textnormal{(G\arabic*)},leftmargin=2.6em]
\item Every cross-degree lies in
$[c_{\ref{thm:fixed block quasirandomness}}^{(4)},
1-c_{\ref{thm:fixed block quasirandomness}}^{(4)}]$
times the opposite block size.
\item For every fresh pair $e$,
\begin{align*}
 |T_e|\ge {c_{\ref{thm:fixed block quasirandomness}}^{(2)}}N,
 \qquad |2a_e-|T_e||\le |T_e|/10.
\end{align*}
\item For each $b$ and distinct $e,e'\in E_b$,
$|T_e\cap T_{e'}|\ge {c_{\ref{thm:fixed block quasirandomness}}^{(3)}}N$.
\item For each $b$ and every $\mathcal J\subseteq E_b$ with
$|\mathcal J|=\ell_{\ref{thm:fixed block quasirandomness}}$,
$|Y_b\setminus\bigcup_{e\in\mathcal J}T_e|\le\varepsilon|Y_b|$.
\item The coordinate-pair coverage \eqref{eq:coordinate pair coverage}
holds for both halves.
\end{enumerate}
An exposure satisfying these conditions is called a good switching
certificate. Each condition is determined by the exposed data, so
$\mathcal Q_{X,Y}\in\mathcal G_{X,Y}$.
\end{definition}

Choose the orderings independently and uniformly before sampling
the graph. Proposition~\ref{thm:fixed block quasirandomness} gives
(G1)--(G4) for every ordering, since
$2a_e-|T_e|=\Delta_e(Y_{b(e)})$ and pairs of one parity are disjoint.
Conditional on this event, Lemma~\ref{lem:pair coverage and energy}
gives (G5). A finite union therefore yields
\begin{align}
 \mathbb P\left(\bigcap_{(X,Y,Y_0,Y_1)}\mathcal Q_{X,Y}\right)
 \ge1-e^{-cN},\label{eq:good certificate probability}
\end{align}
uniformly under \eqref{eq:bounded defect hypothesis}, with
$c=c(\lambda,B,\delta,\varepsilon,K_0)>0$.
The intersection over blocks is used only to bound failure
probabilities, never as additional conditioning for a single block.

We will also use the following deterministic consequence of (G4).
If $P\subseteq Y_b$ has $|P|\ge2\varepsilon|Y_b|$, fewer than
$\ell=\ell_{\ref{thm:fixed block quasirandomness}}$ indices satisfy
$|T_e\cap P|<\varepsilon|Y_b|/\ell$.
Otherwise their union would meet $P$ in fewer than
$\varepsilon|Y_b|$ points, contradicting (G4).
For two such sets $P_+,P_-$, both lower bounds therefore hold
outside at most $2(\ell-1)$ indices, simultaneously for every
choice of the sets.

\subsection{The complete conditional matrix}
For the extension $A_m$, we must retain both the switching orientations
and the original-vertex parts of its added columns. The two-row
sigma-field $\mathcal F_m$ fixes both, so we now use a coarser exposure.

Let $O:=[n]$, $W:=W_0$, and $U_m:=\{n+1,\ldots,m\}$.
Retain $\mathcal T$ and $\eta$ from Lemma~\ref{lem:Sigma fields}.
For the pair profile and the residual degrees, set
\begin{align*}
 \Pi
 :=\left(A_{12},(s_i)_{i\in W}\right),
 \qquad s_i:=A_{1i}+A_{2i},\qquad d_i':=d-s_i.
\end{align*}
Let $G'$ be the graph induced by $A_n$ on $W$, and write
\begin{align*}
 A_m=\begin{pmatrix}A_n&B_m\\ B_m^T&C_m\end{pmatrix}.
\end{align*}
For $j\in U_m$, write $B_m^{(j)}\in\{0,1\}^{O}$ for its column
on the original vertices, and define its $C_m$-measurable weight by
$a_j=r_{j-1}-\sum_{i=n+1}^{j-1}(C_m)_{ij}$.
Here the indices distinguish these column weights from the switching
weights $a_e$. Set $\mathcal H_m^0:=\sigma(\Pi,C_m)$.

\begin{proposition}\label{prop:pair profile factorization}
For every admissible value of $\mathcal H_m^0$,
\begin{align*}
 \operatorname{Law}\left(\eta,G',(B_m^{(j)})_{j\in U_m}
                    \mid\mathcal H_m^0\right)=
 \operatorname{Unif}(\Omega_{\mathcal T}^0)
 \otimes\operatorname{Unif}\left(\mathfrak G_W((d_i')_{i\in W})\right)
 \otimes\bigotimes_{j\in U_m}\operatorname{Unif}\binom{O}{a_j}.
\end{align*}
Moreover $|d_i'-\overline{d'}|\le2$, and the residual mean has
density margin at least $\lambda/2$ on $W$ for sufficiently large $n$.
\end{proposition}

\begin{proof}
Fix the pair profile. At every $i\in\mathcal T$ exactly one edge to
$\{1,2\}$ is present, and equality of the two original degrees requires
a balanced orientation. Independently, the graph on $W$ must have degrees
$d-s_i$. Every such graph and balanced orientation reconstruct exactly
one original graph with that profile. Uniformity proves the first two
independent factors.

The upper-triangular entries of $C_m$ expose, separately for each birth
vector $X_{j-1}$, its coordinates in $\{n+1,\ldots,j-1\}$.
The remaining restriction to $O$ is a uniform $a_j$-slice.
These restrictions remain independent across birth vectors and independent
of the original graph. This proves the factorization.
Finally $d-2\le d_i'\le d$, giving defect two and residual density
margin $\lambda/2$ for sufficiently large $n$.
\end{proof}

Let $\mathcal D_m\in\sigma(C_m)$ be the event that
\begin{align}
 (\lambda_*/2)n\le a_j\le(1-\lambda_*/2)n
 \qquad(j\in U_m).\label{eq:dense extension weights}
\end{align}
Each $a_j$ is hypergeometric with mean
$nr_{j-1}/(j-1)\in[\lambda_*n,(1-\lambda_*)n]$.
By Hoeffding's inequality for sampling without replacement
\cite[Section~6]{Hoeffding-jasa}, each weight violates these bounds
with probability at most $2e^{-cn}$.
The weights are fixed at birth, so a union over the birth columns
gives $\mathbb P(\bigcup_{m=n}^{2n-2}\mathcal D_m^c)\le e^{-c'n}$.

On $W$, fix deterministic block and splitter families as above before
sampling $A_n$ or the extension. Choose their auxiliary orderings
independently of all graph variables. Apply the preceding graph and
certificate results with $N=n-2$, defect bound $2$, and density margin
$\lambda/2$; constants in this application always refer to those
parameter values.
For one fixed block and splitter, expose $G'$ according to
\eqref{eq:alternating sigma field} and define
\begin{align*}
 \mathcal H_{m;X,Y}
 :=\sigma(\Pi,C_m,\mathcal G_{X,Y}),
 \qquad
 \mathcal E_{m;X,Y}:=\mathcal D_m\cap\mathcal Q_{X,Y}.
\end{align*}
Only this block's switching exposure is included in
$\mathcal H_{m;X,Y}$. In particular,
$\mathcal E_{m;X,Y}\in\mathcal H_{m;X,Y}$.

\begin{theorem}\label{thm:exact hybrid factorization}
For every $n\le m\le2n-2$ and every fixed block and splitter,
\begin{align}
 \operatorname{Law}\left(
       \eta,(\mathcal O_e)_{e\in E},(B_m^{(j)})_{j\in U_m}
       \mid\mathcal H_{m;X,Y}\right) =
 \operatorname{Unif}(\Omega_{\mathcal T}^0)
 \otimes\bigotimes_{e\in E}
             \operatorname{Unif}\binom{T_e}{a_e}
 \otimes\bigotimes_{j\in U_m}\operatorname{Unif}\binom{O}{a_j}.\label{eq:exact hybrid factorization}
\end{align}
This identity holds on every nonempty exposure fibre. On
$\mathcal E_{m;X,Y}$, the switching slices have linear-size supports
and weights in $[|T_e|/3,2|T_e|/3]$, the extension slices satisfy
\eqref{eq:dense extension weights}, and the intersection, union,
coordinate-pair and centered energy conclusions of the certificate hold.

There is $c_{\ref{thm:exact hybrid factorization}}=c_{\ref{thm:exact hybrid factorization}}(\lambda,\delta,\varepsilon,K_0)>0$, depending only
on $\lambda,\delta,\varepsilon,K_0$, such that
\begin{align*}
 &\mathbf 1_{\mathcal D_m}
 \mathbb P\left(\bigcup_{(X,Y,Y_0,Y_1)}\mathcal Q_{X,Y}^c
                   \,\middle|\,\mathcal H_m^0\right)
 \le\mathbf 1_{\mathcal D_m}
                 e^{-c_{\ref{thm:exact hybrid factorization}}n}.
\end{align*}
In particular, simultaneously for every $n\le m\le2n-2$ and all
the fixed block choices, $\mathcal E_{m;X,Y}$ holds outside an event
of probability at most $e^{-c_{\ref{thm:exact hybrid factorization}}n}$,
after decreasing the constant if necessary.
\end{theorem}

\begin{proof}
On an atom of $\mathcal H_m^0$, Proposition~\ref{prop:pair profile factorization}
leaves three independent factors. The additional exposure
$\mathcal G_{X,Y}$ depends only on the residual graph and its independent
ordering. The reconstruction identity \eqref{eq:exact switching fibre}
therefore changes only the middle factor, replacing it by its exact product of orientation
slices. In particular every admissible configuration has conditional mass
\begin{align*}
 |\Omega_{\mathcal T}^0|^{-1}
 \prod_{e\in E}\binom{|T_e|}{a_e}^{-1}
 \prod_{j\in U_m}\binom n{a_j}^{-1}.
\end{align*}
This proves the factorization on every nonempty fibre.
On the measurable event $\mathcal E_{m;X,Y}$, (G2) gives
$a_e/|T_e|\in[9/20,11/20]$ and the remaining properties are those
of the certificate and $\mathcal D_m$.

Conditional on $\mathcal H_m^0$, the residual graph has defect at most
two and density margin $\lambda/2$, uniformly over the pair profile.
Estimate~\eqref{eq:good certificate probability} gives
the stated conditional estimate; no union over profile values
is needed. Average this estimate and add the bound for
$\bigcup_m\mathcal D_m^c$ above. The union over at most
$n$ levels and the fixed block family still costs only $e^{-cn}$.
\end{proof}

For completeness, we express how this environment acts on the equation
in \eqref{eq:qm}. Extend support vectors by zero to $\mathbb R^m$,
and, on a good fibre, set
\begin{align*}
 \xi_e:=\mathbf 1_{\mathcal O_e}-\frac{a_e}{|T_e|}\mathbf 1_{T_e},
 \qquad
 \zeta_j:=B_m^{(j)}-\frac{a_j}{n}\mathbf 1_O,
 \qquad b_e:=\mathbf e_x-\mathbf e_{x'}\quad(e=(x,x')).
\end{align*}
Let $\bar A_m:=\mathbb E(A_m\mid\mathcal H_{m;X,Y})$.

\begin{corollary}
On every good fibre,
\begin{align}
 A_m=\bar A_m
 +\frac1{\sqrt2}(h\eta^T+\eta h^T)
 +\sum_{e\in E}(b_e\xi_e^T+\xi_e b_e^T)+\sum_{j\in U_m}(\mathbf e_j\zeta_j^T+\zeta_j\mathbf e_j^T).\label{eq:hybrid affine representation}
\end{align}
Consequently, for every fixed $v\in H_m$,
\begin{align*}
 P_hA_mv=P_h\bar A_mv
 +\sum_{e\in E}
       \left(b_e\langle\xi_e,v\rangle
                     +\xi_e\langle b_e,v\rangle\right)+\sum_{j\in U_m}
       \left(\mathbf e_j\langle\zeta_j,v\rangle
                     +v_jP_h\zeta_j\right).
\end{align*}
Conditional on $\mathcal H_{m;X,Y}$, the summands indexed by distinct
$e$ and $j$ in the compressed action are independent.
The same assertions hold for an $\mathcal H_{m;X,Y}$-measurable $v$.
\end{corollary}

\begin{proof}
On the rectangle for $e=(x,x')$, the two deviations from the conditional
means are $\xi_e(y)$ and $-\xi_e(y)$. They contribute
$b_e\xi_e^T+\xi_e b_e^T$. The central deviations are
$\eta_i/2,-\eta_i/2$, contributing
$2^{-1/2}(h\eta^T+\eta h^T)$; column $j$ contributes
$\mathbf e_j\zeta_j^T+\zeta_j\mathbf e_j^T$.
These are disjoint sets of undirected edges, and all other entries are
exposed. This proves \eqref{eq:hybrid affine representation}.
For $v\in H_m$, the central term vanishes after applying $P_h$.
The switching vectors lie in $H_m$, while the extension vector requires
$P_h\zeta_j$. Independence follows from
\eqref{eq:exact hybrid factorization}, with $v$ fixed on the atom.
\end{proof}

To display all remaining randomness at once, order the coordinates as
$\{1,2\},X,Y,U_m$. Write $J=H-\bar H$ from
\eqref{eq:conditional cross block diagram} and
$Z=[\zeta_j]_{j\in U_m}$, with row restrictions $Z_{\{1,2\}},Z_X,Z_Y$.
Then
\begin{align}
 A_m-\bar A_m=
 \begin{pmatrix}
  0_{2\times2}&h_{\{1,2\}}\eta_X^T/\sqrt2
                 &h_{\{1,2\}}\eta_Y^T/\sqrt2&Z_{\{1,2\}}\\
  \eta_Xh_{\{1,2\}}^T/\sqrt2&0&J&Z_X\\
  \eta_Yh_{\{1,2\}}^T/\sqrt2&J^T&0&Z_Y\\
  Z_{\{1,2\}}^T&Z_X^T&Z_Y^T&0
 \end{pmatrix}.\label{eq:conditional complete matrix diagram}
\end{align}
Thus the unfixed data are exactly the central balanced slice,
one slice for each alternating rectangle, and one slice for each
extension column. The central slice disappears in $P_hA_m|_{H_m}$;
the other factors survive.

Finally,
$\mathcal F_m=\sigma(\Pi,G',C_m,(B_m^{(j)})_{j\in U_m})$
and $\mathcal H_{m;X,Y}\subseteq\sigma(\mathcal F_m,\pi)$.
A coefficient chosen from $C_{A_m}^{-1}$ in Section~3 need not be
fixed under the present exposure. The Fourier estimates below are
therefore applied to fixed coefficients; the deterministic nets in
Sections~\ref{Sec:threshold control} and~\ref{Sec:main proof} will
handle the existence of witnesses.

\section{Fourier comparison}\label{Sec:fourier}
Fix a good atom $\mathcal H_{m;X,Y}=\mathfrak c$ of the exposure in
Section~\ref{Sec:quasirandomness}. All supports, weights, pair vectors
and exposed entries are fixed. We write $\mathbb P,\mathbb E$ for the
true conditional law, and
$\phi_Z(\theta)=\mathbb E e^{2\pi i\langle Z,\theta\rangle}$.
The notation $\mathbb P_\omega,\mathbb E_\omega$ is reserved for the
auxiliary matching walks. Matrix operator and Hilbert--Schmidt norms
are denoted by $\|\cdot\|_{\mathrm{op}}$ and $\|\cdot\|_{\mathrm{HS}}$.

\subsection{The Fourier factors of the conditional matrix}

For fixed $v\in\mathbb R^m$, the decomposition in
\eqref{eq:hybrid affine representation} gives
\begin{align*}
 A_mv={}&\bar A_mv+\frac1{\sqrt2}
       \bigl(h\langle\eta,v\rangle+\eta\langle h,v\rangle\bigr)\\
 &+\sum_{e\in E}
       \bigl(b_e\langle\xi_e,v\rangle+\xi_e\langle b_e,v\rangle\bigr)
 +\sum_{j\in U_m}
       \bigl(\mathbf e_j\langle\zeta_j,v\rangle+v_j\zeta_j\bigr).
\end{align*}
The coefficient of each slice in the Fourier transform is therefore
\begin{align*}
 R_{e,v}(\theta)&=\langle b_e,\theta\rangle v_{T_e}
                      +\langle b_e,v\rangle\theta_{T_e},\\
 S_{j,v}(\theta)&=\theta_jv_O+v_j\theta_O,\\
 R_{*,v}(\theta)&=2^{-1/2}
       \bigl(\langle h,\theta\rangle v_{\mathcal T}
                         +\langle h,v\rangle\theta_{\mathcal T}\bigr).
\end{align*}
By the exact conditional product law,
\begin{align}
 \phi_{A_mv}(\theta)
 &=e^{2\pi i\langle\bar A_mv,\theta\rangle}
       \phi_\eta(R_{*,v}(\theta))
       \prod_{e\in E}\phi_{\xi_e}(R_{e,v}(\theta))
       \prod_{j\in U_m}\phi_{\zeta_j}(S_{j,v}(\theta)).\label{eq:true hybrid Fourier factorization}
\end{align}
This formula determines the auxiliary matrix: replace each centered
Boolean slice by a zero-sum matching walk with nonnegative characteristic
function dominating its absolute Fourier transform. The deterministic
matrix $\bar A_m$ contributes only a phase.

\subsection{From slices to matching walks}
\begin{proposition}\label{prop:dense slice Fourier}
Fix $c_0\in(0,1/2)$. There is
$c_{\ref{prop:dense slice Fourier}}=c_{\ref{prop:dense slice Fourier}}(c_0)>0$, depending only on $c_0$,
such that, for $s:=|T|\ge2$,
$q\sim\operatorname{Unif}\binom T a$ and
$a/s\in[c_0,1-c_0]$,
\begin{align}
 \left|\mathbb E e^{2\pi i\langle q,x\rangle}\right|
 \le \exp\left(-c_{\ref{prop:dense slice Fourier}}
                         \|x_T\|_{\mathbb D}^2\right),
 \qquad x\in\mathbb R^T.\label{eq:dense slice Fourier}
\end{align}
The same bound holds for $q-(a/s)\mathbf 1_T$.
\end{proposition}

\begin{proof}
Roos' permutation inequality \cite[Theorem~1.1]{Roos}, in the form
\cite[Theorem~3.4]{Tran-arxiv}, applied to a vector of $a$ ones and
$s-a$ zeros, gives
\begin{align*}
 |\phi_q(x)|
 &\le\left(1-\frac{a(s-a)}{\binom s2^2}
           \sum_{i<j}\sin^2(\pi(x_i-x_j))\right)^{(s-1)/4}\\
 &\le\exp\left(-\frac{4a(s-a)}{s^2}\|x_T\|_{\mathbb D}^2\right).
\end{align*}
Use $\sin^2(\pi u)\ge4\|u\|_{\mathbb R/\mathbb Z}^2$ and
$1-u\le e^{-u}$. Thus $c_{\ref{prop:dense slice Fourier}}=4c_0(1-c_0)$
is admissible; centering changes only the phase.
\end{proof}

\begin{definition}\label{def:lazy matching walk}
For a fixed finite support $T$ and $0<\nu\le1/4$, choose a uniform
perfect matching $\mathcal M_T$; if $|T|$ is odd, first leave a uniform
vertex unmatched. Given the matching, let its edges carry independent
signs with $\mathbb P_\omega(\varepsilon_{ij}=\pm1)=\nu/2$ and
$\mathbb P_\omega(\varepsilon_{ij}=0)=1-\nu$. Define
\begin{align*}
 \omega_{T,\nu}:=\sum_{\{i,j\}\in\mathcal M_T}
                 \varepsilon_{ij}(\mathbf e_i-\mathbf e_j),
 \qquad
 \boldsymbol\omega_T:=(\mathcal M_T,(\varepsilon_{ij})_{\{i,j\}\in\mathcal M_T}).
\end{align*}
The orientation convention is immaterial. For $|T|\le1$ the walk is
zero. Each walk has coordinate sum zero, and its complete seed
$\boldsymbol\omega_T$ includes both matching and signs.
The law $\mathbb P_\omega$ integrates these seeds, with the certificate
and coefficients fixed, independently of the true matrix.
\end{definition}

\begin{proposition}\label{prop:matching Fourier comparison}
There is an absolute $C_{\ref{prop:matching Fourier comparison}}>0$
such that, for $s:=|T|\ge2$ and $0<\nu\le1/4$,
\begin{align}
 \phi_{\omega_{T,\nu}}(x)
 \ge\exp\left(-C_{\ref{prop:matching Fourier comparison}}
                         \nu\|x_T\|_{\mathbb D}^2\right)>0.\label{eq:matching Fourier lower}
\end{align}
For a fixed dense slice window $[c_0,1-c_0]$, choose
\begin{align*}
 0<\nu_{\ref{prop:matching Fourier comparison}}
 \le\min\left\{\frac14,
 \frac{c_{\ref{prop:dense slice Fourier}}}
      {C_{\ref{prop:matching Fourier comparison}}}\right\}.
\end{align*}
Then every Boolean slice $q$ in this window satisfies, pointwise in $x$,
\begin{align*}
 \left|\mathbb E e^{2\pi i\langle q,x\rangle}\right|
 \le\phi_{\omega_{T,\nu_{\ref{prop:matching Fourier comparison}}}}(x),
 \qquad
 0\le\phi_{\omega_{T,\nu_{\ref{prop:matching Fourier comparison}}}}(x)
 \le1.
\end{align*}
The same domination holds for the centered slice.
\end{proposition}
\begin{proof}
Conditional on the matching, the characteristic function is
\begin{align}
 \mathbb E_{\omega}\left(e^{2\pi i\langle\omega_{T,\nu},x\rangle}
                         \mid\mathcal M_T\right)
 =\prod_{\{i,j\}\in\mathcal M_T}
       \left(1-2\nu\sin^2(\pi(x_i-x_j))\right).\label{eq:matching characteristic product}
\end{align}
Each factor lies in $[1/2,1]$. Since $1-u\ge e^{-2u}$ for
$0\le u\le1/2$, the product is at least
\begin{align*}
 \exp\left(-4\pi^2\nu
       \sum_{\{i,j\}\in\mathcal M_T}
                      \|x_i-x_j\|_{\mathbb R/\mathbb Z}^2\right).
\end{align*}
A fixed pair is present with probability $1/(s-1)$ when $s$ is
even and $1/s$ when $s$ is odd. Jensen's inequality therefore gives
\eqref{eq:matching Fourier lower} with
$C_{\ref{prop:matching Fourier comparison}}=4\pi^2$.
Combine this lower bound with \eqref{eq:dense slice Fourier} and
the choice of laziness. Nonnegativity and the upper bound
one also follow from \eqref{eq:matching characteristic product}.
\end{proof}
Fix henceforth one comparison laziness
$\nu_{\ref{prop:matching Fourier comparison}}$ valid for the common
window with
$c_0:=\min\{1/3,\lambda_*/2\}$.
It applies to every switching slice on a good certificate and every
extension slice on $\mathcal D_m$, as well as to the balanced central
slice. We abbreviate the corresponding walk by $\omega_T$.

\subsection{Construction of the auxiliary matrix}

\begin{definition}\label{def:auxiliary switching operator}
Take independent walks $\omega_e\sim\omega_{T_e}$ for $e\in E$
and $\omega_j\sim\omega_O$ for $j\in U_m$.
On $W=X\sqcup Y$ define
\begin{align}
 H_{\mathfrak c}
  :=\sum_{e\in E}(b_e|_X)(\omega_e|_Y)^T,
 \qquad
 M_{\mathfrak c}
  :=\begin{pmatrix}0&H_{\mathfrak c}\\
                   (H_{\mathfrak c})^T&0\end{pmatrix}.\label{eq:auxiliary switching operator}
\end{align}
Let $\iota_W:\mathbb R^W\to\mathbb R^O$ insert zeros at coordinates
$1,2$, and set $\mathsf Z_{\mathfrak c}=[\omega_j]_{j\in U_m}$.
The complete auxiliary matrix is
\begin{align}
 \widehat M_{\mathfrak c}
 &:=\begin{pmatrix}
 \iota_W M_{\mathfrak c}\iota_W^T&\mathsf Z_{\mathfrak c}\\
 \mathsf Z_{\mathfrak c}^T&0
 \end{pmatrix}\notag\\
 &=\sum_{e\in E}(b_e\omega_e^T+\omega_e b_e^T)
  +\sum_{j\in U_m}(\mathbf e_j(\omega_j)^T
                                +\omega_j\mathbf e_j^T).\label{eq:hybrid auxiliary operator}
\end{align}
For comparison before compression also take an independent
$\omega_*\sim\omega_{\mathcal T}$ and put
$\widetilde M_{\mathfrak c}
 =\widehat M_{\mathfrak c}+\sqrt2(h\omega_*^T+\omega_*h^T)$.
All walks are extended by zero to their ambient coordinates.
\end{definition}

The block structure can now be read directly from
\eqref{eq:conditional complete matrix diagram}:
\begin{align*}
 \widehat M_{\mathfrak c}
 =\begin{pmatrix}
 0_{2\times2}&0&0&(\mathsf Z_{\mathfrak c})_{\{1,2\}}\\
 0&0&H_{\mathfrak c}&(\mathsf Z_{\mathfrak c})_X\\
 0&(H_{\mathfrak c})^T&0&(\mathsf Z_{\mathfrak c})_Y\\
 (\mathsf Z_{\mathfrak c})_{\{1,2\}}^T
   &(\mathsf Z_{\mathfrak c})_X^T&(\mathsf Z_{\mathfrak c})_Y^T&0
 \end{pmatrix},
 \qquad
 H_{\mathfrak c}=
 \begin{array}{c|cc}
 &Y_0&Y_1\\\hline
 x_1&\omega_{e_1}^T&0\\
 x_2&-\omega_{e_1}^T&\omega_{e_2}^T\\
 x_3&\omega_{e_3}^T&-\omega_{e_2}^T\\
 \vdots&\vdots&\vdots
 \end{array}.
\end{align*}
Each true switching factor $\xi_e$ has become $\omega_e$, each extension
factor $\zeta_j$ has become $\omega_j$, and the exposed
mean has disappeared. In particular,
\begin{align}
 \widehat M_{\mathfrak c}v
 &=\sum_e\bigl(b_e\langle\omega_e,v\rangle
                          +\omega_e\langle b_e,v\rangle\bigr)
   +\sum_{j\in U_m}\bigl(\mathbf e_j\langle\omega_j,v\rangle
                                      +v_j\omega_j\bigr)
                                      \notag\\
 &=\begin{pmatrix}
  \iota_W M_{\mathfrak c}v_W
                 +\sum_{j\in U_m}v_j\omega_j\\
  (\langle\omega_j,v_O\rangle)_{j\in U_m}
 \end{pmatrix}.\label{eq:hybrid auxiliary action}
\end{align}
The factor $\sqrt2$ in the full central term comes from writing
$\eta=2q_*-\mathbf1_{\mathcal T}$: its characteristic function is
compared with $\phi_{\omega_*}(2u)$, rather than $\phi_{\omega_*}(u)$.

\begin{proposition}\label{prop:hybrid Fourier domination}
On the fixed good hybrid fibre, for every fixed $v,\theta\in\mathbb R^m$,
\begin{align*}
 |\phi_{A_mv}(\theta)|
 \le\phi_{\widetilde M_{\mathfrak c}v,\omega}(\theta),
 \qquad
 \phi_{\widetilde M_{\mathfrak c}v,\omega}(\theta)\ge0.
\end{align*}
If $v\in H_m$ and $Q$ is an orthogonal projection with
$\operatorname{Im}Q\subseteq H_m$, then
\begin{align}
 |\phi_{QA_mv}(\theta)|
 \le\phi_{Q\widehat M_{\mathfrak c}v,\omega}(\theta),
 \qquad
 \phi_{Q\widehat M_{\mathfrak c}v,\omega}(\theta)\ge0.\label{eq:compressed hybrid Fourier domination}
\end{align}
The assertions also hold when $v$ and $Q$ are
$\mathcal H_{m;X,Y}$-measurable.
\end{proposition}

\begin{proof}
Take absolute values in \eqref{eq:true hybrid Fourier factorization}.
The single-slice comparison gives
\begin{align*}
 |\phi_{A_mv}(\theta)|
 &\le\phi_{\omega_*}(2R_{*,v}(\theta))
       \prod_e\phi_{\omega_e}(R_{e,v}(\theta))
       \prod_{j\in U_m}\phi_{\omega_j}
                                  (S_{j,v}(\theta))\\
 &=\phi_{\widetilde M_{\mathfrak c}v,\omega}(\theta)\ge0.
\end{align*}
All factors are nonnegative and the walks are independent.
For $v\in H_m$ substitute $Q\theta$ for $\theta$. Since $Qh=0$,
$R_{*,v}(Q\theta)=0$ and
$Q\widetilde M_{\mathfrak c}v=Q\widehat M_{\mathfrak c}v$.
This proves the compressed comparison. For $Q=P_h$ the extension
summand is
$\mathbf e_j\langle\omega_j,v\rangle
       +v_jP_h\omega_j$; the latter projection is necessary.
Exposure-measurable coefficients and projections are fixed on each atom.
\end{proof}

For the residual graph $A_{\mathfrak c}$ the same proof, retaining
only the switching factors, gives
\begin{align*}
 |\phi_{A_{\mathfrak c}v}(\theta)|
 \le\phi_{M_{\mathfrak c}v,\omega}(\theta)
 =\prod_e\phi_{\omega_e}(R_{e,v}(\theta))\ge0.
\end{align*}
Every fixed linear output map preserves this comparison by
$\phi_{DZ}(\theta)=\phi_Z(D^T\theta)$.

\subsection{Gaussian smoothing and Fourier capture}

The Fourier domination gives concentration uniformly over the center
after Gaussian smoothing. Its exponent is the rank of the output
projection, including the one or two dimensions removed in the final
application.

\begin{definition}\label{def:auxiliary threshold}
Fix $\tau_0\in(0,1/8)$ once and for all. Let $L\ge2$
and $s_0:=\tau_0/L$. For an auxiliary random operator
$\mathsf M$ and a fixed orthogonal output projection $Q$ of rank
$r\ge1$, define
\begin{align}
 T_L(v;\mathsf M,Q)
 :=\sup\left(\{0\}\cup
 \left\{s\in(0,s_0]:
 \mathbb P_{\omega}
       (\|Q\mathsf Mv\|_2\le s\sqrt r)\ge(4Ls)^r\right\}\right).\label{eq:auxiliary threshold}
\end{align}
The probability integrates all randomness of $\mathsf M$, including
its matchings. For a certificate-dependent operator, this is its
auxiliary law on that fixed certificate. If $\mathsf M,Q$ and the
certificate are specified, we write simply $T_L(v)$.
\end{definition}

For the hybrid law write
$T_{L,\mathfrak c}^Q(v):=T_L(v;\widehat M_{\mathfrak c},Q)$.
The certificate, operator law and projection are held fixed when $v$
varies. This is abbreviated to $T_L(v)$ on a fixed fibre.

\begin{theorem}\label{thm:Fourier capture}
Let $Z$ be a random vector under a true law, let $\mathsf M$ be an
auxiliary random operator, and fix $v$ and an orthogonal projection
$Q$ of rank $r\ge1$. Suppose, pointwise on $\operatorname{Im}Q$,
\begin{align*}
 |\phi_{QZ}(\theta)|
 \le\phi_{Q\mathsf Mv,\omega}(\theta),
 \qquad \phi_{Q\mathsf Mv,\omega}(\theta)\ge0.
\end{align*}
There is $C_{\ref{thm:Fourier capture}}=C_{\ref{thm:Fourier capture}}(\tau_0)>0$, depending only on
$\tau_0$, such that
\begin{align}
 T_L(v;\mathsf M,Q)<t\le s_0/2
 \quad\Longrightarrow\quad
 \mathcal L(QZ,t\sqrt r)
 \le(C_{\ref{thm:Fourier capture}}Lt)^r.\label{eq:Fourier capture}
\end{align}
No assumption that $r$ is close to the ambient dimension is needed.
\end{theorem}

We use the Gaussian Fourier identity
\begin{align}
 \mathbb E e^{-\pi\|Z-z\|_2^2/a^2}
 =\int e^{-2\pi i\langle z,\theta\rangle/a}
                 \phi_Z(\theta/a)\,d\gamma_r(\theta).\label{eq:Gaussian Fourier identity}
\end{align}
Fourier inversion and Fubini prove it for every $Z\in\mathbb R^r$,
$a>0$ and $z\in\mathbb R^r$. In particular, if
$|\phi_Z|\le\phi_{Z',\omega}$ with $\phi_{Z',\omega}\ge0$, then
\begin{align*}
 \sup_z\mathbb E e^{-\pi\|Z-z\|_2^2/a^2}
 \le\mathbb E_\omega e^{-\pi\|Z'\|_2^2/a^2}.
\end{align*}
A law with nonnegative characteristic function thus has its largest
Gaussian moment at zero. Both assertions hold intrinsically in any
Euclidean subspace.

\begin{proof}[Proof of Theorem~\ref{thm:Fourier capture}]
Work intrinsically in $\operatorname{Im}Q$. Centres outside this
space can be projected into it, increasing the probability of the
corresponding ball. For every centre $z$ in this space,
\begin{align*}
 \mathbb P(\|QZ-z\|_2\le t\sqrt r)\le e^{\pi r}
       \mathbb Ee^{-\pi\|QZ-z\|_2^2/t^2}\le e^{\pi r}
       \mathbb E_{\omega}e^{-\pi\|Q\mathsf Mv\|_2^2/t^2}.
\end{align*}
The second line uses \eqref{eq:Gaussian Fourier identity},
before any supremum or coefficient average is taken.

Set $R:=\|Q\mathsf Mv\|_2/\sqrt r$ and
$F(s):=\mathbb P_{\omega}(R\le s)$.
The threshold hypothesis gives $F(s)<(4Ls)^r$ for
$t\le s\le s_0$, and monotonicity gives $F(s)\le(4Lt)^r$
for $s<t$. Tonelli's theorem yields
\begin{align*}
 \mathbb E_{\omega}e^{-\pi rR^2/t^2}
 &=\int_0^\infty\frac{2\pi r s}{t^2}
                  e^{-\pi rs^2/t^2}F(s)\,ds\\
 &\le(4Lt)^r\left(1+2\pi r\int_0^\infty
                         u^{r+1}e^{-\pi ru^2}\,du\right)
       +e^{-\pi r(s_0/t)^2}.
\end{align*}
The parenthesized factor is bounded by an absolute constant to
the power $r$. For example, bound
$(u e^{-\pi u^2/2})^r$ by its supremum and integrate the remaining
factor $2\pi r u e^{-\pi ru^2/2}$.
Writing $u=s_0/t\ge2$, the last term is
\begin{align*}
 e^{-\pi ru^2}
 =\left(Lt\,\frac{u}{\tau_0}e^{-\pi u^2}\right)^r
 \le(CLt)^r,
\end{align*}
where $C$ depends only on $\tau_0$.
Absorb the factor $e^{\pi r}$ and take the supremum over $z$.
Every exponent in this calculation is the intrinsic output
dimension $r$, proving \eqref{eq:Fourier capture}.
\end{proof}

\begin{corollary}\label{cor:conditional hybrid capture}
Let $v\in H_m$ and $Q$ be $\mathcal H_{m;X,Y}$-measurable, with
$\operatorname{Im}Q\subseteq H_m$ and $r:=\operatorname{rank}Q\ge1$.
For $0<t\le s_0/2$, set
\begin{align*}
 \mathcal J_t
 :=\mathcal E_{m;X,Y}\cap\{T_{L,\mathfrak c}^Q(v)<t\}.
\end{align*}
Then $\mathcal J_t\in\mathcal H_{m;X,Y}$ and, almost surely,
\begin{align*}
 \mathbf 1_{\mathcal J_t}
   \sup_{z\in\mathbb R^m}
      \mathbb P\left(\|QA_mv-z\|_2\le t\sqrt r
                              \mid\mathcal H_{m;X,Y}\right)\le
 \mathbf 1_{\mathcal J_t}(C_{\ref{thm:Fourier capture}}Lt)^r.
\end{align*}
The supremum is evaluated on each conditional fibre.
\end{corollary}

\begin{proof}
Apply Theorem~\ref{thm:Fourier capture} on each good atom using
\eqref{eq:compressed hybrid Fourier domination}. The auxiliary law is
fixed by the exposed data, so $T_{L,\mathfrak c}^Q(v)$ is measurable
there; define it as $s_0$ on bad atoms. Multiplication by the measurable indicator gives the claim.
\end{proof}
For $Q=P_h$ the exponent is $m-1$. For the projection onto
$\operatorname{span}\{h,b\}^\perp$, with fixed $b\in H_m$, it is
$m-2$ if $b\ne0$ and $m-1$ otherwise.

\subsection{A local augmented small-ball estimate}

Two independent factors will test the same arithmetic direction together
with additional soft columns. The Gaussian cylinder lemma is imported
from \cite{CJMS-jams}; the centered support and matching-walk estimates
needed for its application are proved in
Appendix~\ref{app:matching Gaussian estimates}.

For a set $T$ of size $s$, a vector $y\in\mathbb R^T$, and
$F=(F_0^T,F_1^T)^T\in\mathbb R^{2s\times k}$ with
$F_b\in\mathbb R^{T\times k}$, define
\begin{align*}
 \mathcal A_T(F,y)
 :=\begin{bmatrix}
       F_0&y&0\\
       F_1&0&y
   \end{bmatrix}
 \in\mathbb R^{2s\times(k+2)}.
\end{align*}
For $z=(z^{(0)},z^{(1)})\in\mathbb R^T\oplus\mathbb R^T$, set
\begin{align*}
 \|z\|_{\mathbb D,T\oplus T}^2
 :=\|z^{(0)}_T\|_{\mathbb D}^2+\|z^{(1)}_T\|_{\mathbb D}^2.
\end{align*}

For a Borel $S\subseteq\mathbb R^{k+2}$ put
$F_w(S;a,b)=\{u\in\mathbb R^k:(u,a,b)\in S-w\}$ and
$\mathfrak m(S)=\sup_{a,b,w}\gamma_k(F_w(S;a,b)-F_w(S;a,b))$.
With the Gaussian normalization of Section~\ref{Sec:notation}, the
following is the Gaussian cylinder lemma of \cite{CJMS-jams}.
For $k=0$, its conclusion follows by partitioning $\mathbb R^2$
into squares of side $16$: if the conclusion fails, each occupied
piece has coordinate diameters at most $v_0$, giving
$\gamma_2(S)\le2v_0^2$ and contradicting the premise.

\begin{lemma}\label{lem:Gaussian cylinder}
Let $k\ge0$ and $v_0>0$. If $S\subset\mathbb R^{k+2}$ is
Borel and
\begin{align*}
 \gamma_{k+2}(S)
 >8v_0^2e^{-k/8}+64v_0^2\mathfrak m(S)^{1/4},
\end{align*}
then $S-S$ contains a point $(\theta,\varphi_0,\varphi_1)$ with
\begin{align*}
 \|\theta\|_2\le2\sqrt k,\qquad
 v_0<\max\{|\varphi_0|,|\varphi_1|\}\le16.
\end{align*}
\end{lemma}
\begin{theorem}\label{thm:local augmented small ball}
Fix $c_0\in(0,1/2)$, $\nu_1\in(0,1/4]$, $\gamma\in(0,1)$
and $\kappa,\kappa_y>0$.
There are positive constants
$c_{\ref{thm:local augmented small ball}}^{(i)}=c_{\ref{thm:local augmented small ball}}^{(i)}(c_0,\nu_1,\gamma,\kappa,\kappa_y)$ for $1\le i\le3$,
$b_{\ref{thm:local augmented small ball}}=b_{\ref{thm:local augmented small ball}}(c_0,\nu_1,\gamma,\kappa,\kappa_y)$ and
$C_{\ref{thm:local augmented small ball}}=C_{\ref{thm:local augmented small ball}}(c_0,\nu_1,\gamma,\kappa,\kappa_y)$, depending only on
the displayed parameters. The following holds. Let $s:=|T|\ge2$, $0<\mu<1$, $t>0$ and let $k\ge0$ be an integer.
Suppose $y\in\mathbb R^T$ and
$F\in\mathbb R^{2s\times k}$ satisfy
\begin{align}
 F_b^T\mathbf 1_T=0\quad(b=0,1),\qquad
 \|F\|_{\mathrm{op}}\le2,\qquad
 \|F\|_{\mathrm{HS}}^2\ge\kappa k,\label{eq:local soft hypotheses}
\end{align}
and
\begin{align*}
 k \le {c_{\ref{thm:local augmented small ball}}^{(1)}}\mu^2s, \
 t \ge e^{-{c_{\ref{thm:local augmented small ball}}^{(2)}}\mu^2s}, \  \|P_{T,0}y\|_2 \ge\kappa_y/t, \  \operatorname{CLCD}_{\mu s,\gamma}^{T}(y) >16.
\end{align*}
Let $q^{(0)},q^{(1)}$ be independent vectors, each either a
Boolean slice on $T$ with density in $[c_0,1-c_0]$, or a matching
walk with law $\omega_{T,\nu_1}$. Then
\begin{align}
 \mathcal L\left(
   \mathcal A_T(F,y)^T(q^{(0)},q^{(1)}),
   b_{\ref{thm:local augmented small ball}}\sqrt{k+2}\right)
 \le (C_{\ref{thm:local augmented small ball}}t)^2
                       e^{-{c_{\ref{thm:local augmented small ball}}^{(3)}}k}.\label{eq:local augmented small ball}
\end{align}
For $k=0$ the matrix $F$ and its conditions are absent.
The concentration function is computed under the stated independent
product law of the two inputs: $\mathbb P$ for two slices,
$\mathbb P_{\omega}$ for two matching walks, and the product of those
independent laws in the mixed case. The constants are independent of
$\mu$, and there is no additional restriction $t<1$.
\end{theorem}

\begin{proof}
Write $\mathcal A=\mathcal A_T(F,y)$, $q=(q^{(0)},q^{(1)})$ and
$E(\theta,\varphi_0,\varphi_1)=\sum_{b=0}^1
                  \|(F_b\theta+\varphi_by)_T\|_{\mathbb D}^2$.
Proposition~\ref{prop:dense slice Fourier} and
Proposition~\ref{prop:matching estimates} give
$|\phi_{\mathcal A^Tq}|\le C'e^{-aE}$ with $a>0,C'\ge1$
independent of $\mu$. Gaussian smoothing at scale one yields
\begin{align*}
 \mathcal L(\mathcal A^Tq,b_*\sqrt{k+2})
 \le C'e^{\pi b_*^2(k+2)}\int e^{-aE}\,d\gamma_{k+2}.
\end{align*}

Choose an auxiliary laziness $\nu_0$ with
$C_{\ref{prop:matching Fourier comparison}}\nu_0<a/8$,
and let $C'',c'>0$ be the constants in
Proposition~\ref{prop:matching estimates} for $(\kappa,\nu_0)$.
Choose $b_*,c_*>0$ so that
$d_*:=\pi b_*^2+c_*<\min(1/16,c'/8)$.
These choices precede $\mu$ and do not change the law of $q$.
If the conclusion with $b_*,c_*$ and a sufficiently large $C_*$ fails,
the layer-cake formula supplies a level $u\ge0$ such that
\begin{align}
 S:=\{E\le u\},\qquad
 \gamma_{k+2}(S)\ge B^2t^2 e^{au/2-d_*k},\label{eq:augmented large level}
\end{align}
where $B$ can be made arbitrarily large by choosing $C_*$ large.
Indeed, failure of this inequality for all $u$ would imply
$\int e^{-aE}\,d\gamma_{k+2}\le2B^2t^2e^{-d_*k}$, contradicting the smoothing bound.

The torus triangle inequality gives $S-S\subseteq\{E\le4u\}$.
Differences within a horizontal fibre have zero arithmetic coordinates;
Proposition~\ref{prop:matching estimates} consequently gives
\begin{align*}
 \mathfrak m(S)\le C''\exp\bigl(
       4C_{\ref{prop:matching Fourier comparison}}\nu_0u-c'k\bigr).
\end{align*}
Fix $D\ge8/(\gamma\kappa_y)$ and set
$v_0=Dt(\sqrt u+\sqrt k+1)$.
Dividing the premise in Lemma~\ref{lem:Gaussian cylinder} by
\eqref{eq:augmented large level} bounds the ratio by
\begin{align*}
 \frac{CD^2(u+k+1)}{B^2}\bigl[
 e^{-au/2-(1/8-d_*)k}
 +{C''}^{1/4}e^{-(a/2-C_{\ref{prop:matching Fourier comparison}}\nu_0)u
                         -(c'/4-d_*)k}\bigr].
\end{align*}
All exponent coefficients are positive, so increasing $B$ makes
this less than one uniformly in $u,k$. The cylinder lemma supplies
$(\theta,\varphi_0,\varphi_1)\in S-S$ with
$\|\theta\|_2\le2\sqrt k$, $|\varphi_b|\le16$, and
$|\varphi_b|>Dt(\sqrt u+\sqrt k+1)$ for at least one $b$.

For that $b$, centering of $F_b$ and $\|F_b\|_{\mathrm{op}}\le2$ give
\begin{align*}
 \operatorname{dist}(\varphi_bD_Ty,\mathbb Z^{\binom s2})
 &\le2\sqrt{su}+\|D_TF_b\theta\|_2
 \le4\sqrt s(\sqrt u+\sqrt k)\\
 &<\gamma|\varphi_b|\|D_Ty\|_2.
\end{align*}
The last inequality uses $\|D_Ty\|_2=\sqrt s\|P_{T,0}y\|_2$
and the choice of $D$.
Since $\gamma_{k+2}(S)\le1$ and $B\ge1$, \eqref{eq:augmented large level}
also implies
$u+k\le(4/a)\log_+(1/t)+(1+2d_*/a)k$.
Choose the two arithmetic range constants in the statement
small enough that $32s(u+k)<\mu^2s^2$.
The square of the preceding distance bound
is at most $32s(u+k)$, so the absolute CLCD cutoff is satisfied as well.
Thus $0<|\varphi_b|\le16$ contradicts the CLCD hypothesis.
The argument includes $k=0$ and $t\ge1$; the additional one in $v_0$
ensures that the arithmetic frequency is nonzero.
\end{proof}

The local estimate also applies fibrewise. If $\mathcal F$ fixes
$T,F,y,k$ and leaves the two tested inputs independent with the
stated laws, then on any $\mathcal F$-measurable event $\mathcal E$
where its hypotheses hold,
\begin{align*}
 \mathbf1_{\mathcal E}\,
 \mathcal L\left(\mathcal A_T(F,y)^T(q^{(0)},q^{(1)}),
 b_{\ref{thm:local augmented small ball}}\sqrt{k+2}
 \,\middle|\,\mathcal F\right)\le\mathbf1_{\mathcal E}
 (C_{\ref{thm:local augmented small ball}}t)^2
 e^{-c_{\ref{thm:local augmented small ball}}^{(3)}k}.
\end{align*}
The supremum over centres permits an $\mathcal F$-measurable target.
The conditioning may include all other complete matching seeds,
but must leave both current walks fresh.

For later use, signed coordinates $u_i\ge a$ on $P_+$ and $u_i\le-a$
on $P_-$ imply
\begin{align}
 \|P_{T_e,0}u\|_2^2
 \ge\frac{4a^2}{|T_e|}|P_+\cap T_e|\,|P_-\cap T_e|.\label{eq:local hard norm from support}
\end{align}
This follows from \eqref{eq:difference-identities}; the certificate
supplies the intersection sizes on all but a fixed number of supports.

\begin{corollary}\label{cor:extension scalar Gaussian estimate}
Fix $a_0,r_0,c_g>0$ and $\gamma\in(0,1)$, and let $O=[n]$.
Write $\nu=\nu_{\ref{prop:matching Fourier comparison}}$.
There are constants $c_{\ref{cor:extension scalar Gaussian estimate}}=c_{\ref{cor:extension scalar Gaussian estimate}}(a_0,r_0,c_g,\gamma,\nu)>0$
and $C_{\ref{cor:extension scalar Gaussian estimate}}=C_{\ref{cor:extension scalar Gaussian estimate}}(a_0,r_0,c_g,\gamma,\nu)>0$, depending
only on these parameters and the comparison laziness, such that
the following holds. Let $m\ge n\ge2$, $M\ge1$, $0<\mu<1$ and
$x\in\mathbb R^O$. Set $\widehat x:=(r_0/\sqrt m)x$ and suppose
\begin{align*}
 \|P_{O,0}\widehat x\|_2\ge a_0M, \ 
 \operatorname{CLCD}_{\mu n,\gamma}^{O}(\widehat x)>16, \ 
 \log M\le c_{\ref{cor:extension scalar Gaussian estimate}}\mu^2n.
\end{align*}
For $\omega\sim\omega_O$,
\begin{align}
 \sup_{a\in\mathbb R}
 \mathbb E_{\omega}
      \exp\left(-\frac{c_g}{m}(\langle\omega,x\rangle-a)^2\right)
 \le\frac{C_{\ref{cor:extension scalar Gaussian estimate}}}{M}.\label{eq:extension scalar Gaussian estimate}
\end{align}
The constants are independent of $\mu$.
\end{corollary}

\begin{proof}
Apply Theorem~\ref{thm:local augmented small ball} with $k=0$,
$y=\widehat x$, $\kappa_y=1$ and $t=R_0/M$, where
$R_0\ge\max(1,a_0^{-1})$ is fixed. The scale restriction follows by
decreasing $c_{\ref{cor:extension scalar Gaussian estimate}}$.
Two independent copies and a square root give
$\mathcal L(\langle\omega,x\rangle,u_0)\le C/M$, where
$u_0=b_{\ref{thm:local augmented small ball}}\sqrt m/r_0$.
An interval of radius $(j+1)u_0$ is covered by $C(j+1)$ intervals
of radius $u_0$. Gaussian shell summation therefore gives, uniformly in $a$,
\begin{align*}
 \mathbb E_\omega e^{-c_g(\langle\omega,x\rangle-a)^2/m}
 \le\frac CM\sum_{j\ge0}(j+1)
     e^{-c_gb_{\ref{thm:local augmented small ball}}^2j^2/r_0^2}
 \le\frac{C_{\ref{cor:extension scalar Gaussian estimate}}}{M}.
\end{align*}
\end{proof}

\section{Inversion of randomness}\label{Sec:inversion}
On a fixed coefficient box, two independent residual operators test the
same vector. Small singular directions provide extra columns in the
local estimate of Section~\ref{Sec:fourier}; tensorization along the
alternating path gives a residual second moment. The extension walks
supply the remaining dimensions. Throughout this section the certificate
is fixed, $\mathbb P_Z$ is uniform product measure on the stated box,
and $\mathbb P_\omega$ integrates only the independent auxiliary seeds.

\subsection{The product-box count}\label{subsec:inversion setting}
Work on a fixed good hybrid fibre from
Theorem~\ref{thm:exact hybrid factorization}. Retain its notation
$O=[n]$, $W=\{3,\ldots,n\}$, $U_m=\{n+1,\ldots,m\}$ and
$W=X\sqcup Y$, where $Y=Y_0\sqcup Y_1$ is the fixed balanced splitter.
Write $N=n-2$, $d_Y=|Y|$, $p_X=|X|=N-d_Y$ and $d_0=m-n$.
We assume $n\le m\le(1+\kappa)n$, with fixed $0<\kappa<1/4$.
The ordering of $X$ and every support $T_e$ are part of the
certificate. Write $H_{\mathfrak c}$, $M_{\mathfrak c}$ and
$\widehat M_{\mathfrak c}$ for the operators in
\eqref{eq:auxiliary switching operator} and
\eqref{eq:hybrid auxiliary operator}. In particular, all matching
walks used in these operators are sampled on the auxiliary
probability space of Section~\ref{Sec:fourier}.

Fix $\varepsilon\in(0,1/4)$ as in
Definition~\ref{def:good switching certificate}, and let
\begin{align*}
 \ell_*:=\ell_{\ref{thm:fixed block quasirandomness}},\qquad
 J_*:=4(\ell_*-1),\qquad \sigma_*:=\frac{\varepsilon}{3\ell_*}.
\end{align*}
All block proportions below are fixed independently of $L$.
We use a sufficiently small number $\theta_0>0$ and require
\begin{align}
 \theta_0N\le d_Y\le2\theta_0N.\label{eq:inversion block proportion}
\end{align}
The family realizing these proportions is chosen in
Section~\ref{subsec:final signed geometry}.
The coverage coefficient in \eqref{eq:centered energy coverage}
may be fixed before $\theta_0$: the proof of
Lemma~\ref{lem:pair coverage and energy} takes this coefficient
to depend only on the column-density window. The lower dimension
and the probability of obtaining the certificate may depend on
$\theta_0$. Here the certificate has already been fixed.

\begin{definition}\label{def:signed coefficient box}
Fix $0<a<b$, $c_I,C_I,B,B_Y>0$ and $M_0\ge1$.
A signed coefficient box of scale $M\ge M_0$ is a set
$\mathcal B=\prod_{i=1}^m I_i\subset\mathbb Z^m$, where each $I_i$
is a nonempty interval of consecutive integers satisfying
\begin{align}
 |I_i|\ge c_I M,\qquad \operatorname{diam}(I_i)\le C_I M.\label{eq:coefficient interval bounds}
\end{align}
The box also satisfies
\begin{align*}
 \sup_{z\in\mathcal B}\|z\|_2\le B\sqrt m\,M,
 \qquad
 \sup_{z\in\mathcal B}\|z_Y\|_2\le B_Y\sqrt{d_Y}\,M.
\end{align*}
For each $b'\in\{0,1\}$ there are fixed disjoint sets
$P_{b',+},P_{b',-}\subseteq Y_{b'}$, each of cardinality at least
$2\varepsilon|Y_{b'}|$, such that
\begin{align*}
 I_i\subseteq[aM,bM] \text{ for }i\in P_{b',+}, \ 
 I_i\subseteq[-bM,-aM]\text{ for }i\in P_{b',-}.
\end{align*}
Uniform measure on $\mathcal B$ is denoted by $\mathbb P_Z$.
It is independent of every auxiliary matching walk.

In $\mathbb P_Z\{\mathbb P_\omega(\cdot)\ge a\}$, the inner
probability fixes $Z$ and integrates the walks; the outer probability
counts coefficients under the unchanged product law. The notation
$\mathbb P_Z$ also denotes its marginals, including those named $U$
or $(V,U)$ below.
\end{definition}
Let $\mathcal D$ consist of the fresh support indices in half $Y_{b'}$
for which either signed intersection has size less than
$\varepsilon|Y_{b'}|/\ell_*$. By the consequence of (G4) noted after
Definition~\ref{def:good switching certificate},
\begin{align*}
 |\mathcal D|\le J_*,\qquad
 |T_e\cap P_{b',+}|,\ |T_e\cap P_{b',-}|
       \ge\sigma_*d_Y\quad(e\in E_{b'}\setminus\mathcal D),
\end{align*}
for all sufficiently large $N$. The set $\mathcal D$ depends only
on the certificate and the prescribed signed sets. It does not
depend on the sampled coefficient vector. Equal supports with
different indices are still counted separately.

\begin{theorem}\label{thm:hybrid box inversion}
Fix $\lambda,\kappa,\varepsilon$, the parameters
$a,b,c_I,C_I,B,B_Y,M_0$ of Definition~\ref{def:signed coefficient box},
a number $C_h>0$ and an integer $\ell_0\ge0$. Choose $\theta_0$ in
\eqref{eq:inversion block proportion} sufficiently small in terms
of this fixed data. Write
\begin{align*}
 \mathbf p=(\lambda,\kappa,\varepsilon,a,b,c_I,C_I,B,B_Y,M_0,C_h,\ell_0,\theta_0).
\end{align*}
There are constants
$R_{\ref{thm:hybrid box inversion}}=R_{\ref{thm:hybrid box inversion}}(\mathbf p)>1$ and integers
${C_{\ref{thm:hybrid box inversion}}^{(1)}=C_{\ref{thm:hybrid box inversion}}^{(1)}(\mathbf p)},
{C_{\ref{thm:hybrid box inversion}}^{(2)}=C_{\ref{thm:hybrid box inversion}}^{(2)}(\mathbf p)}\ge0$ with the following property.
For every sufficiently large $L$, there are
${c_{\ref{thm:hybrid box inversion}}=c_{\ref{thm:hybrid box inversion}}(\mathbf p,L)}>0$ and
$m_{L,\ref{thm:hybrid box inversion}}=m_{\ref{thm:hybrid box inversion}}(\mathbf p,L)$ such that, whenever
\begin{align*}
 m\ge m_{L,\ref{thm:hybrid box inversion}},\qquad
 M_0\le M\le \exp({c_{\ref{thm:hybrid box inversion}}(L)}m),
\end{align*}
every signed coefficient box and every deterministic orthogonal
projection $Q$ on $\mathbb R^m$ of codimension
$\ell_Q\le\ell_0$ satisfy
\begin{align}
 \mathbb P_Z\left\{
   \mathbb P_{\omega}\left(
     \|Q\widehat M_{\mathfrak c} Z\|_2\le C_hm\right)
       \ge (L/M)^{m-\ell_Q}\right\}
 \le m^{{C_{\ref{thm:hybrid box inversion}}^{(1)}}}
       M^{{C_{\ref{thm:hybrid box inversion}}^{(2)}}}
       \left(\frac{R_{\ref{thm:hybrid box inversion}}}{L}\right)^{2m}.\label{eq:hybrid box inversion}
\end{align}
The estimate is uniform over certificates, boxes and projections,
which are fixed before either probability is taken. Only the scale
range and the lower dimension may depend on $L$.
\end{theorem}

\subsection{Arithmetic exceptions and deterministic bounds}

The arithmetic condition is required on the actual switching supports.
Its failures are counted under the original product measure; that
measure is never conditioned on the arithmetic good set.

\begin{lemma}\label{lem:box CLCD exceptions}
Fix $a,b,c_I,C_I,B_*,r_*,\rho,\sigma,C_*>0$, with $a<b$ and
$\rho,\sigma<1$, and fix
$0<\gamma_*<\min\{1,a\sigma/(2B_*)\}$. Put
$\mathbf p_0=(a,b,c_I,C_I,B_*,r_*,\rho,\sigma,C_*,\gamma_*)$.
Let $\mathcal I=\prod_{i=1}^D I_i$ satisfy
\eqref{eq:coefficient interval bounds} and
$\sup_{u\in\mathcal I}\|u\|_2\le B_*\sqrt D\,M$.
Let $N_*$ be an integer with
$\rho N_*\le D\le\rho^{-1}N_*$.
Suppose $\mathscr T$ is a fixed family of at most $C_*N_*$
indexed subsets of $[D]$. For each $T\in\mathscr T$ there are
fixed subsets $P_{T,+},P_{T,-}\subseteq T$, each of size at least
$\sigma D$, on which every point of $\mathcal I$ satisfies,
respectively, $u_i\in[aM,bM]$ and $u_i\in[-bM,-aM]$.
For every $A>1$ there are
$\mu_{A,\ref{lem:box CLCD exceptions}}=\mu_{\ref{lem:box CLCD exceptions}}(\mathbf p_0,A)\in(0,1)$,
${c_{\ref{lem:box CLCD exceptions}}=c_{\ref{lem:box CLCD exceptions}}(\mathbf p_0,A)}>0$ and
$N_{A,\ref{lem:box CLCD exceptions}}=N_{\ref{lem:box CLCD exceptions}}(\mathbf p_0,A)$ such that, for
$N_*\ge N_{A,\ref{lem:box CLCD exceptions}}$ and
$1\le M\le e^{{c_{\ref{lem:box CLCD exceptions}}(A)}N_*}$,
\begin{align}
 \mathbb P_Z\left\{\exists T\in\mathscr T:
  \operatorname{CLCD}_{\mu_{A,\ref{lem:box CLCD exceptions}}|T|,
                    \gamma_*}^{T}
       \left(\frac{r_*}{\sqrt{N_*}}U_T\right)\le16\right\}
 \le A^{-2N_*}.\label{eq:box CLCD exceptions}
\end{align}
Here $\mathbb P_Z$ is the uniform product measure of $U$ on the fixed
coefficient box $\mathcal I$.
The conclusion also holds if the displayed absolute CLCD parameter
is decreased.
\end{lemma}
The proof is given in Appendix~\ref{app:product box estimates}.

For the residual box we apply this lemma with $D=d_Y$, $N_*=N$,
$B_*=B_Y$, and the retained supports $e\notin\mathcal D$. Fix $0<\gamma_0<\min\{1,a\sigma_* /(2B_Y)\}$.
This choice precedes $\theta_0$. The normalization $r_*$ will be
fixed in Proposition~\ref{prop:residual rank inversion}.
For the full support $O$ we will apply the same lemma with
$D=n$, $N_*=m$ and the union of the prescribed signed sets.

We use unbiased coordinate rounding in the sense of
\cite{Livshyts-jam}. For fixed $x,B$, independently round each
coordinate of $x$ to adjacent multiples of $h_*$, preserving its
expectation. Then
\begin{align}
 \|R_{h_*}(x)-x\|_\infty\le h_*,
 \qquad
 \mathbb E_R\|B(R_{h_*}(x)-x)\|_2^2
 \le\frac{h_*^2}{4}\|B\|_{\mathrm{HS}}^2.\label{eq:rounding estimate}
\end{align}
All matching realizations satisfy
\begin{align}
 \|H_{\mathfrak c}\|_{\mathrm{HS}}^2\le4p_Xd_Y,\qquad
 \|M_{\mathfrak c}\|_{\mathrm{HS}}\le2N,\qquad
 \|\widehat M_{\mathfrak c}\|_{\mathrm{HS}}\le2m.\label{eq:auxiliary deterministic bounds}
\end{align}
Indeed, each row of $H_{\mathfrak c}$ meets at most two path
edges, and the extension entries are bounded by one. Hence
\begin{align}
 \sup_{z\in\mathcal B}\|\widehat M_{\mathfrak c}z\|_2
 \le2B m^{3/2}M.\label{eq:hybrid deterministic image bound}
\end{align}

\begin{lemma}\label{lem:Stiefel image rounding}
Let $1\le k\le D$ and $0<\xi<1/4$. There is a deterministic set
$\mathcal W_{D,k}(\xi)\subset\mathbb R^{D\times k}$ of size at most
$(C_{\ref{lem:Stiefel image rounding}}/\xi)^{Dk}$ such that for
every $V^TV=I_k$ and every matrix $B_0$ with $D$ columns, some
$V^\sharp\in\mathcal W_{D,k}(\xi)$ satisfies
\begin{align*}
 \|V^\sharp-V\|_{\mathrm{op}} \le\xi, \ 
 \|V^\sharp-V\|_{\mathrm{HS}} \le\xi\sqrt k, \ 
 \|B_0(V^\sharp-V)\|_{\mathrm{HS}} \le\xi\sqrt{k/D}\,\|B_0\|_{\mathrm{HS}}.
\end{align*}
Every net point may be required to have operator norm at most two
and squared Hilbert--Schmidt norm at least $k/2$.
\end{lemma}
\begin{proof}
Round entries independently at mesh $c\xi/\sqrt D$. By
\eqref{eq:rounding estimate} and Markov's inequality, the two
Hilbert--Schmidt bounds fail with probability at most $2c^2$.
For the centered error matrix $\mathsf E$, Hoeffding's inequality
\cite{Hoeffding-jasa} gives
$\mathbb P_R(|u^T\mathsf Ev|>t)\le
2e^{-t^2D/(2c^2\xi^2)}$ for fixed unit vectors $u,v$.
A union over $1/4$-nets of the two spheres makes the operator-norm
failure probability less than $1/4$ for small absolute $c$.
Thus some rounding satisfies all three bounds. The lattice matrices
in the Hilbert--Schmidt ball of radius $(1+\xi)\sqrt k$ number
at most $(C/\xi)^{Dk}$ by the lattice-ball volume bound.
Retain only those satisfying the two stated norm restrictions;
the successful rounding remains.
\end{proof}

\begin{lemma}\label{lem:reverse tensorization}
Let $\zeta_1,\ldots,\zeta_s$ be independent seeds, set
$\mathcal F_i^+=\sigma(\zeta_{i+1},\ldots,\zeta_s)$, and fix
$I\subseteq[s]$. For $i\in I$, let $Y_i\in\mathbb R^{r_i}$
be measurable with respect to $\sigma(\zeta_i,\ldots,\zeta_s)$,
where $r_i\ge1$ is deterministic. Suppose, almost surely,
\begin{align*}
 \sup_z\mathbb P(\|Y_i-z\|_2\le\beta\sqrt{r_i}
                  \mid\mathcal F_i^+)\le p_i
\end{align*}
for deterministic $\beta,p_i>0$.
For every $\eta>0$ there is
$a_{\eta,\ref{lem:reverse tensorization}}>0$ such that
\begin{align*}
 \mathbb P\left(\sum_{i\in I}\|Y_i\|_2^2
 \le a_{\eta,\ref{lem:reverse tensorization}}\beta^2
                        \sum_{i\in I}r_i\right)
 \le e^{\eta\sum_{i\in I}r_i}\prod_{i\in I}p_i.
\end{align*}
The assertion includes $I=\varnothing$.
\end{lemma}
\begin{proof}
Cover a ball of radius $(j+1)\beta\sqrt r$ by
$(C(j+1))^r$ balls of radius $\beta\sqrt r$. Gaussian shell
summation then gives, for sufficiently large $\lambda=\lambda(\eta)$,
\begin{align*}
 \mathbb E\left(e^{-\lambda\|Y_i\|_2^2/\beta^2}
                         \mid\mathcal F_i^+\right)
 \le p_i\left[1+\sum_{j\ge1}(C(j+1))^{r_i}
                                  e^{-\lambda j^2r_i}\right]
 \le p_i e^{\eta r_i/2}.
\end{align*}
Integrate the seeds in increasing index order. Each later factor
is measurable with respect to the current future sigma-field,
including seeds at indices outside $I$. The exponential moment
of the sum is therefore at most
$e^{(\eta/2)\sum r_i}\prod p_i$. Markov's inequality with
$a_{\eta,\ref{lem:reverse tensorization}}=\eta/(2\lambda)$ finishes
the proof.
\end{proof}

\subsection{Two copies of the residual operator}

The centered-energy certificate supplies soft directions on the local
supports. For their tensorization it matters that each row output
contains one current seed and one future seed; no current matching is
exposed before the local small-ball estimate.

\begin{proposition}\label{prop:residual rank inversion}
Fix the density and signed support geometry above,
the preceding choice of $\gamma_0$, and $C_h>0$.
Let $\mathbf p$ denote the fixed data in
Theorem~\ref{thm:hybrid box inversion}. For sufficiently small fixed $\theta_0$ there are positive constants
$r_{\ref{prop:residual rank inversion}}
=r_{\ref{prop:residual rank inversion}}(\mathbf p,\gamma_0)$,
$s_{\ref{prop:residual rank inversion}}
=s_{\ref{prop:residual rank inversion}}(\mathbf p,\gamma_0)$,
$R_{\ref{prop:residual rank inversion}}
=R_{\ref{prop:residual rank inversion}}(\mathbf p,\gamma_0)$ and
$c_{\ref{prop:residual rank inversion}}^{(i)}
=c_{\ref{prop:residual rank inversion}}^{(i)}(\mathbf p,\gamma_0)$,
$1\le i\le4$,
and nonnegative integers
$k_{\ref{prop:residual rank inversion}}=k_{\ref{prop:residual rank inversion}}(\mathbf p,\gamma_0)$ and
$C_{\ref{prop:residual rank inversion}}=C_{\ref{prop:residual rank inversion}}(\mathbf p,\gamma_0)$, all independent of
$\mu\in(0,1)$, with the following property.
Let $U\in\mathbb R^Y$ satisfy the signed bounds of Definition~\ref{def:signed coefficient box} and
$\|U\|_2\le B_Y\sqrt{d_Y}M$, and set $\widehat U=r_{\ref{prop:residual rank inversion}}U/\sqrt N$
and $\alpha=c_{\ref{prop:residual rank inversion}}^{(2)}\mu^2$.
Suppose
\begin{align*}
 \operatorname{CLCD}_{\mu|T_e|,\gamma_0}^{T_e}
           (\widehat U_{T_e})>16
          \qquad(e\notin\mathcal D).
\end{align*}
Let $H_1,H_2$ be independent copies of $H_{\mathfrak c}$, and
write $\mathsf H=[H_1,H_2]$. Here $\mathbb P_\omega$ integrates two
independent copies $\boldsymbol\omega^{(1)},\boldsymbol\omega^{(2)}$
of the residual seed in Definition~\ref{def:auxiliary switching operator},
with the certificate, $U$, and the centers fixed.
For sufficiently large $N$ depending
on $\mu$, for
\begin{align*}
 M\ge M_0,\qquad
 \log M\le {c_{\ref{prop:residual rank inversion}}^{(3)}}\alpha N,
 \qquad 0\le k\le\lfloor\alpha d_Y\rfloor,
\end{align*}
and every fixed $a_1,a_2\in\mathbb R^X$, one has
\begin{align}
 &\mathbb P_{\omega}\left\{
   \|H_iU-a_i\|_2\le C_hN\ (i=1,2),\quad
   \sigma_{2d_Y-k+1}(\mathsf H)
           \le s_{\ref{prop:residual rank inversion}}\sqrt N
                    \right\}\notag\\
 &\hspace{1em}\le
   \left(\frac{R_{\ref{prop:residual rank inversion}}}{M}\right)
                  ^{2p_X-C_{\ref{prop:residual rank inversion}}}
   \exp\left[-{c_{\ref{prop:residual rank inversion}}^{(1)}}N
                (k-k_{\ref{prop:residual rank inversion}})_+\right].\label{eq:residual rank inversion}
\end{align}
The singular values are ordered as in Section~\ref{Sec:notation},
and the singular value condition is omitted when $k=0$.
At $k=\lfloor\alpha d_Y\rfloor$, the exponential factor is at most
$\exp(-c_{\ref{prop:residual rank inversion}}^{(4)}\alpha Nd_Y)$.
\end{proposition}

\begin{proof}
We first establish local dimension bounds for any deterministic
$V^\sharp\in\mathbb R^{2d_Y\times k}$ with
$\|V^\sharp\|_{\mathrm{op}}\le2$ and
$\|V^\sharp\|_{\mathrm{HS}}^2\ge k/2$.
When $k=0$, all soft matrices have no columns.
Let $\Pi_0$ project onto the constants on $Y_0,Y_1$ in
both copies, a space of dimension at most four, and set
$R_e=\operatorname{diag}(P_{T_e,0},P_{T_e,0})
(V^\sharp)_{T_e\sqcup T_e}$.
Centered energy coverage gives
\begin{align*}
 \sum_e\|R_e\|_{\mathrm{HS}}^2
 \ge c p_X\|(I-\Pi_0)V^\sharp\|_{\mathrm{HS}}^2
 \ge c p_X(k/2-16).
\end{align*}
Deleting the at most $J_*$ exceptional indices costs at most $4J_*k$.
Thus, for fixed $c_*,C_*>0$,
\begin{align*}
 \sum_{e\notin\mathcal D}\|R_e\|_{\mathrm{HS}}^2
 \ge c_*Nk-C_*N.
\end{align*}
Choose a sufficiently small $\tau>0$. Let $k_e$ count the singular
values of $R_e$ at least $\sqrt\tau$, and let $J_e$ be an isometry
onto the corresponding right singular space. Since
$\|R_e\|_{\mathrm{HS}}^2\le\tau k+4k_e$, there are fixed $c'>0$ and
$k_0\ge1$ such that
\begin{align}
 \sum_{e\notin\mathcal D}k_e
 &\ge c'N(k-k_0)_+,\notag\\
 D_0:=\sum_{e\notin\mathcal D}(k_e+2)
 &\ge c'N\bigl((k-k_0)_++1\bigr),
 \qquad N(k+2)\le C_0D_0.\label{eq:rank sum local dimensions}
\end{align}
For the last two bounds use the $p_X-1-|\mathcal D|\ge N/3$
retained indices. These constants can be chosen before $\theta_0$.
The matrix $F_e=R_eJ_e$ is centered on both copies of $T_e$, has
operator norm at most two, and satisfies
$\|F_e\|_{\mathrm{HS}}^2\ge\tau k_e$.

Fix the constants $b_0,c''>0$ in
Theorem~\ref{thm:local augmented small ball} for
$\kappa=\tau$, $\kappa_y=1$ and $\gamma=\gamma_0$.
Choose $0<\eta<c''/2$ and then $\xi,s,r>0$, with $\xi<1/4$,
so small that \eqref{eq:rank sum local dimensions} gives
\begin{align*}
 C\bigl((s^2+\xi^2)k+r^2C_h^2\bigr)N
 \le a_{\eta,\ref{lem:reverse tensorization}}b_0^2D_0.
\end{align*}
Here $C$ is a sufficiently large absolute constant for the image
estimate below. Put
$C'=2\log(C_{\ref{lem:Stiefel image rounding}}/\xi)$ and
$c_1=(c''-\eta)c'$, where $c'$ is from
\eqref{eq:rank sum local dimensions}.
Choose $\theta_0$ so small that $2C'\theta_0<c_1/4$, and then fix
$R_0\ge\max\{1,(2a\sigma_*r\sqrt{\theta_0})^{-1}\}$.
Finally choose the range constants in the proposition so that
$k\le\alpha d_Y$ and the bound on $\log M$ ensure the dimension
and scale hypotheses of the local theorem on every retained support.
All these choices precede $\mu$.

Fix $U$, the centers and the certificate. For $k\ge1$, the singular-value
event gives $V\in\mathbb R^{2d_Y\times k}$ with $V^TV=I_k$ and
$\|\mathsf HV\|_{\mathrm{HS}}\le s\sqrt{Nk}$.
By Lemma~\ref{lem:Stiefel image rounding} and
\eqref{eq:auxiliary deterministic bounds}, some point $V^\sharp$ of a
deterministic net of size $e^{C'd_Yk}$ satisfies
\begin{align*}
 \|V^\sharp\|_{\mathrm{op}}\le2,\qquad
 \|V^\sharp\|_{\mathrm{HS}}^2\ge k/2,\qquad
 \|\mathsf H(V^\sharp-V)\|_{\mathrm{HS}}\le C\xi\sqrt{Nk}.
\end{align*}
For $k=0$ omit the net and soft columns. With
$\widehat U=rU/\sqrt N$, set
\begin{align*}
 \mathsf A=\left[V^\sharp,\binom{\widehat U}{0},\binom0{\widehat U}\right],
 \qquad Z_0=\left[0,ra_1/\sqrt N,ra_2/\sqrt N\right].
\end{align*}
The event to be bounded implies, for one net point,
\begin{align}
 \|\mathsf H\mathsf A-Z_0\|_{\mathrm{HS}}^2
 \le C\bigl((s^2+\xi^2)k+r^2C_h^2\bigr)N.\label{eq:rank global image error}
\end{align}
Fix that net point for the local argument.

Write the path as $e_j=(x_j,x_{j+1})$ and let
$\omega_j^{(i)}$ be its walk in copy $i$. Set
$\omega_0^{(i)}=\omega_{p_X}^{(i)}=0$.
The row of $H_i$ indexed by $x_{j+1}$ is
$(\omega_{j+1}^{(i)}-\omega_j^{(i)})^T$.
For a retained index put
\begin{align*}
 Y_j=(\mathsf H\mathsf A-Z_0)_{x_{j+1},\cdot}
                          \operatorname{diag}(J_{e_j},I_2).
\end{align*}
Condition on the complete seeds at indices $j+1,\ldots,p_X-1$
in both copies, whose sigma-field is denoted by $\mathcal F_j^+$.
The two current walks remain independent and
\begin{align*}
 Y_j^T=-\mathcal A_{T_{e_j}}(F_{e_j},\widehat U_{T_{e_j}})^T
                  (\omega_j^{(1)},\omega_j^{(2)})+d_j,
\end{align*}
where $d_j$ is $\mathcal F_j^+$-measurable.
All coefficient matrices and $J_e$ were fixed before the walks.
Centering their restrictions changes no inner product with a walk.

The signed intersections and \eqref{eq:local hard norm from support}
give, by the choice of $R_0$,
$\|P_{T_e,0}\widehat U\|_2
 \ge2a\sigma_*r\sqrt{\theta_0}M\ge M/R_0$.
Every retained support has size at least $2\sigma_*d_Y$.
The range constants therefore ensure
$k_e\le c\mu^2|T_e|$ and $R_0/M\ge e^{-c\mu^2|T_e|}$.
Apply \eqref{eq:local augmented small ball} at $t=R_0/M$;
its CLCD assumption is the arithmetic hypothesis of the proposition.
For each retained index,
\begin{align*}
 \sup_z\mathbb P_\omega
   (\|Y_j-z\|_2\le b_0\sqrt{k_{e_j}+2}\mid\mathcal F_j^+)
 \le(CR_0/M)^2e^{-c''k_{e_j}}.
\end{align*}

The choice of $\xi,s,r$ makes \eqref{eq:rank global image error}
at most $a_{\eta,\ref{lem:reverse tensorization}}b_0^2D_0$.
The $Y_j$ use distinct rows and contractions on the right, so their
squared norms sum to at most this image budget.
Lemma~\ref{lem:reverse tensorization} bounds the probability for
the fixed net point by
\begin{align*}
 e^{\eta D_0}
 \prod_{e\notin\mathcal D}(CR_0/M)^2e^{-c''k_e}
 \le(R_1/M)^{2(p_X-1-|\mathcal D|)}
                          e^{-c_1N(k-k_0)_+}.
\end{align*}
The future filtration retains even the seeds at omitted indices,
so it introduces no dependence between current walks.

For $k\ge2k_0$, the chosen $\theta_0$ absorbs the net cost:
$C'd_Yk-c_1N(k-k_0)\le-c_1Nk/4$.
For bounded $k<2k_0$ the cost is absorbed into a fixed exponential base.
Since $|\mathcal D|\le J_*$, this proves the proposition
with dimension loss $2+2J_*$, after enlarging the base and $k_0$.
The bases and dimension losses are independent of $\mu$; only
$\alpha$, the scale range and the lower dimension depend on it.
For large $N$ depending on $\mu$,
$\lfloor\alpha d_Y\rfloor\ge2k_0$, which gives the terminal bound.
\end{proof}

For a product box $\mathcal B_X\times\mathcal B_Y$ inherited from
Definition~\ref{def:signed coefficient box}, define
\begin{align}
 \mathcal A_U(\mu)
   &:=\left\{u\in\mathcal B_Y:
       \operatorname{CLCD}_{\mu|T_e|,\gamma_0}^{T_e}
       \left(\frac{r_{\ref{prop:residual rank inversion}}}{\sqrt N}
                                                     u_{T_e}\right)>16
                       \text{ for every }e\notin\mathcal D\right\},
       \notag\\
 \mathcal A_0(\mu)&:=
                         \mathcal B_X\times\mathcal A_U(\mu).\label{eq:residual arithmetic good set}
\end{align}
The constants here are those of
Proposition~\ref{prop:residual rank inversion} at the radius
coefficient being used. These sets are deterministic subsets of
the coefficient box.

\begin{proposition}\label{prop:residual fixed center second moment}
Fix $C_h>0$ and the geometry of
Proposition~\ref{prop:residual rank inversion}.
Let $Z_W=(V,U)$ be uniform on
$\mathcal B_X\times\mathcal B_Y$, independently of the auxiliary
operator. There are fixed constants
$R_{\ref{prop:residual fixed center second moment}}=R_{\ref{prop:residual fixed center second moment}}(\mathbf p,\gamma_0)>1$ and
$C_{\ref{prop:residual fixed center second moment}}=C_{\ref{prop:residual fixed center second moment}}(\mathbf p,\gamma_0)\ge0$ such that,
for every $\mu\in(0,1)$, there are
${c_{\ref{prop:residual fixed center second moment}}=c_{\ref{prop:residual fixed center second moment}}(\mathbf p,\gamma_0,\mu)}>0$ and
$N_{\mu,\ref{prop:residual fixed center second moment}}=N_{\ref{prop:residual fixed center second moment}}(\mathbf p,\gamma_0,\mu)$ for which
\begin{align*}
 \sup_{v\in\mathbb R^W}
   \mathbb E_Z\left[
      \mathbf 1_{\mathcal A_0(\mu)}(Z_W)
      \mathbb P_{\omega}\{
          \|M_{\mathfrak c}Z_W-v\|_2\le C_hN\}^2
                                     \right] \le
       \left(\frac{R_{\ref{prop:residual fixed center second moment}}}
                        {M}\right)
                    ^{2N-C_{\ref{prop:residual fixed center second moment}}}
\end{align*}
whenever $N\ge N_{\mu,\ref{prop:residual fixed center second moment}}$,
$M\ge M_0$ and
$\log M\le {c_{\ref{prop:residual fixed center second moment}}(\mu)}N$.
The expectation is with respect to the uniform product measure
$\mathbb P_Z$ on the fixed coefficient box.
\end{proposition}

\begin{proof}
Fix $v=(a,b)\in\mathbb R^X\oplus\mathbb R^Y$. By Fubini the square
is represented by independent copies $H_1,H_2$, with necessary constraints
\begin{align*}
 \|H_iU-a\|_2\le C_hN,\qquad
 \|H_i^TV-b\|_2\le C_hN\quad(i=1,2).
\end{align*}
Let $K=\lfloor\alpha d_Y\rfloor$, and stratify $\mathsf H=[H_1,H_2]$
by the number $k$ of singular values at most
$s_{\ref{prop:residual rank inversion}}\sqrt N$, with a final stratum
$k\ge K$. For $k<K$, fix $U\in\mathcal A_U(\mu)$ and $H_1,H_2$ in
that stratum. Their $r_k=2d_Y-k\ge d_Y$ large singular directions give
isometries $L_k,R_k$ and a diagonal $\Sigma_k$ such that the bottom
constraints imply
\begin{align*}
 \|L_k^TV-\Sigma_k^{-1}R_k^T(b,b)\|_2
 \le C\sqrt N\le A_*\sqrt{r_k}.
\end{align*}
The constant $A_*$ is fixed independently of $\mu$.
The row coefficients $V$ retain their product law under this conditioning.
Since $\mathcal L(V_i,u)\le C(u+1)/M$, the projection
small-ball theorem \cite[Corollary~1.4]{RV-imrn} gives
$(B_1/M)^{r_k}=(B_1/M)^{2d_Y-k}$, with $B_1$ fixed independently
of $\mu$.
Combine it with \eqref{eq:residual rank inversion} for the
top constraints, and average $U$ under its original marginal. The
contribution is at most
\begin{align*}
 B_2^{2N}M^{-2N+C_{\ref{prop:residual rank inversion}}+k}
 e^{-c_{\ref{prop:residual rank inversion}}^{(1)}N
                         (k-k_{\ref{prop:residual rank inversion}})_+}.
\end{align*}
For $k=0$ use the proposition without its singular-value condition.
If $\log M\le c_{\ref{prop:residual rank inversion}}^{(1)}N/2$,
the sum of the $k$-dependent factors is bounded by $CM^{k_*}$,
where $k_*=k_{\ref{prop:residual rank inversion}}$.

On the final stratum drop the bottom constraints. The terminal top
bound is
\begin{align*}
 (B_3/M)^{2p_X-C_{\ref{prop:residual rank inversion}}}
 e^{-c_{\ref{prop:residual rank inversion}}^{(4)}\alpha Nd_Y}.
\end{align*}
For $\log M\le c_{\ref{prop:residual rank inversion}}^{(4)}\alpha N/4$,
its exponential saving recovers $M^{-2d_Y}$.
Both contributions have the required form with, for example,
$C_{\ref{prop:residual fixed center second moment}}
 =C_{\ref{prop:residual rank inversion}}+k_*+2$.
Only the allowed scale constant and lower dimension depend on $\mu$.
The estimates are uniform in the initially fixed $v$, so its supremum
may be taken after the coefficient expectation.
\end{proof}

\subsection{Extension columns and the full count}

The translation caused by the extension walks depends on the
coefficients. We remove it by Gaussian smoothing before averaging
coefficients; the residual second moment is only a fixed-center bound.

\begin{proof}[Proof of Theorem~\ref{thm:hybrid box inversion}]
Use Proposition~\ref{prop:residual fixed center second moment}
at radius coefficient one, and fix all its constants, including
the normalization in \eqref{eq:residual arithmetic good set}.

We first construct the arithmetic good set, without changing
the product law. The union of the positive prescribed sets has
at least $2\varepsilon d_Y$ coordinates, as does the union of
the negative sets. For large $n$, each therefore has at least
$\sigma_O n$ coordinates, where
$\sigma_O:=\varepsilon\theta_0/2$.
Fix
\begin{align*}
 0<\gamma_O<\min\{1,a\sigma_O/(8B)\},\qquad a_O:=a\sigma_O.
\end{align*}
For $z=(x,y)\in\mathbb R^O\oplus\mathbb R^{U_m}$ in the box,
the signed difference bound implies
\begin{align*}
 \left\|P_{O,0}\frac{x}{\sqrt m}\right\|_2
        \ge2a\sigma_O\sqrt{n/m}\,M\ge a_OM.
\end{align*}
Given $L$, apply Lemma~\ref{lem:box CLCD exceptions} with $A=L^4$
first to the residual supports and then to the singleton family
$\{O\}$. In the second application take $D=n$, $N_*=m$,
$r_*=1$ and $B_*=2B$.
Choose a common $\mu=\mu(L)>0$ no larger than either parameter
supplied by that lemma. Define
\begin{align*}
 \mathcal A_O(\mu)
   &:=\left\{z=(x,y)\in\mathcal B:
         \operatorname{CLCD}_{\mu n,\gamma_O}^{O}(x/\sqrt m)>16\right\},\\
 \mathcal A(\mu)
   &:=\{z\in\mathcal B:z_W\in\mathcal A_0(\mu)\}
                                      \cap\mathcal A_O(\mu).
\end{align*}
After restricting the scale range and increasing the lower
dimension as permitted in the theorem,
the arithmetic lemma gives
\begin{align}
 \mathbb P_Z\{Z\notin\mathcal A(\mu)\}
             \le L^{-8N}+L^{-8m}\le L^{-3m}.\label{eq:hybrid arithmetic exceptional probability}
\end{align}
Also require the scale ranges of
Proposition~\ref{prop:residual fixed center second moment} and
Corollary~\ref{cor:extension scalar Gaussian estimate}, the latter
with $r_0=1$, $a_0=a_O$, $c_g=1$ and $\gamma=\gamma_O$.
Since $N,n,m$ are comparable, all these requirements hold in the stated range with a sufficiently small
${c_{\ref{thm:hybrid box inversion}}(L)}$ and large
$m_{L,\ref{thm:hybrid box inversion}}$.

Fix a deterministic center $v=(v_O,v_{U_m})\in\mathbb R^m$, and set
\begin{align*}
 g_v(z)&:=\mathbb P_{\omega}
                   \{\|\widehat M_{\mathfrak c}z-v\|_2\le3C_hm\},\\
 G_v(z)&:=\mathbb E_{\omega}
                   e^{-\|\widehat M_{\mathfrak c}z-v\|_2^2/m},\qquad
 G_0(z_W):=
       \mathbb E_{\omega}
                   e^{-\|M_{\mathfrak c}z_W\|_2^2/m}.
\end{align*}
Gaussian smoothing gives $g_v(z)\le e^{9C_h^2m}G_v(z)$.
For this fixed coefficient vector, expose the extension walks
$\omega_j$, $j\in U_m$, leaving the residual auxiliary operator
independent. By \eqref{eq:hybrid auxiliary action}, the top block
is a translate of $\iota_WM_{\mathfrak c}z_W$ by
$\sum_{j\in U_m}y_j\omega_j-v_O$. The bottom coordinates are
$\langle\omega_j,x\rangle-v_j$.
The characteristic function of the residual vector is nonnegative
by Proposition~\ref{prop:hybrid Fourier domination}.
The Gaussian comparison following \eqref{eq:Gaussian Fourier identity},
applied in $\mathbb R^O$ with Gaussian scale $\sqrt{\pi m}$, therefore gives
the pointwise inequality
\begin{align}
 G_v(z)\le G_0(z_W)
           \prod_{j\in U_m}\mathbb E_{\omega}
             e^{-(\langle\omega_j,x\rangle-v_j)^2/m}.\label{eq:hybrid Gaussian separation}
\end{align}
The extension walks are independent; this is the only
factorization used after removing the top translation.
On $\mathcal A_O(\mu)$, the centered norm bound and
\eqref{eq:extension scalar Gaussian estimate} bound every factor
by $C/M$. Consequently,
\begin{align*}
 \mathbf 1_{\mathcal A(\mu)}(z)G_v(z)^2
       \le(C/M)^{2d_0}
         \mathbf 1_{\mathcal A_0(\mu)}(z_W)
                                   G_0(z_W)^2.
\end{align*}

We bound the remaining Gaussian functional by the residual
second moment. Since $m$ is comparable to $N$, decomposition
into shells of width $N$ gives
\begin{align*}
 G_0(z_W)
  \le\sum_{j\ge0}e^{-c j^2N}
       \mathbb P_{\omega}
             \{\|M_{\mathfrak c}z_W\|_2\le(j+1)N\}.
\end{align*}
For each $j$, cover the ball of radius $(j+1)N$ in
$\mathbb R^W$ by at most $(C(j+1))^N$ balls of radius $N$.
Their centers are deterministic. Apply Minkowski's inequality
in $L^2$ of the finite measure
$\mathbf 1_{\mathcal A_0(\mu)}\,d\mathbb P_Z$,
and then Proposition~\ref{prop:residual fixed center second moment}
separately at each fixed center. This yields
\begin{align*}
 &\left(\mathbb E_Z
       \mathbf 1_{\mathcal A_0(\mu)}(Z_W)
                                      G_0(Z_W)^2\right)^{1/2}
       \\
 &\quad\le
  \left(\frac{R_{\ref{prop:residual fixed center second moment}}}{M}
         \right)^{N-C_{\ref{prop:residual fixed center second moment}}/2}
         \sum_{j\ge0}e^{-c j^2N}(C(j+1))^N
       \\
 &\quad\le (R/M)^{N-C_{\ref{prop:residual fixed center second moment}}/2}.
\end{align*}
The sum is at most a fixed constant to the power $N$.
Combining this with the extension-factor bound from
\eqref{eq:hybrid Gaussian separation}, and using
$N+d_0=m-2$, gives, uniformly over the initially fixed center $v$,
\begin{align}
 \mathbb E_Z\mathbf 1_{\mathcal A(\mu)}(Z)g_v(Z)^2
       \le (R/M)^{2m-C_0},\qquad
 C_0:=C_{\ref{prop:residual fixed center second moment}}+4.\label{eq:hybrid fixed center second moment}
\end{align}
The unused marginals on $\{1,2\}\cup U_m$ integrate to one;
no product law on $\mathcal A(\mu)$ is asserted.

It remains to account for the projection. Let
\begin{align*}
 f(z):=\mathbb P_{\omega}
                      \{\|Q\widehat M_{\mathfrak c}z\|_2\le C_hm\}.
\end{align*}
By \eqref{eq:hybrid deterministic image bound}, the component
in $\ker Q$ belongs to a deterministic ball of radius
$Cm^{3/2}M$. Cover this ball inside $\ker Q$ by a deterministic
$C_hm$-net $\mathcal V_Q$ of size
\begin{align*}
 K_Q:=|\mathcal V_Q|\le(1+C\sqrt m\,M)^{\ell_Q}.
\end{align*}
If $\|Q\widehat M_{\mathfrak c}z\|_2\le C_hm$, then
$\|\widehat M_{\mathfrak c}z-v\|_2\le3C_hm$ for some $v\in\mathcal V_Q$.
Hence $f(z)\le\sum_{v\in\mathcal V_Q}g_v(z)$, and
Cauchy--Schwarz followed by
\eqref{eq:hybrid fixed center second moment} gives
\begin{align*}
 \mathbb E_Z\mathbf 1_{\mathcal A(\mu)}(Z)f(Z)^2
                             \le K_Q^2(R/M)^{2m-C_0}.
\end{align*}

If $M<L$, the event in \eqref{eq:hybrid box inversion} is empty
since its probability cutoff exceeds one. Suppose $M\ge L$.
Markov's inequality now yields
\begin{align*}
 \mathbb P_Z\{Z\in\mathcal A(\mu),\ f(Z)\ge(L/M)^{m-\ell_Q}\}
& \le K_Q^2(R/M)^{2m-C_0}(M/L)^{2(m-\ell_Q)}\\
 & \le m^{\ell_0}M^{C_0+2\ell_0}(R'/L)^{2m}.
\end{align*}
In the last line we used
$K_Q^2\le C^{2\ell_0}m^{\ell_0}M^{2\ell_Q}$ and $L\le M$.
Add \eqref{eq:hybrid arithmetic exceptional probability} and
enlarge the fixed exponential base. This proves
\eqref{eq:hybrid box inversion}, for example with
${C_{\ref{thm:hybrid box inversion}}^{(1)}}=\ell_0$ and
${C_{\ref{thm:hybrid box inversion}}^{(2)}}=C_0+2\ell_0$.
The powers and exponential base were fixed before $\mu(L)$;
only the scale range and lower dimension depend on it.
The argument holds on each good fibre, proving conditional uniformity.
\end{proof}
\section{Threshold control}\label{Sec:threshold control}
Rounding controls matrix images, but need not preserve the auxiliary
threshold. In particular, $x\approx y$ and $A_mx\approx A_my$ do not give
$T_L(y)\lesssim T_L(x)$. A union bound confined to the original threshold
class would therefore be unjustified.

We use a stopping chain on one fixed fibre. A rounded point with usable
threshold is terminal; otherwise it is reclassified at a strictly larger
dyadic scale. The certificate, auxiliary law and projection remain fixed.
Only the coefficient point and its scale change. Each terminal point
belongs to a net chosen before the true orientations, which permits one
union bound over terminal nets.

\subsection{The geometric domain and the main estimate}

Retain the dimensions from Section~\ref{subsec:inversion setting}
and the block proportion \eqref{eq:inversion block proportion}. Fix $0<a_g<b_g$,
$0<\sigma_g<1/4$, $B_g\ge1$ and an integer $\ell_0\ge1$.
Choose the support-coverage parameter of
Definition~\ref{def:good switching certificate} so that
$0<\varepsilon<\sigma_g/8$. The block proportion $\theta_0$
is taken sufficiently small for
Theorem~\ref{thm:hybrid box inversion} with the fixed, enlarged
box bounds used below. Its smallness requirement is the residual
requirement in Proposition~\ref{prop:residual rank inversion} at
radius coefficient one: the proof of
Theorem~\ref{thm:hybrid box inversion} uses that residual estimate
before Gaussian smoothing, so changing the fixed raw radius in
the box theorem changes its exponential base but not this
requirement. These choices precede $L$.

\begin{definition}
Let $\mathcal K_0$ be the subset of $H_m\cap S^{m-1}$ consisting
of vectors $v$ such that, for each $b\in\{0,1\}$,
\begin{align}
 \left|\{i\in Y_b:a_g/\sqrt m\le v_i\le b_g/\sqrt m\}\right|
       &\ge\sigma_g|Y_b|,\notag\\
 \left|\{i\in Y_b:-b_g/\sqrt m\le v_i\le-a_g/\sqrt m\}\right|
       &\ge\sigma_g|Y_b|,
 \qquad \|v_Y\|_2\le B_g\sqrt{d_Y/m}.\label{eq:threshold signed geometry}
\end{align}
For $\rho\ge0$, define the relative neighborhood
\begin{align*}
 \mathcal U_\rho
       :=\mathcal K_0+\rho(B_2^m\cap H_m).
\end{align*}
\end{definition}

\begin{lemma}
There is a fixed $\delta_0>0$ such that every vector
$u\in\mathbb R^m$ with
$\operatorname{dist}(u,\mathcal K_0)\le\delta_0$ satisfies
\begin{align*}
 \|u\|_2\le2,\qquad
 \|u_Y\|_2\le(B_g+1)\sqrt{d_Y/m},
\end{align*}
and, in each half $Y_b$, at least $(\sigma_g/2)|Y_b|$ coordinates
belong to each of
\begin{align*}
 [a_g/(2\sqrt m),\,2b_g/\sqrt m],\qquad
 [-2b_g/\sqrt m,\,-a_g/(2\sqrt m)].
\end{align*}
In particular, either signed set has size at least
$2\varepsilon|Y_b|$. The constant $\delta_0$ is independent of $L$.
\end{lemma}

\begin{proof}
If $\|u-v\|_2\le\delta_0$ with $v\in\mathcal K_0$, at most
$4\delta_0^2m/a_g^2$ coordinates change by more than $a_g/(2\sqrt m)$.
Since $|Y_b|\ge c\theta_0m$, sufficiently small fixed $\delta_0$
leaves half of each prescribed signed set. The norm bounds follow by
the triangle inequality after also requiring
$\delta_0\le\min(1,\sqrt{d_Y/m})$ uniformly in the allowed dimensions.
\end{proof}

Let $\mathcal H_{m;X,Y}$ and $\mathcal E_{m;X,Y}$ be the exposure
and its good event from Theorem~\ref{thm:exact hybrid factorization}.
On a fixed good fibre $\mathfrak c$, fix an orthogonal projection $Q$ with
\begin{align*}
 \operatorname{Im}Q\subseteq H_m,\qquad
 \ell_Q:=m-\operatorname{rank}Q\le\ell_0,
 \qquad r:=\operatorname{rank}Q.
\end{align*}
We use $T_{L,\mathfrak c}^Q$ from Definition~\ref{def:auxiliary threshold}, abbreviated
to $T_L$ on this fibre, with the same cap
$s_0=\tau_0/L$ as in
Definition~\ref{def:auxiliary threshold}.
The probability defining $T_L$ always integrates the auxiliary
operator $\widehat M_{\mathfrak c}$. Write $\mathbb P$
for the true conditional law of $A_m$ on the fibre.

\begin{theorem}\label{thm:uniform threshold control}
Fix $\lambda,\kappa,a_g,b_g,\sigma_g,B_g,\ell_0$ and the preceding
choices of $\varepsilon,\theta_0,\delta_0$. Put
$\mathbf q=(\lambda,\kappa,a_g,b_g,\sigma_g,B_g,\ell_0,\varepsilon,\theta_0,\delta_0)$.
There are
$L_{\ref{thm:uniform threshold control}}=L_{\ref{thm:uniform threshold control}}(\mathbf q)\ge2$ and
${c_{\ref{thm:uniform threshold control}}^{(2)}=c_{\ref{thm:uniform threshold control}}^{(2)}(\mathbf q)}>0$ such that, for every
fixed $L\ge L_{\ref{thm:uniform threshold control}}$, there are
${c_{\ref{thm:uniform threshold control}}^{(1)}=c_{\ref{thm:uniform threshold control}}^{(1)}(\mathbf q,L)}>0$ and
$m_{L,\ref{thm:uniform threshold control}}=m_{\ref{thm:uniform threshold control}}(\mathbf q,L)$ with the following
property. Set
\begin{align*}
 \Delta_m:=\exp(-{c_{\ref{thm:uniform threshold control}}^{(1)}(L)}m),
 \qquad \mathcal K:=\mathcal U_{\delta_0/4}.
\end{align*}
For every $m\ge m_{L,\ref{thm:uniform threshold control}}$, every
fixed good hybrid fibre and every projection as above,
\begin{align}
 \mathbb P\left\{
   \exists v\in\mathcal K:
       T_L(v)\ge\Delta_m,\quad
       \|QA_mv\|_2\le T_L(v)\sqrt m\right\}
       \le e^{-{c_{\ref{thm:uniform threshold control}}^{(2)}}m}.\label{eq:uniform threshold control}
\end{align}
The constants depend only on these fixed parameters and, where indicated,
$L$; they are uniform over the fibres and projections.
Equivalently, outside one event of this probability, simultaneously
for every $v\in\mathcal K$,
\begin{align}
 T_L(v)<\max\{\Delta_m,\,\|QA_mv\|_2/\sqrt m\}.\label{eq:threshold versus image}
\end{align}
If $Q$ is $\mathcal H_{m;X,Y}$-measurable, the conditional form is
\begin{align*}
 &\mathbf 1_{\mathcal E_{m;X,Y}}
  \mathbb P\left(
   \exists v\in\mathcal K:
      T_{L,\mathfrak c}^Q(v)\ge\Delta_m,
      \|QA_mv\|_2\le T_{L,\mathfrak c}^Q(v)\sqrt m
          \,\middle|\,\mathcal H_{m;X,Y}\right)\\
 &\qquad\le
   \mathbf 1_{\mathcal E_{m;X,Y}}
                     e^{-{c_{\ref{thm:uniform threshold control}}^{(2)}}m}.
\end{align*}
\end{theorem}

\subsection{Counting and approximating a threshold class}

The auxiliary count first yields a Euclidean cover. A deterministic
image net then returns the coefficients to $H_m$, where the true Fourier
comparison applies. These two approximations serve different purposes.

\begin{lemma}\label{lem:threshold product box cover}
Let $M\ge M_{\ref{lem:threshold product box cover}}$.
The integer vectors $z\in\mathbb Z^m$ satisfying
\begin{align*}
 \|z\|_2\le2M\sqrt m,\qquad
 \|z_Y\|_2\le(B_g+1)M\sqrt{d_Y},
\end{align*}
and having at least $(\sigma_g/2)|Y_b|$ coordinates in each of
$[a_gM/2,2b_gM]$ and $[-2b_gM,-a_gM/2]$ in each half $Y_b$,
are covered by at most
$C_{\ref{lem:threshold product box cover}}^m$ signed coefficient
boxes of scale $M$. Every box has at most
$(C_{\ref{lem:threshold product box cover}}M)^m$ points and
satisfies Definition~\ref{def:signed coefficient box} with fixed
parameters independent of $M,L$ and the certificate. Here
$C_{\ref{lem:threshold product box cover}}=C_{\ref{lem:threshold product box cover}}(\mathbf q)$
and $M_{\ref{lem:threshold product box cover}}=M_{\ref{lem:threshold product box cover}}(\mathbf q)$,
with $\mathbf q$ as in Theorem~\ref{thm:uniform threshold control}.
\end{lemma}

\begin{proof}
Prescribe positive and negative coordinate sets in each half, with
the required cardinalities. On them use the full integer intervals
$[a_gM/2,2b_gM]$ and $[-2b_gM,-a_gM/2]$; divide the other
coordinates into intervals of $\lceil M\rceil$ consecutive integers.
Retain boxes meeting the stated lattice set. Their diameters are
$CM\sqrt m$ globally and $CM\sqrt{d_Y}$ on $Y$, so they satisfy
the required norm bounds with enlarged fixed constants.
The signed proportions suffice because $\varepsilon<\sigma_g/8$.
For fixed signed sets, the indices $j\in\mathbb Z^d$, $d\le m$,
of the remaining intervals satisfy $\|j\|_2\le C\sqrt m$.
There are at most $C_1^m$ such indices by lattice-ball counting.
The at most $3^{d_Y}$ choices of signed sets give the box count;
each factor has between $cM$ and $CM$ points.
\end{proof}
For the approximations below we use the random-rounding estimate
\eqref{eq:rounding estimate}, averaging only the coordinate choices.

\begin{lemma}\label{lem:fixed image net}
Let $E\subseteq\mathbb R^m$ be a fixed subspace of dimension
$d\ge m/2$, let $K\subset E$ be bounded, and let $t>0$.
There is a deterministic finite set $\mathcal N\subset E$ with
\begin{align*}
 |\mathcal N|\le C_{\ref{lem:fixed image net}}^m
                              N(K,tB_2^m),\qquad
 \mathcal N\subset K+3t(B_2^m\cap E),
\end{align*}
such that, for every real matrix $B$ with $m$ columns and every
$x\in K$, some $y=y(B,x)\in\mathcal N$ satisfies
\begin{align*}
 \|B(x-y)\|_2\le\frac{t}{2\sqrt d}\|B|_E\|_{\mathrm{HS}}.
\end{align*}
The net is fixed from $E,K,t$ before $B$ is supplied.
\end{lemma}

\begin{proof}
Cover $K$ by $N(K,tB_2^m)$ balls of radius $t$ with centers in $E$
and within $t$ of $K$. Fix an isometry $J:\mathbb R^d\to E$.
For each center $z$ include all points of
$z+J((t/\sqrt d)\mathbb Z^d)$ within $2t$ of $z$.
There are at most $C^d$ such points by the standard lattice-ball count,
and they lie within $3t$ of $K$.
For $x$ in the original ball, unbiased rounding of $J^T(x-z)$ with mesh
$t/\sqrt d$ stays in the enlarged ball. Some rounding has squared
image error at most its expectation
$t^2\|BJ\|_{\mathrm{HS}}^2/(4d)$.
The set of all allowed roundings was fixed before $B$ and gives the claim.
\end{proof}

\begin{proposition}\label{prop:threshold class nets}
Fix $c_E>0$ and $C'\ge1$ independently of $L$.
There are positive constants
$A_{\ref{prop:threshold class nets}}=A_{\ref{prop:threshold class nets}}(\mathbf q,c_E,C^\prime)$ and
$C_{\ref{prop:threshold class nets}}^{(3)}=C_{\ref{prop:threshold class nets}}^{(3)}(\mathbf q,c_E,C^\prime)$,
and nonnegative integers
$C_{\ref{prop:threshold class nets}}^{(i)}=C_{\ref{prop:threshold class nets}}^{(i)}(\mathbf q,c_E,C^\prime)$
for $i=1,2$, with the following property. For every sufficiently large fixed $L$
there are ${c_{\ref{prop:threshold class nets}}=c_{\ref{prop:threshold class nets}}(\mathbf q,c_E,C^\prime,L)}>0$ and
$m_{L,\ref{prop:threshold class nets}}=m_{\ref{prop:threshold class nets}}(\mathbf q,c_E,C^\prime,L)$ such that, if
\begin{align}
 m\ge m_{L,\ref{prop:threshold class nets}},\qquad
 e^{-{c_{\ref{prop:threshold class nets}}(L)}m}\le t\le s_0,\label{eq:threshold net scale range}
\end{align}
and $K$ is any deterministic subset of
\begin{align*}
 \mathcal U_{3\delta_0/4}\cap\{v:t\le T_L(v)<2t\},
\end{align*}
then
\begin{align}
 N(K,c_EtB_2^m)
 \le m^{{C_{\ref{prop:threshold class nets}}^{(1)}}}
       t^{-{C_{\ref{prop:threshold class nets}}^{(2)}}}
       \left(\frac{A_{\ref{prop:threshold class nets}}}{L^2t}\right)^m.\label{eq:threshold Euclidean entropy}
\end{align}
There is also a deterministic finite set $\mathcal N(K,t)\subset H_m$
with the same cardinality bound, after increasing
$A_{\ref{prop:threshold class nets}}$, such that for every
$B$ with $\|B\|_{\mathrm{HS}}\le C'm$ and every $x\in K$,
some $y\in\mathcal N(K,t)$ satisfies
\begin{align}
 \mathcal N(K,t)\subset K+
 C_{\ref{prop:threshold class nets}}^{(3)}t(B_2^m\cap H_m),\ 
 \|B(x-y)\|_2\le C_{\ref{prop:threshold class nets}}^{(3)}t\sqrt m.\label{eq:threshold image net approximation}
\end{align}
Both the Euclidean cover and the image net are chosen using
only the fixed coefficient class and auxiliary law.
\end{proposition}

\begin{proof}
Fix a sufficiently small $c\le\min(1,c_E)$, independently of $L$,
and round $x\in K$ on the ambient lattice of mesh $ct/\sqrt m$.
Every rounded point $u$ has $\|u-x\|_2\le ct$.
For each auxiliary realization, the deterministic Hilbert--Schmidt bound
and unbiased rounding give
$\mathbb E_R\|Q\widehat M_{\mathfrak c}(u-x)\|_2^2\le Cc^2t^2m$.
Choose $c$ so that an error at most $t\sqrt m$ has probability at least
$3/4$ for every realization.

Since $t\le T_L(x)<2t$, admissible scales tending to $T_L(x)$ in
\eqref{eq:auxiliary threshold} show that
\begin{align*}
 \mathbb P_\omega(\|Q\widehat M_{\mathfrak c}x\|_2\le2t\sqrt m)
 \ge(4Lt)^r.
\end{align*}
This includes $T_L(x)=t$.
Average the successful rounding event first over rounding, then over walks.
Fubini supplies a deterministic rounded point satisfying
\begin{align*}
 \|u-x\|_2\le ct,\qquad
 \mathbb P_\omega(\|Q\widehat M_{\mathfrak c}u\|_2\le3t\sqrt m)
                         \ge\tfrac34(4Lt)^r.
\end{align*}
This transfers probability to $u$ without asserting any upper bound on
$T_L(u)$.

Set $M=(ct)^{-1}$, so $u=Z/(M\sqrt m)$ for $Z\in\mathbb Z^m$.
For large $L$, $cs_0\le\delta_0/4$, and geometric stability places all
these coefficients in the lattice set of
Lemma~\ref{lem:threshold product box cover}.
There are at most $C^m$ covering boxes, each with at most $(CM)^m$ points.
In each box apply Theorem~\ref{thm:hybrid box inversion} with
$L'=2L/c$ and $C_h=3/c$.
The raw radius is $3m/c$, and
$\tfrac34(4Lt)^r\ge(2Lt)^r=(L'/M)^r$.
Writing $a_i=C_{\ref{thm:hybrid box inversion}}^{(i)}$,
the number of retained lattice points is therefore at most
\begin{align*}
 C^m(CM)^m m^{a_1}M^{a_2}(R/L')^{2m}
 \le m^{a_1}t^{-a_2}\left(\frac{A}{L^2t}\right)^m.
\end{align*}
Every $x\in K$ has such an approximation, proving the Euclidean cover.
Choose the permitted lower-scale exponent small enough for the box
inversion range at $L'$, and increase the lower dimension to absorb
fixed factors in $M$.

Apply Lemma~\ref{lem:fixed image net} in $H_m$ at radius $c_Et$.
Projection of the covering centers to $H_m$ preserves the cover.
The new net has only a fixed exponential overhead, lies within $3c_Et$
of $K$, and has image error at most
$c_EtC'm/(2\sqrt{m-1})$.
This proves the remaining assertions after increasing the fixed constants.
The product boxes are used only for the ambient count; the image-net
points used with $A_m$ belong to $H_m$.
\end{proof}

\subsection{The stopping chain}
Use Proposition~\ref{prop:threshold class nets} with $c_E=C'=1$.
It applies to every true realization since
$\|QA_m\|_{\mathrm{HS}}\le\|A_m\|_{\mathrm{HS}}\le m$.
For $s>0$, let $\operatorname{dyad}(s)=2^j$ when $2^j\le s<2^{j+1}$.
The following two observations provide the successor rule and an
upper barrier for the chain.

\begin{lemma}\label{lem:threshold successor}
There are fixed constants
$B_{\ref{lem:threshold successor}}\ge
{C_{\ref{prop:threshold class nets}}^{(3)}}$ and
\begin{align*}
 a_{\ref{lem:threshold successor}}
   >\max\left\{8,\,
       2\sqrt2\left(2+\frac43{C_{\ref{prop:threshold class nets}}^{(3)}}\right)
       \right\}
\end{align*}
with the following property. Set, for each permitted dyadic $t$,
\begin{align*}
 \mathcal K_t
  :=\left(\mathcal K+
       B_{\ref{lem:threshold successor}}t(B_2^m\cap H_m)\right)
           \cap\{v:t\le T_L(v)<2t\},
\end{align*}
where $\mathcal K=\mathcal U_{\delta_0/4}$.
Take $L$ sufficiently large that
\begin{align*}
 \left(B_{\ref{lem:threshold successor}}
             +{C_{\ref{prop:threshold class nets}}^{(3)}}\right)s_0
                              \le\delta_0/4.
\end{align*}
Then $\mathcal K_t$ and its image net $\mathcal N_t$
from Proposition~\ref{prop:threshold class nets} are contained
in $\mathcal U_{\delta_0/2}$.
If $y\in\mathcal N_t$ satisfies
$T_L(y)\ge a_{\ref{lem:threshold successor}}t$, then
\begin{align*}
 t':=\operatorname{dyad}(T_L(y))\ge4t,
 \qquad y\in\mathcal K_{t'}.
\end{align*}
The block, splitter, certificate, projection and auxiliary law
are the same in both classes.
\end{lemma}

\begin{proof}
Choose $B_{\ref{lem:threshold successor}}\ge
\max\{1,{C_{\ref{prop:threshold class nets}}^{(3)}}\}$ and choose the
jump constant as in the statement.
The neighborhood assertion follows from
\eqref{eq:threshold image net approximation} and the choice of $L$.
For the successor, the definition of the dyadic scale gives
$t'\ge T_L(y)/2\ge(a_{\ref{lem:threshold successor}}/2)t\ge4t$.
Also
\begin{align*}
 \operatorname{dist}(y,\mathcal K)
 &\le\left(B_{\ref{lem:threshold successor}}
                +{C_{\ref{prop:threshold class nets}}^{(3)}}\right)t
 \le2B_{\ref{lem:threshold successor}}t
 \le B_{\ref{lem:threshold successor}}t'.
\end{align*}
Together with $t'\le T_L(y)<2t'$, this proves
$y\in\mathcal K_{t'}$.
\end{proof}
\begin{lemma}\label{lem:empty upper threshold range}
Set $t_*:=s_0/(4a_{\ref{lem:threshold successor}})$.
For sufficiently large fixed $L$, and then sufficiently large
$m$ depending on $L$,
\begin{align*}
 \{v\in\mathcal U_{3\delta_0/4}:T_L(v)\ge t_*\}
                                      =\varnothing.
\end{align*}
This assertion includes the endpoint $T_L(v)=s_0$.
\end{lemma}

\begin{proof}
Cover $[t_*,s_0]$ by the dyadic bins that meet it. Their lower
endpoints $t$ satisfy $t\ge t_*/2$, and their number is bounded
in terms of $a_{\ref{lem:threshold successor}}$.
For large $m$ all these scales lie in
\eqref{eq:threshold net scale range}.
For each bin apply \eqref{eq:threshold Euclidean entropy} to
$K=\mathcal U_{3\delta_0/4}\cap\{t\le T_L<2t\}$.
Its covering number is at most
\begin{align*}
 m^{{C_{\ref{prop:threshold class nets}}^{(1)}}}
 \left(\frac{8a_{\ref{lem:threshold successor}}L}{\tau_0}
       \right)^{{C_{\ref{prop:threshold class nets}}^{(2)}}}
 \left(\frac{8a_{\ref{lem:threshold successor}}
                      A_{\ref{prop:threshold class nets}}}
                    {\tau_0L}\right)^m.
\end{align*}
Choose $L$ so that the final base is less than $1/2$.
Since $L$ is now fixed, the bound is less than one for large
$m$. Every nonempty set has covering number at least one.
Thus all the bins are empty, proving the assertion.
\end{proof}

Write $a=a_{\ref{lem:threshold successor}}$.
For each scale, the transition takes the form
\begin{align}
 (x_i,t_i)\ \longmapsto\ y\in\mathcal N_{t_i}
 \quad\begin{cases}
 T_L(y)<at_i:&\text{stop at }y,\\
 T_L(y)\ge at_i:&
 (x_{i+1},t_{i+1})=(y,\operatorname{dyad}(T_L(y))),
 \quad t_{i+1}\ge4t_i.
 \end{cases}\label{eq:stopping chain schematic}
\end{align}
All thresholds in this rule use the same certificate, operator law and
projection. In particular, passing to a successor introduces no new
conditioning of the true matrix.

\begin{proof}[Proof of Theorem~\ref{thm:uniform threshold control}]
Fix a good fibre and a projection $Q$ satisfying the stated hypotheses.
All nets below are constructed on this fibre before the fresh
true variables are sampled. Fix all geometric, rounding and
image-net constants first, then the constants in Lemma~\ref{lem:threshold successor}. Choose $L$
large enough for its neighborhood condition and
Lemma~\ref{lem:empty upper threshold range}.

Write $c_*>0$ for the lower-scale exponent to be chosen, and
let $\Delta_m=e^{-c_*m}$. Consider the finite dyadic family
$\mathscr S_m:=\{2^j:j\in\mathbb Z,\ \Delta_m/2\le2^j\le s_0\}$.
Choose $c_*\le {c_{\ref{prop:threshold class nets}}(L)}/2$ and
increase $m$ so that every scale in $\mathscr S_m$ is in
\eqref{eq:threshold net scale range}.
For each $t\in\mathscr S_m$, fix an image net $\mathcal N_t$
for $\mathcal K_t$ and a deterministic ordering of its points.
For an empty class take the empty net. Define the deterministic
terminal subset
\begin{align*}
 \mathcal N_t^0
       :=\{y\in\mathcal N_t:
                 T_L(y)<a_{\ref{lem:threshold successor}}t\}.
\end{align*}

Suppose that a true realization has a vector $x_0\in\mathcal K$
with
\begin{align*}
 T_L(x_0)\ge\Delta_m,\qquad
 \|QA_mx_0\|_2\le T_L(x_0)\sqrt m.
\end{align*}
Let $B=QA_m$ and $t_0=\operatorname{dyad}(T_L(x_0))$.
Then $t_0\in\mathscr S_m$, $x_0\in\mathcal K_{t_0}$, and
$\|Bx_0\|_2\le2t_0\sqrt m$.
Following \eqref{eq:stopping chain schematic}, given
$x_i\in\mathcal K_{t_i}$ choose the first
$x_{i+1}\in\mathcal N_{t_i}$ satisfying
\begin{align*}
 \|B(x_{i+1}-x_i)\|_2
                \le {C_{\ref{prop:threshold class nets}}^{(3)}}t_i\sqrt m.
\end{align*}
Existence follows from $\|QA_m\|_{\mathrm{HS}}\le m$
and \eqref{eq:threshold image net approximation}.
Stop if $x_{i+1}\in\mathcal N_{t_i}^0$.
Otherwise let $t_{i+1}$ be the dyadic scale containing $T_L(x_{i+1})$.
Lemma~\ref{lem:threshold successor} shows that
$t_{i+1}\ge4t_i$ and $x_{i+1}\in\mathcal K_{t_{i+1}}$.
The successor scale is still in $\mathscr S_m$.

All selected points lie in $\mathcal U_{\delta_0/2}$.
By Lemma~\ref{lem:empty upper threshold range}, their thresholds,
and the threshold of $x_0$, are less than $t_*$.
Hence every scale used by the procedure is less than $t_*$.
Since successive nonterminal scales increase by at least four,
the procedure terminates after finitely many steps.

Suppose it terminates at scale $t_j$ with
$x_{j+1}\in\mathcal N_{t_j}^0$.
The increasing scales satisfy
$\sum_{i=0}^jt_i\le(4/3)t_j$ and $t_0\le t_j$.
Summing the successive image errors and using the initial image bound gives
\begin{align*}
 \|Bx_{j+1}\|_2 \le\left(2t_0+
      {C_{\ref{prop:threshold class nets}}^{(3)}}\sum_{i=0}^jt_i\right)\sqrt m \le\left(2+\frac43{C_{\ref{prop:threshold class nets}}^{(3)}}\right)t_j\sqrt m
   \le a_{\ref{lem:threshold successor}}t_j\sqrt r.
\end{align*}
For the last inequality we used $r\ge m-\ell_0\ge m/2$ for
large $m$ and the choice of $a_{\ref{lem:threshold successor}}$.
It follows pointwise, before taking any true probability, that
the event in \eqref{eq:uniform threshold control} is contained in
\begin{align*}
 \bigcup_{\substack{t\in\mathscr S_m\\t<t_*}}
 \ \bigcup_{y\in\mathcal N_t^0}
       \{\|QA_my\|_2\le
                  a_{\ref{lem:threshold successor}}t\sqrt r\}.
\end{align*}
There is one union over terminal nets and scales; the possible
paths have already been accounted for by this containment.

Fix a point $y$ in this union. It belongs to $H_m$, is fixed on
the conditional fibre, and satisfies
$T_L(y)<a_{\ref{lem:threshold successor}}t$.
Moreover, since $t<t_*$, $a_{\ref{lem:threshold successor}}t<s_0/4<s_0/2$.

Corollary~\ref{cor:conditional hybrid capture} therefore yields
\begin{align*}
 \mathbb P
     \{\|QA_my\|_2\le a_{\ref{lem:threshold successor}}t\sqrt r\}
 \le\left(C_{\ref{thm:Fourier capture}}
                       a_{\ref{lem:threshold successor}}Lt\right)^r.
\end{align*}

Multiplying this probability by
the cardinality bound in Proposition~\ref{prop:threshold class nets}
gives, for each scale, an upper bound of the form
\begin{align*}
 m^{{C_{\ref{prop:threshold class nets}}^{(1)}}}
        t^{-D_*}(R_*/L)^m,
 \qquad
 D_*:={C_{\ref{prop:threshold class nets}}^{(2)}}+\ell_0,
\end{align*}
where $R_*>1$ is fixed before $L$.
Here the rank loss $m-r=\ell_Q\le\ell_0$ leaves only the
additional power $t^{-\ell_0}$.

Fix a positive entropy margin $\chi$, for example $\chi=1$.
Increase $L$ so that $R_*/L\le e^{-8\chi}$, and choose
\begin{align*}
 0<c_*\le
 \min\left\{1,\,
       \tfrac12{c_{\ref{prop:threshold class nets}}(L)},\,
       \frac{\chi}{D_*+1}\right\}.
\end{align*}
For every $t\in\mathscr S_m$,
$t^{-D_*}\le2^{D_*}e^{D_*c_*m}\le2^{D_*}e^{\chi m}$.
There are at most $Cm$ scales in $\mathscr S_m$.
Thus the sum of the preceding entropy--probability bound
over these scales is at most $e^{-\chi m}$ for sufficiently
large $m$ depending on the fixed $L$.
Taking ${c_{\ref{thm:uniform threshold control}}^{(1)}(L)}=c_*$ and,
if necessary, a smaller
${c_{\ref{thm:uniform threshold control}}^{(2)}}>0$ proves
\eqref{eq:uniform threshold control}.

All nets and terminal subsets were fixed before the true orientations.
There are finitely many exposure atoms, so these choices are measurable.
Multiplying the uniform fibrewise estimate by
$\mathbf 1_{\mathcal E_{m;X,Y}}$ gives the conditional assertion.
The complement of the event in
\eqref{eq:uniform threshold control} is exactly the simultaneous
inequality \eqref{eq:threshold versus image}, completing the proof.
\end{proof}

\begin{corollary}\label{cor:approximate kernel threshold}
Under the hypotheses of Theorem~\ref{thm:uniform threshold control},
outside a conditional event of probability at most
$e^{-{c_{\ref{thm:uniform threshold control}}^{(2)}}m}$, the implication
\begin{align*}
 v\in\mathcal U_{\delta_0/4},\quad
 t\ge\Delta_m,\quad \|QA_mv\|_2\le t\sqrt m
       \quad\Longrightarrow\quad T_{L,\mathfrak c}^Q(v)<t
\end{align*}
holds simultaneously for all $v$ and $t$.
In particular, simultaneously for every
$v\in\mathcal K_0$ with $QA_mv=0$,
\begin{align*}
 T_{L,\mathfrak c}^Q(v)<
             e^{-{c_{\ref{thm:uniform threshold control}}^{(1)}(L)}m}.
\end{align*}
\end{corollary}

\begin{proof}
On the event where \eqref{eq:threshold versus image} holds,
$\max\{\Delta_m,\|QA_mv\|_2/\sqrt m\}\le t$ under the stated
hypotheses. Take $t=\Delta_m$ for the exact-kernel assertion.
\end{proof}
\section{Proof of the main theorem}\label{Sec:main proof}

We prove an exponential bound for $q_m(\beta)$ and substitute it into
Theorem~\ref{thm:reduction}. After excluding compressible solutions,
we place each remaining solution in one of finitely many signed domains
and classify its restriction to the fixed support realizing $\rho_*$.

For a fixed $b\in H_m\cap\mathbf 1^\perp$, let $Q_b$ be the
orthogonal projection onto $(\operatorname{span}\{h,b\})^\perp$.
Write
\begin{align*}
 \ell_b:=m-\operatorname{rank}Q_b
 =\begin{cases}1,&b=0,\\2,&b\ne0,\end{cases}
 \qquad r_b:=m-\ell_b.
\end{align*}
Thus $\operatorname{Im}Q_b\subset H_m$, and, for $v\in H_m$,
\begin{align*}
 Q_bA_mv=0
 \quad\Longleftrightarrow\quad
 P_hA_mv\in\mathbb Rb.
\end{align*}
The projection is fixed before the matrix is sampled.

\subsection{Compressible vectors}

We establish a fixed-vector estimate and use a deterministic image
net with arbitrarily small exponential overhead. This overhead is
needed because the fixed-vector estimate supplies only a fixed
exponential gain. The estimate is uniform in the target line and
in the number of extension columns.

We use a fixed small-ball gap for dense slices: if $q$ is uniform
on $\binom{[p]}a$ with $\tau p\le a\le(1-\tau)p$, then
\begin{align*}
 \mathcal L(\langle q,u\rangle,c\|P_{[p],0}u\|_2)\le1-c'
 \qquad(P_{[p],0}u\ne0),
\end{align*}
where $c,c'>0$ depend only on $\tau$.
For centered unit $u$, the second moment is
$a(p-a)/(p(p-1))$ and the fourth moment is bounded by an absolute
constant. Paley--Zygmund applied to the squared difference of two
independent copies proves the estimate. The same argument applies
to Rademacher sums, whose fourth moment is at most three times
the squared variance.

\begin{lemma}\label{lem:final independent columns}
Fix $\tau\in(0,1/2)$. Let $B$ be a $p\times q$ matrix with
independent columns, each uniform on a slice of weight in
$[\tau p,(1-\tau)p]$, where $1\le q\le2p$.
There are positive constants
${c_{\ref{lem:final independent columns}}^{(1)}=c_{\ref{lem:final independent columns}}^{(1)}(\tau)}$ and
${c_{\ref{lem:final independent columns}}^{(2)}=c_{\ref{lem:final independent columns}}^{(2)}(\tau)}$, depending only on
$\tau$, such that, for every fixed $y\in\mathbb R^q\setminus\{0\}$,
\begin{align*}
 \mathcal L\left(P_{[p],0}By,
       {c_{\ref{lem:final independent columns}}^{(1)}}\sqrt p\,\|y\|_2\right)
 \le e^{-{c_{\ref{lem:final independent columns}}^{(2)}}p}.
\end{align*}
\end{lemma}

\begin{proof}
Normalize $\|y\|_2=1$. Pair all but at most one of the $p$ row coordinates.
In each column expose the pair sums and the unpaired entry.
All mixed-pair orientations are then independent uniform signs.
Permutation concentration and a union over $q\le2p$ columns show that
each column has at least $c_\tau p$ mixed pairs outside probability
$e^{-cp}$. This event is determined by the exposed sums.

The row difference on a pair $e$ is a Rademacher sum with variance
$V_e=\sum_{j:\,j\text{ mixed on }e}y_j^2$.
On the good event $\sum_eV_e\ge c_\tau p$ and $V_e\le1$, so at least
$cp$ pairs have $V_e\ge c$.
Their differences are conditionally independent. For a Rademacher
sum of variance $V_e$, the fourth moment is at most $3V_e^2$;
Paley--Zygmund applied to an independent difference gives a fixed
small-ball gap at radius $c\sqrt{V_e}$. Tensorization over these pairs
gives concentration at most $e^{-cp}$ at radius $c'\sqrt p$.
The disjoint-difference map has norm at most $\sqrt2$ and annihilates
constants, so the same estimate, with adjusted radius, holds for
$P_{[p],0}By$. Average over the exposed sums.
\end{proof}

\begin{lemma}\label{lem:final extended fixed vector}
There are positive constants
${c_{\ref{lem:final extended fixed vector}}^{(1)}=c_{\ref{lem:final extended fixed vector}}^{(1)}(\lambda)}$ and
${c_{\ref{lem:final extended fixed vector}}^{(2)}=c_{\ref{lem:final extended fixed vector}}^{(2)}(\lambda)}$, depending only
on $\lambda$, such that, for $n\le m\le2n-2$, every fixed
$z\in H_m$ with $1/2\le\|z\|_2\le3/2$, and every fixed
$b\in H_m\cap\mathbf 1^\perp$,
\begin{align}
 \mathbb P\left\{\operatorname{dist}(P_hA_mz,\mathbb Rb)
                   \le {c_{\ref{lem:final extended fixed vector}}^{(1)}}
                                                       \sqrt m\right\}
 \le e^{-{c_{\ref{lem:final extended fixed vector}}^{(2)}}m}.\label{eq:final extended fixed vector}
\end{align}
\end{lemma}

\begin{proof}
We first record the original-block estimate. For deterministic
$z\in\mathbf1^\perp\setminus\{0\}$ and $w,b_0\in\mathbb R^n$,
\begin{align*}
 \mathbb P\{\operatorname{dist}(P_hA_nz,w+\mathbb Rb_0)
                         \le c\sqrt n\,\|z\|_2\}\le e^{-c'n},
\end{align*}
where $c,c'>0$ depend only on $\lambda$.
Normalize $\|z\|_2=1$. Averaging over $s=\lfloor n/2\rfloor$ subsets
$Y\subset[n]$ gives
$\mathbb E_Y\|P_{Y,0}z\|_2^2=(s-1)/(n-1)$.
Fix $Y$ with this energy at least $1/3$, and put $X=[n]\setminus Y$.
The split depends only on the fixed vector.
Proposition~\ref{thm:fixed block quasirandomness}, applied to arbitrary
balanced splitters on $X,Y$, gives dense cross-margins and balanced
linear-size disagreement supports for every row pair, except with
probability $e^{-cn}$.

Independently order $X$ uniformly and pair successive positions.
For its disjoint pairs $e=(x,x')$, write
$S_e=\{i\in Y:A_{xi}\ne A_{x'i}\}$.
The coordinate-pair argument of Lemma~\ref{lem:pair coverage and energy}
applies on the whole of $Y$: dense column margins and permutation
bounded differences give $cn$ pairs disagreeing at every two fixed
columns, simultaneously, outside probability $e^{-c'n}$.
Consequently
\begin{align*}
 \sum_e\|P_{S_e,0}z\|_2^2
 \ge cn\|P_{Y,0}z\|_2^2\ge c'n.
\end{align*}

Expose the induced blocks, the unused row and the pair sums on $Y$.
The reconstruction in Proposition~\ref{prop:exact switching fibre},
now using disjoint pairs on all of $Y$, gives independent orientation
slices: their choices preserve the exposed row and column sums and
reconstruct exactly one compatible block. Supports, weights and the preceding energy bound
are measurable in this exposure.
Discard the at most two pairs meeting $\{1,2\}$.
Each energy summand is at most one, so at least $cn$ remaining pairs
have $\|P_{S_e,0}z\|_2\ge c'$.
Their row differences are fixed translations of
$2\langle q_e,z_{S_e}\rangle$ for independent dense slices $q_e$.
The slice gap above and tensorization bound the concentration
of this difference vector by $e^{-cn}$ at radius $c'\sqrt n$.

Its difference map $D$ satisfies $Dh=0$, $\|D\|\le\sqrt2$ and
$\|DA_nz\|_2\le\sqrt2n$.
The relevant portion of $Dw+\mathbb RDb_0$ has length at most $Cn$,
so it has a $c'\sqrt n$-net of at most $C\sqrt n$ points, fixed on the
exposure fibre. A union bound gives proximity to this line with probability
$e^{-c''n}$. Average over the exposure and independent ordering and
include their exceptional probabilities.

We now write $z=(a\mathbf1_O+u,y)$ on $O\sqcup U_m$, where $u\perp\mathbf1_O$;
then $u\in H_n$. Fix a small $\eta=\eta(\lambda)>0$.
If $\|u\|_2\ge\eta$, expose the complete extension, independently of
$A_n$. The original coordinates of $P_hA_mz$ are
$P_hA_nu+ad\mathbf1_O+P_hB_my$.
The original-block estimate applies to this fixed affine
target and proves the required bound.

If $\|y\|_2\ge\eta$, condition on $(A_n,C_m)$.
The columns of $B_m$ are independent prescribed slices, with dense
weights on $\mathcal D_m$.
Proximity of $P_hA_mz$ to $\mathbb Rb$ implies, after lifting $P_h$
and then centering on $O$, proximity of $P_{O,0}B_my$ to
\begin{align*}
 -A_nu+\operatorname{span}\{P_{O,0}b_O,h\}.
\end{align*}
This fixed affine space has dimension at most two.
Since $\|P_{O,0}B_my\|_2\le Cm$, its relevant part has a
$c\sqrt m$-net of polynomial size. Apply
Lemma~\ref{lem:final independent columns} at each center and add
$\mathbb P(\mathcal D_m^c)\le e^{-cn}$.

Otherwise $\|u\|_2,\|y\|_2<\eta$ and
$|a|\sqrt n\ge\sqrt{1/4-2\eta^2}$.
Every coordinate of $d^{(m)}=A_m\mathbf1$ lies in $[c_\lambda m,m]$,
using the dense birth degrees for the new vertices.
For sufficiently small $\eta$,
\begin{align*}
 |\langle d^{(m)},z\rangle|
 \ge |a|\sum_{i\in O}d_i^{(m)}
       -\|d_O^{(m)}\|_2\|u\|_2-\|d_{U_m}^{(m)}\|_2\|y\|_2
 \ge c_\lambda' m^{3/2}.
\end{align*}
But the asserted proximity, $b\perp\mathbf1$ and $P_h\mathbf1=\mathbf1$
would give $|\langle d^{(m)},z\rangle|\le Cm$, a contradiction for
large $m$. These cases prove the claim.
\end{proof}

\begin{proposition}\label{prop:final compressible}
There are $\delta_c,\rho_c\in(0,1/4)$ and
$c_{\ref{prop:final compressible}}=c_{\ref{prop:final compressible}}(\lambda)>0$, depending only on $\lambda$,
such that, uniformly for $n\le m\le2n-2$ and fixed
$b\in H_m\cap\mathbf 1^\perp$,
\begin{align*}
 \mathbb P\{\exists v\in H_m\cap
       \mathrm{Comp}_{\delta_c,\rho_c}(m):Q_bA_mv=0\}
 \le e^{-c_{\ref{prop:final compressible}}m}.
\end{align*}
\end{proposition}

\begin{proof}
We use the deterministic rounding theorem
\cite[Theorem~3.8]{JSS-imrn}, a specialization of
\cite[Theorem~4]{Livshyts-jam}. In an $s$-dimensional Euclidean
space, for $K$ in its unit sphere,
$0<\alpha<1/2$, $0<\xi<\alpha/10$, and $s\ge\alpha^{-2}$,
it gives a net $\mathcal N\subset K+(4\xi/\alpha)B_2^s$ with
\begin{align*}
 |\mathcal N| \le N(K,\xi B_2^s)
       \exp\left(C\alpha^{0.08}\log(1/\alpha)s\right),
\end{align*}
and
\begin{align*}
 \forall B\ \forall v\in K\ \exists z\in\mathcal N:
 \|B(v-z)\|_2\le
             \frac{2\xi}{\alpha\sqrt s}\|B\|_{\mathrm{HS}}.
\end{align*}
Apply this in $H_m$ through a fixed isometry, with
$s=m-1$ and
$K=H_m\cap\mathrm{Comp}_{\delta_c,\rho_c}(m)$.
First choose $\alpha$ so that its exponential overhead is
at most $\exp({c_{\ref{lem:final extended fixed vector}}^{(2)}}m/8)$.
Next choose $\xi/\alpha$ small enough that $4\xi/\alpha<1/2$
and the image error for $B=Q_bA_m|_{H_m}$ is at most
${c_{\ref{lem:final extended fixed vector}}^{(1)}}\sqrt m/2$.
This uses the deterministic bound $\|Q_bA_m\|_{\mathrm{HS}}\le m$.

Choose $\rho_c<\xi/8$. Cover the ball of radius two on
each coordinate subspace of dimension $\lfloor\delta_c m\rfloor$
by balls of radius $\xi/4$. Approximation by sparse vectors
and then projection of centers to $H_m$ give
\begin{align*}
 N(K,\xi B_2^m)
 \le \binom m{\lfloor\delta_c m\rfloor}
                  (C/\xi)^{\delta_c m}.
\end{align*}
Choose $\delta_c$ small enough that this is at most
$\exp({c_{\ref{lem:final extended fixed vector}}^{(2)}}m/8)$.
All choices depend only on $\lambda$.
Every net point belongs to $H_m$ and has norm between $1/2$
and $3/2$. If $Q_bA_mv=0$ for some $v\in K$, a net point
$z$ satisfies the event in
\eqref{eq:final extended fixed vector}, since
$\|Q_bA_mz\|_2=\operatorname{dist}(P_hA_mz,\mathbb Rb)$.
A union bound over the deterministic net proves the proposition.
\end{proof}

\subsection{Incompressible vector}\label{subsec:final signed geometry}
The equation $Q_bA_mv=0$ implies orthogonality to the positive degree
vector $A_m\mathbf1$. It therefore supplies signed coordinates even
though the extension is not regular.

\begin{lemma}\label{lem:final weighted spread}
Fix $\delta,\rho,\tau\in(0,1)$.
There are $a_{\ref{lem:final weighted spread}}>0$ and
$0<a_g<b_g$, depending only on these parameters, such that,
if $v\in\mathrm{Incomp}_{\delta,\rho}(m)$,
$w\in[\tau,1]^m$, and $\langle w,v\rangle=0$, then
\begin{align*}
 |\{i:a_g/\sqrt m\le v_i\le b_g/\sqrt m\}|,\quad
 |\{i:-b_g/\sqrt m\le v_i\le-a_g/\sqrt m\}|
 \ge a_{\ref{lem:final weighted spread}}m.
\end{align*}
\end{lemma}

\begin{proof}
The set of coordinates with $|v_i|>1/\sqrt{\delta m}$
has size at most $\delta m$. Incompressibility implies
that the squared norm outside this set exceeds $\rho^2$.
Coordinates with $|v_i|<\rho/\sqrt{2m}$ contribute at most
$\rho^2/2$. Hence at least $\delta\rho^2m/2$ coordinates
have absolute value between $\rho/\sqrt{2m}$ and
$1/\sqrt{\delta m}$.
One sign occurs on at least half of these coordinates.
Its weighted mass is at least $c_{\delta,\rho,\tau}\sqrt m$.
The opposite sign has the same total weighted mass by
$\langle w,v\rangle=0$.

For either sign, coordinates of absolute value less than
$a_g/\sqrt m$ have total weighted mass at most $a_g\sqrt m$.
Coordinates of absolute value greater than $b_g/\sqrt m$
have total weighted mass at most $\sqrt m/b_g$, by
Cauchy--Schwarz and $\|v\|_2=1$.
Choose $a_g$ sufficiently small and $b_g$ sufficiently large
that their sum is less than half the preceding positive
mass bound, and so that the first intermediate interval is
contained in $[a_g/\sqrt m,b_g/\sqrt m]$.
Each remaining coordinate contributes at most $b_g/\sqrt m$.
Both signs must therefore occur at least a positive constant
times $m$ in the required interval.
\end{proof}
Apply Lemma~\ref{lem:final weighted spread} with
$(\delta,\rho)=(\delta_c,\rho_c)$ and with a fixed lower
bound $\tau=\tau_\lambda$ for $d_i^{(m)}/m$.
Choose and fix
\begin{align}
 0<\kappa<
 \min\{1/4,\ a_{\ref{lem:final weighted spread}}/8\}.\label{eq:final short extension choice}
\end{align}
This is the value of $\kappa$ used in Theorem~\ref{thm:reduction}
from now on. All constants below may depend on this fixed choice.

We now choose the deterministic block families needed to place these
signed coordinates. This construction involves sets only and precedes
every graph exposure.

\begin{lemma}\label{lem:deterministic block family}
Fix $a\in(0,1/2)$ and an integer $K\ge4$. There are positive constants
$N_{\ref{lem:deterministic block family}}=N_{\ref{lem:deterministic block family}}(a,K)$ and
$C_{\ref{lem:deterministic block family}}=C_{\ref{lem:deterministic block family}}(a,K)$
such that, for every set $V$ of size $N\ge N_{\ref{lem:deterministic block family}}$,
there is a deterministic family $\mathfrak P_{a,K}$ of at most
$C_{\ref{lem:deterministic block family}}$ partitions
$V=Y_1\sqcup\cdots\sqcup Y_K$, with cell sizes differing by at most
one, such that every disjoint pair $P_+,P_-\subseteq V$ with
$|P_+|,|P_-|\ge aN$ satisfies
\begin{align*}
 |P_+\cap Y_j|,\ |P_-\cap Y_j|\ge\frac{aN}{2K},
 \qquad 1\le j\le K,
\end{align*}
for some partition in this family.

For each cell and each cell complement $Z$ arising in this construction,
there is also a deterministic family $\mathfrak S_Z$ of balanced
bipartitions $Z=Z_0\sqcup Z_1$, of cardinality at most
$C_{\ref{lem:deterministic block family}}$, such that every disjoint
$P_+,P_-\subseteq Z$ with $|P_+|,|P_-|\ge a|Z|/2$ satisfies
\begin{align*}
 |P_\pm\cap Z_b|\ge\frac{a|Z|}{8},\qquad b=0,1,
\end{align*}
for some member of $\mathfrak S_Z$.
\end{lemma}

\begin{proof}
For any fixed eligible pair $(P_+,P_-)$, a uniformly balanced $K$-partition
fails one of the intersection bounds with probability $Ce^{-cN}$ by
hypergeometric concentration \cite[Section~6]{Hoeffding-jasa}.
Take $L=L(a,K)$ independent partitions with $cL>\log3+1$.
The probability that some pair is missed by all of them is at most
$3^N(Ce^{-cN})^L<1$ for large $N$.
Fix a successful family. The identical argument for balanced bipartitions
of each cell or complement gives the splitter families. There are only
a fixed number of these sets.
\end{proof}

\begin{lemma}\label{lem:final geometric domains}
There is a fixed finite family $\mathfrak F$ of block and
splitter choices on $W$, with parameters satisfying
Section~\ref{Sec:threshold control} and $\ell_0=2$, such that,
for $n\le m\le(1+\kappa)n$, every
$v\in H_m\cap\mathrm{Incomp}_{\delta_c,\rho_c}(m)$ satisfying
$Q_bA_mv=0$ belongs to the geometric domain $\mathcal K_0$
of at least one choice in $\mathfrak F$.
The family and all its geometric parameters are deterministic.
\end{lemma}

\begin{proof}
The projected equation gives
\begin{align*}
 \langle A_m\mathbf 1,v\rangle
 =\langle\mathbf 1,P_hA_mv\rangle=0,
\end{align*}
because $b\perp\mathbf 1$. Lemma~\ref{lem:final weighted spread}
therefore applies. Removing $U_m\cup\{1,2\}$ leaves positive
and negative sets of size at least $a|W|$, for a fixed
$a>0$; decrease $a$ if necessary.

Choose a sufficiently large fixed integer $K$ and use
Lemma~\ref{lem:deterministic block family} on $W$, with its
density parameter $a/2$. For the signed sets just obtained,
some partition has at least $a|W|/(4K)$ coordinates of
each sign in every cell. For all large $n$ this is at least
$(a/8)|Y_j|$. Since
$\sum_j\|v_{Y_j}\|_2^2\le1$, one cell satisfies
$\|v_{Y_j}\|_2^2\le1/K$.
Use that cell as $Y$ and its complement as $X$.
To split $Y$, include the balanced splitter family furnished
by the same lemma with density parameter $a/4$; its
hypothesis holds by the preceding count. Both halves then
contain at least $(a/32)|Y|$ coordinates of each sign.
Thus \eqref{eq:threshold signed geometry} holds with a fixed
$\sigma_g>0$, the constants $a_g,b_g$ already chosen, and,
for example, $B_g=2$.

Choose $K$ large enough that the resulting proportions
satisfy \eqref{eq:inversion block proportion} with a
$\theta_0$ allowed by Section~\ref{Sec:threshold control}.
The signed proportions and $B_g$ can be fixed before $K$,
so this choice respects the required parameter order.
Fix the support-coverage parameter below $\sigma_g/8$.
Take $\mathfrak F$ to contain all cells and all these
splitters, including prescribed balanced splitters on the
complements when needed for Section~\ref{Sec:quasirandomness}.
Its size is bounded independently of $m,n$.
Every selection above is from this deterministic family;
no block is chosen as additional conditioning data.
\end{proof}

\subsection{Classification on a fixed support}

The support used to define $\rho_*$ need not be a switching
support or the distinguished central support. Throughout this
subsection, $T\subseteq W$ is fixed, $s=|T|$, and $s$ is even.
For a unit vector $v$, let
\begin{align*}
 \bar v_T:=s^{-1}\sum_{i\in T}v_i,\qquad
 a_T(v):=\|P_{T,0}v\|_2,\qquad
 u_T(v):=\frac{P_{T,0}v}{a_T(v)}\quad\text{if }a_T(v)>0.
\end{align*}
Here $P_{T,0}v$ is regarded as a vector in $\mathbb R^T$.
Thus
$v_T=\bar v_T\mathbf 1_T+a_T(v)u_T(v)$ when $a_T(v)>0$.
The mean and the centered length will both be retained in
the covers.

\begin{lemma}\label{lem:final concentration CLCD}
Fix $\delta_s,\rho_s\in(0,1/4)$. There are
$c_{\ref{lem:final concentration CLCD}}
=c_{\ref{lem:final concentration CLCD}}(\delta_s,\rho_s)>0$ and
$\gamma_{\ref{lem:final concentration CLCD}}
=\gamma_{\ref{lem:final concentration CLCD}}(\delta_s,\rho_s)\in(0,1)$
such that, for $0<\gamma\le\gamma_{\ref{lem:final concentration CLCD}}$
and every unit centered
$u\notin\mathrm{Cons}_{\delta_s,\rho_s}(T)$,
\begin{align*}
 \operatorname{CLCD}_{\mu s,\gamma}^{T}(u)
 \ge c_{\ref{lem:final concentration CLCD}}\sqrt s
 \qquad(\mu>0).
\end{align*}
There is also $C_{\ref{lem:final concentration CLCD}}
=C_{\ref{lem:final concentration CLCD}}(\gamma)>0$ such that
if $\rho_T(v)>\beta\ge
2C_{\ref{lem:final concentration CLCD}}(\gamma)e^{-2\mu^2s}$,
then either $a_T(v)=0$ or
\begin{align*}
 \operatorname{CLCD}_{\mu s,\gamma}^{T}(u_T(v))
 \le2C_{\ref{lem:final concentration CLCD}}(\gamma)/\beta.
\end{align*}
\end{lemma}
\begin{proof}
For the lower bound, suppose $\theta$ satisfies the CLCD inequalities
and write $e_{ij}=\theta(u_i-u_j)-k_{ij}$ with antisymmetric integers
$k_{ij}$. Since $\sum_{i<j}e_{ij}^2<\gamma^2\theta^2s$,
some root $j$ with $|u_j|\le2/\sqrt s$ satisfies
$\sum_{i\ne j}e_{ij}^2\le4\gamma^2\theta^2$.
Outside at most $(M^{-2}+4\gamma^2/\eta^2)s$ indices,
$|u_i|\le M/\sqrt s$ and $|e_{ij}|\le\eta\theta/\sqrt s$.
Choose $M$ large, $\eta<\rho_s/2$, then $\gamma$ small so that
this fraction is below $\delta_s/2$. If $\theta<c\sqrt s$ for
small fixed $c$, both $|\theta(u_i-u_j)|$ and $|e_{ij}|$ are
below $1/4$ there. Thus $k_{ij}=0$ and $u$ is almost constant,
a contradiction.

For the upper bound, balanced signs annihilate the coordinate mean
and nonzero scaling preserves point masses, so
$\rho_T(v)=\rho_T(u_T(v))$ when $a_T(v)>0$.
Tran's slice small-ball theorem \cite[Theorem~1.5]{Tran-arxiv},
with $\|D_Tu_T(v)\|_2=\sqrt s$, gives at radius zero
\begin{align*}
 \rho_T(v)\le
 \frac{C_{\ref{lem:final concentration CLCD}}(\gamma)}
 {\operatorname{CLCD}_{\mu s,\gamma}^{T}(u_T(v))}
 +C_{\ref{lem:final concentration CLCD}}(\gamma)e^{-2\mu^2s}.
\end{align*}
The last term is at most $\beta/2$, proving the claim.
\end{proof}

\begin{lemma}\label{lem:final support covers}
Let $K\subseteq H_m\cap S^{m-1}$ and let $T\subseteq W$
have size $s$. Fix $\delta_s,\rho_s\in(0,1/4)$ and let
$k_s=\lfloor\delta_s s\rfloor$.
Define
\begin{align}
 \mathcal C_T(K):=
 \{v\in K:a_T(v)=0\text{ or }
       [a_T(v)>0,\ u_T(v)\in\mathrm{Cons}_{\delta_s,\rho_s}(T)]\}.\label{eq:final constant class}
\end{align}
There is an absolute constant ${C_{\ref{lem:final support covers}}^{(1)}}>1$
such that
\begin{align}
 N\left(\mathcal C_T(K),
                  {C_{\ref{lem:final support covers}}^{(1)}}\rho_s B_2^m\right)
 \le \binom s{k_s}
          ({C_{\ref{lem:final support covers}}^{(1)}}/\rho_s)^{m-s+k_s+1}.\label{eq:final constant cover}
\end{align}

Fix also $c''>0$, $\gamma\in(0,1)$, and
$0<\mu<\min\{c''/10,1/10\}$. For $D\ge c''\sqrt s$ let $\varepsilon_D:=\mu\sqrt s/D$,
\begin{align*}
 \mathcal A_{T,D}(K):=\{v\in K:\ a_T(v)>0,
             u_T(v)\notin\mathrm{Cons}_{\delta_s,\rho_s}(T),
       D\le\operatorname{CLCD}_{\mu s,\gamma}^{T}(u_T(v))<2D\}.
\end{align*}
For a constant ${C_{\ref{lem:final support covers}}^{(2)}(c'')}$ independent
of $\mu$,
\begin{align}
 N\left(\mathcal A_{T,D}(K),
             {C_{\ref{lem:final support covers}}^{(1)}}\varepsilon_D B_2^m\right)\le
 \mu^{-3}D^3
 \left(\frac{{C_{\ref{lem:final support covers}}^{(2)}(c'')}}{\varepsilon_D}
                                                    \right)^{m+2}
            ({C_{\ref{lem:final support covers}}^{(2)}(c'')}\mu)^s.\label{eq:final arithmetic cover}
\end{align}
The approximation constant can be chosen absolute in both assertions.
\end{lemma}

\begin{proof}
For the constant class choose $G\subset T$ with $|G|=s-k_s$ on which
the centered direction is within $\rho_s/\sqrt s$ of a constant.
Then $v$ is within $\rho_s$ of the space
$V_G=\{x:x_i=x_j\text{ for }i,j\in G\}$, of dimension
$m-s+k_s+1$. Projection onto $V_G$, a volumetric net and the
$\binom{s}{k_s}$ possible sets $G$ prove \eqref{eq:final constant cover}.

For an arithmetic class, write $u=u_T(v)$ and choose a defining
frequency $\theta\in[D,2D)$ with
$\sum_{i<j}(\theta(u_i-u_j)-k_{ij})^2<\mu^2s^2$.
Some root has star energy below $2\mu^2s$; setting $k_i=k_{ij}$ and
$k_j=0$ gives
\begin{align*}
 \left\|u-\frac{P_{T,0}k}{\theta}\right\|_2
 \le\sqrt2\varepsilon_D,\qquad \|P_{T,0}k\|_2\le3D.
\end{align*}
The lattice $\Lambda_T=P_{T,0}\mathbb Z^T$ has covolume $s^{-1/2}$
and covering radius at most $\sqrt s/2$: delete one projected coordinate
vector for a basis, and use coordinate rounding for the radius.
Disjoint Voronoi cells in the ball of radius $3D+\sqrt s/2$ give
\begin{align*}
 |\Lambda_T\cap3DB_2^T|\le\sqrt s(CD/\sqrt s)^{s-1}.
\end{align*}
A mesh of spacing $\mu\sqrt s$ for $\theta\in[D,2D]$ costs
$C/\varepsilon_D$ and changes the approximant by $C\varepsilon_D$.
Retaining approximants of norm at most two therefore covers the
centered directions with at most
$\mu^{-3}D^3(CD/\sqrt s)^s$ points.

Discretize $a_T(v)\in[0,1]$ and $\sqrt s\bar v_T\in[-1,1]$
at mesh $\varepsilon_D$, and cover the complementary unit ball
at that scale. The total count is
\begin{align*}
 \mu^{-3}D^3(CD/\sqrt s)^s(C/\varepsilon_D)^{m-s+2}
 \le\mu^{-3}D^3(C/\varepsilon_D)^{m+2}(C\mu)^s.
\end{align*}
Reconstruction has error $C\varepsilon_D$ and never divides by
$a_T(v)$, so arbitrarily small centered lengths are included.
The two scalar meshes and the complement account for all dimensions.
Lemma~\ref{lem:fixed image net} will place the final image-net points in $H_m$.
\end{proof}

\subsection{Excluding structured solutions}

Fix the family $\mathfrak F$ of Lemma~\ref{lem:final geometric domains}
and $L\ge L_{\ref{thm:uniform threshold control}}$ with $\ell_0=2$.
Keep $s_0=\tau_0/L$, $\delta_0$ and $\Delta_m$ from Section~7.
Every eligible support in \eqref{eq:rho-star} has size at least
$\vartheta m$, where
$\vartheta=c_{\ref{lem:Sigma fields}}^{(2)}/(1+\kappa)$.
The support parameters below are chosen after $L$.

\begin{lemma}\label{lem:final entry capture}
There are absolute constants
${C_{\ref{lem:final entry capture}}^{(1)}}>1$ and
${C_{\ref{lem:final entry capture}}^{(2)}}>1$ with the following
property. Fix one good hybrid fibre $\mathfrak c$, its domain
$\mathcal K_0$, and the projection $Q_b$.
There is an event $\mathcal B_{\mathfrak c}$ under the true conditional
law satisfying
\begin{align*}
 \mathbb P(\mathcal B_{\mathfrak c})
 \le e^{-{c_{\ref{thm:uniform threshold control}}^{(2)}}m}
\end{align*}
such that the following holds for every $K\subseteq\mathcal K_0$
and every $\varepsilon>0$ with
\begin{align}
 {C_{\ref{lem:final entry capture}}^{(1)}}\varepsilon\le\delta_0/4,
 \qquad
 \Delta_m\le {C_{\ref{lem:final entry capture}}^{(1)}}\varepsilon
                                                    \le s_0/2.\label{eq:final entry scale range}
\end{align}
If
$N(K,{C_{\ref{lem:final support covers}}^{(1)}}\varepsilon B_2^m)\le M_K$,
then
\begin{align*}
 \mathbb P
 \left(\{\exists v\in K:Q_bA_mv=0\}\setminus\mathcal B_{\mathfrak c}\right)
 \le {C_{\ref{lem:final entry capture}}^{(2)}}^m M_K
  \left(C_{\ref{thm:Fourier capture}}
          {C_{\ref{lem:final entry capture}}^{(1)}}L\varepsilon\right)^{r_b}.
\end{align*}
The same event $\mathcal B_{\mathfrak c}$ works for all these sets and scales.
\end{lemma}

\begin{proof}
Take $\mathcal B_{\mathfrak c}$ from
Corollary~\ref{cor:approximate kernel threshold}.
Lemma~\ref{lem:fixed image net} gives a deterministic net in $H_m$,
of size at most $C^mM_K$, lying within
$3C_{\ref{lem:final support covers}}^{(1)}\varepsilon$ of $K$.
A kernel witness has an approximant $y$ with
$\|Q_bA_my\|_2\le C\varepsilon\sqrt m$, using
$\|Q_bA_m\|_{\mathrm{HS}}\le m$.
Choose $C_{\ref{lem:final entry capture}}^{(1)}$ to dominate the neighborhood
constant and this error in both units $\sqrt m$ and $\sqrt{r_b}$.
Then \eqref{eq:final entry scale range} ensures
$y\in\mathcal U_{\delta_0/4}$ and, outside $\mathcal B_{\mathfrak c}$,
$T_L(y)<t=C_{\ref{lem:final entry capture}}^{(1)}\varepsilon$.

The subset of net points with $T_L(y)<t$ is deterministic on the fibre,
since its threshold uses only the fixed auxiliary law.
The witness event is contained in the union of
$\{\|Q_bA_my\|_2\le t\sqrt{r_b}\}$ over this subset.
Apply Corollary~\ref{cor:conditional hybrid capture} to each point and
sum. The single exceptional event works for every $K$ and scale.
\end{proof}

\begin{proposition}\label{prop:final incompressible}
There are positive constants
${c_{\ref{prop:final incompressible}}^{(1)}=c_{\ref{prop:final incompressible}}^{(1)}(\lambda)}$ and
${c_{\ref{prop:final incompressible}}^{(2)}=c_{\ref{prop:final incompressible}}^{(2)}(\lambda)}$, depending only on
$\lambda$ and the fixed choices above, such that, for
$n\le m\le(1+\kappa)n$ and
$e^{-{c_{\ref{prop:final incompressible}}^{(1)}}n}\le\beta<1$,
\begin{align*}
 \sup_{b\in H_m\cap\mathbf 1^\perp}
 \mathbb P\{\exists v\in H_m\cap
       \mathrm{Incomp}_{\delta_c,\rho_c}(m):
             Q_bA_mv=0,\ \rho_*(v)>\beta\}
 \le e^{-{c_{\ref{prop:final incompressible}}^{(2)}}n}.
\end{align*}
\end{proposition}
\begin{proof}
Fix $b$ and one block choice in $\mathfrak F$.
Work first on an arbitrary good atom $\mathfrak c$ of its own
$\mathcal H_{m;X,Y}$, with domain $\mathcal K_0$.
All nets and terminal subsets below are fixed on this
atom. Constants in the ensuing estimates are uniform
over good atoms and over $b$.

Fix $\delta_s\in(0,1/4)$, for example $\delta_s=1/8$.
For each eligible $T$, use the constant class
$\mathcal C_T(\mathcal K_0)$ in
\eqref{eq:final constant class}. Combine
\eqref{eq:final constant cover} with
Lemma~\ref{lem:final entry capture} at $\varepsilon=\rho_s$.
Provided \eqref{eq:final entry scale range} holds, the
result is bounded by
\begin{align*}
 C^m L^{r_b}\binom s{k_s}
       \rho_s^{\,s-k_s-1-\ell_b}.
\end{align*}
Here $C$ is absolute after the geometric choices, and is
independent of $\rho_s$. The exponent follows from
$r_b=m-\ell_b$; the mean on the large constant set costs
one dimension. Since $k_s\le\delta_s s$, $s\ge\vartheta m$,
and $\ell_b\le2$, this exponent is at least
$(1-\delta_s)\vartheta m-3$.
Also $\binom s{k_s}\le2^s\le2^m$.
Choose $\rho_s>0$ sufficiently small, after $L$, that the preceding bound is at most
$e^{-8m}$ for all large $m$, and that the upper bounds in
\eqref{eq:final entry scale range} hold.
The lower bound involving $\Delta_m$ then holds for all
large $m$ because $\rho_s$ is fixed.

Fix next
$0<\gamma\le\gamma_{\ref{lem:final concentration CLCD}}$
and let $c'':=\tfrac12c_{\ref{lem:final concentration CLCD}}$.
Choose $\mu>0$ sufficiently small after these parameters.
For a fixed support $T$, cover the non-almost-constant
directions by the dyadic classes
$\mathcal A_{T,D}(\mathcal K_0)$ with
\begin{align}
 D=2^j c''\sqrt s,\qquad j=0,1,2,\ldots,\qquad
 D\le D_{\max}:=
        2{C_{\ref{lem:final concentration CLCD}}(\gamma)}/\beta.\label{eq:final dyadic CLCD levels}
\end{align}
Lemma~\ref{lem:final concentration CLCD} supplies both endpoints
for every non-almost-constant restriction with $\rho_T(v)>\beta$, once
\begin{align*}
 \beta\ge2{C_{\ref{lem:final concentration CLCD}}(\gamma)}e^{-2\mu^2s}.
\end{align*}
If the upper endpoint is below the first level, there
are no such restrictions.

Combining \eqref{eq:final arithmetic cover} with
Lemma~\ref{lem:final entry capture} at $\varepsilon=\varepsilon_D$ gives
\begin{align*}
 C^mL^{r_b}\mu^{-3}D^3
          \varepsilon_D^{-\ell_b-2}(C\mu)^s.
\end{align*}
The ambient powers cancel: the two scalar meshes and the rank loss
leave $\varepsilon_D^{-\ell_b-2}$, while the full factor $(C\mu)^s$
from the support remains.
The constant $C$ here may depend on $c''$ and the already
fixed parameters, but is independent of $\mu,D,\beta$.
For $\mu<\min\{1,c''/10\}$, large $m$, and
$D\le D_{\max}$,
\begin{align*}
 \mu^{-3}D^3\varepsilon_D^{-\ell_b-2}
 \le C_\gamma\mu^{-7}\beta^{-7}.
\end{align*}
This uses $\ell_b\le2$ and $D\le C_\gamma/\beta$.
Choose $\mu$ so small that
\begin{align*}
 (CL)^m(C\mu)^s\le e^{-12m}
       \qquad(s\ge\vartheta m).
\end{align*}
At the same time require
${C_{\ref{lem:final entry capture}}^{(1)}}\mu/c''
<\min\{\delta_0/4,s_0/2\}$.
These choices are compatible because the constants in
the covering and approximation estimates are independent
of $\mu$.

We now choose a positive $c_\beta$ sufficiently small that
\begin{align*}
 c_\beta<
 \min\left\{
   \mu^2{c_{\ref{lem:Sigma fields}}^{(2)}},
   \tfrac12{c_{\ref{thm:uniform threshold control}}^{(1)}(L)},
   \tfrac17
 \right\}.
\end{align*}
For $\beta\ge e^{-c_\beta n}$ and all sufficiently large
$n$, the required scalar remainder bound holds
simultaneously for all eligible $T$.
Moreover, every level in \eqref{eq:final dyadic CLCD levels}
satisfies the lower scale condition, since
\begin{align*}
 {C_{\ref{lem:final entry capture}}^{(1)}}\varepsilon_D
 \ge c_{\gamma,\mu}\beta\sqrt s
 \ge\Delta_m.
\end{align*}
The upper scale conditions follow from
$D\ge c''\sqrt s$ and the preceding choice of $\mu$.
Thus every use of Lemma~\ref{lem:final entry capture}
is within its permitted range.
The arithmetic entropy bound and this choice of $c_\beta$ bound the probability
for each arithmetic class outside $\mathcal B_{\mathfrak c}$ by
$e^{-10m}$ for all sufficiently large $m$.

There are at most $2^{n-2}$ eligible supports, and at most
$C(1+n)$ levels for each support in
\eqref{eq:final dyadic CLCD levels}.
The constant and arithmetic classes exhaust every
$v\in\mathcal K_0$ with $\rho_*(v)>\beta$: the supremum
in \eqref{eq:rho-star} is over a finite family, and for
its chosen $T$ either $a_T(v)=0$, or $u_T(v)$ is almost
constant, or the two CLCD bounds place it in one of the
displayed dyadic classes. Consequently
\begin{align*}
 \mathbb P
 \{\exists v\in\mathcal K_0:
                    Q_bA_mv=0,\ \rho_*(v)>\beta\}
 \le e^{-{c_{\ref{thm:uniform threshold control}}^{(2)}}m}
       +2^{n-2}\left(e^{-8m}+C(1+n)e^{-10m}\right)
 \le e^{-c m}.
\end{align*}
Only the entry probabilities were summed over $T$ and $D$.
The threshold exception $\mathcal B_{\mathfrak c}$ occurs once in
this estimate.

Multiply the conditional estimate by
$\mathbf 1_{\mathcal E_{m;X,Y}}$ and average with respect to
$\mathcal H_{m;X,Y}$.
Theorem~\ref{thm:exact hybrid factorization} bounds the
probability of a bad certificate by $e^{-c'n}$.
Repeat this averaging separately for each of the finitely
many choices in $\mathfrak F$ and take their union.
No intersection of their exposures is used as conditioning
data. Their independent auxiliary orderings can be included
in the underlying probability space because the original
witness event does not depend on them.
Lemma~\ref{lem:final geometric domains} places every
incompressible solution in at least one of these domains.
This proves the proposition, with
${c_{\ref{prop:final incompressible}}^{(1)}}=c_\beta$ and a
sufficiently small
${c_{\ref{prop:final incompressible}}^{(2)}}>0$.
Every bound was uniform in the fixed target $b$, which
justifies the supremum in the statement.
\end{proof}

\subsection{The witness bound}
\begin{theorem}\label{thm:final witness bound}
For the fixed $\kappa$ in \eqref{eq:final short extension choice},
there are positive constants
${c_{\ref{thm:final witness bound}}^{(1)}=c_{\ref{thm:final witness bound}}^{(1)}(\lambda)}$ and
${c_{\ref{thm:final witness bound}}^{(2)}=c_{\ref{thm:final witness bound}}^{(2)}(\lambda)}$, depending only on $\lambda$,
such that
\begin{align}
 q_m(\beta)\le e^{-{c_{\ref{thm:final witness bound}}^{(2)}}n}\label{eq:final witness bound}
\end{align}
for every $n\le m\le n+\lfloor\kappa n\rfloor-1$ and
$e^{-{c_{\ref{thm:final witness bound}}^{(1)}}n}\le\beta<1$,
provided $n$ is sufficiently large.
\end{theorem}
\begin{proof}
Normalize a witness in \eqref{eq:qm} to $x=v/\|v\|_2$.
Then $Q_bA_mx=0$ and $\rho_*(x)=\rho_*(v)>\beta$.
Propositions~\ref{prop:final compressible} and
\ref{prop:final incompressible} bound the two parts of the unit
sphere by exponentially small probabilities, uniformly in $b$.
Add them and take the supremum over $b$, with
$c_{\ref{thm:final witness bound}}^{(1)}
=c_{\ref{prop:final incompressible}}^{(1)}$.
\end{proof}
\begin{proof}[Proof of Theorem~\ref{thm:main}]
Use Theorem~\ref{thm:reduction} with the fixed $\kappa$
chosen in \eqref{eq:final short extension choice}.
Choose $\zeta>0$ so small that
\begin{align*}
 \zeta<
 \min\left\{
 {c_{\ref{thm:reduction}}^{(1)}},
 {c_{\ref{thm:final witness bound}}^{(1)}},
 \frac{{c_{\ref{thm:final witness bound}}^{(2)}}}
             {2{C_{\ref{thm:reduction}}^{(2)}}}
 \right\},
 \qquad \beta=e^{-\zeta n}.
\end{align*}
For all sufficiently large $n$, this $\beta$ lies in the
ranges of both the reduction and
Theorem~\ref{thm:final witness bound}.
Substituting \eqref{eq:final witness bound} into
\eqref{eq:reduction}, and using
$\lfloor\kappa n\rfloor\le n$, gives
\begin{align*}
 \mathbb P(A_n\text{ is singular}) \le e^{-{c_{\ref{thm:reduction}}^{(1)}}n}+n^{{C_{\ref{thm:reduction}}^{(1)}}+1}
   \left[
    e^{-({c_{\ref{thm:final witness bound}}^{(2)}}
                   -{C_{\ref{thm:reduction}}^{(2)}}\zeta)n}
       +e^{-{c_{\ref{thm:reduction}}^{(2)}}\zeta n}
   \right].
\end{align*}
Both exponents in brackets are positive by the choice of $\zeta$.
Absorbing the polynomial factor and the finite sum into
a smaller positive exponential constant proves
Theorem~\ref{thm:main}.
\end{proof}

\section*{Acknowledgments}

\par\medskip
\noindent\textbf{Statement on AI use.}
The authors used AI tools, including models from OpenAI's GPT-5.5 and GPT-5.6, during the preparation of this work.

The main mathematical strategy and the underlying ideas of the proof were developed by the authors. In particular, the reduction from the random \(d\)-regular graph, after conditioning on suitable events, to an auxiliary random model, and the subsequent Fourier analytic transformation underlying the inversion of randomness argument, originated with the authors. AI tools were used during this stage to assist with technical calculations and intermediate derivations.

A further difficulty arises from the absence of a sufficiently strong spectral norm bound for the relevant matrices. The authors introduced the early idea behind the stopping-chain argument, namely to successively move the approximating points to larger or smaller scales in order to overcome this obstruction. AI assistance was used substantially in exploring and refining this mechanism, including its quantitative implementation, and helped develop the initial idea into the stopping-chain argument appearing in the final proof.

The manuscript was written by the authors, with AI tools also used for language editing and stylistic refinement. The authors independently reviewed and verified the arguments and take full responsibility for the correctness and content of the paper.

\appendix
\section{Matching estimates}\label{app:matching Gaussian estimates}
We prove the matching inputs to
Theorem~\ref{thm:local augmented small ball}. All supports and
coefficient data are fixed; only the matching and lazy signs are random.

\begin{lemma}\label{lem:matching captures energy}
Let $\sum_{i\in T}r_i=0$, $\max_i\|r_i\|_2\le L_0$ and
$E_r=\sum_i\|r_i\|_2^2$, where $|T|\ge2$ and $L_0>0$.
For the matching in Definition~\ref{def:lazy matching walk},
\begin{align*}
 \mathbb P_\omega\left(
 \sum_{\{i,j\}\in\mathcal M_T}\|r_i-r_j\|_2^2\le E_r/4\right)
 \le2e^{-E_r/(C_{\ref{lem:matching captures energy}}L_0^2)}
\end{align*}
with an absolute constant $C_{\ref{lem:matching captures energy}}>0$.
\end{lemma}
\begin{proof}
If $|T|=2p$, generate the matching by a uniform $p$-set $I$ and
a uniform bijection $\sigma:I\to T\setminus I$. Conditional on $I$,
the expected energy is
$E_r+(2/p)\|\sum_{i\in I}r_i\|_2^2\ge E_r$.
The permutation lower-tail inequality
\cite[Proposition~1.1]{C-prtf}, applied to
$\|r_i-r_j\|_2^2/(4L_0^2)\in[0,1]$, bounds the probability of
energy below $E_r/2$ by $2e^{-E_r/(80L_0^2)}$.
For odd $s=|T|$, condition on the unmatched vertex $j$ and center
the remaining vectors by adding $r_j/(s-1)$. Their norms are at
most $2L_0$, their differences are unchanged, and their energy is
$E_r-s\|r_j\|_2^2/(s-1)\ge E_r-2L_0^2$.
Apply the even case when $E_r\ge4L_0^2$; otherwise enlarge the
absolute constant so that the bound is trivial.
\end{proof}

\begin{proposition}\label{prop:matching estimates}
For $|T|\ge2$ and $\nu\in(0,1/4]$, there is
$c_{\ref{prop:matching estimates}}^{(1)}
=c_{\ref{prop:matching estimates}}^{(1)}(\nu)>0$ such that
\begin{align*}
 0\le\phi_{\omega_{T,\nu}}(x)
 \le3e^{-c_{\ref{prop:matching estimates}}^{(1)}
                            \|x_T\|_{\mathbb D}^2}.
\end{align*}
If $F$ satisfies \eqref{eq:local soft hypotheses}, then, for $u\ge0$,
\begin{align*}
 \gamma_k\{\theta:\|F\theta\|_{\mathbb D,T\oplus T}^2\le u\}
 \le C_{\ref{prop:matching estimates}}^{(2)}
 e^{C_{\ref{prop:matching Fourier comparison}}\nu u
                 -c_{\ref{prop:matching estimates}}^{(2)}k},
\end{align*}
where $c_{\ref{prop:matching estimates}}^{(2)}
=c_{\ref{prop:matching estimates}}^{(2)}(\kappa,\nu)>0$ and
$C_{\ref{prop:matching estimates}}^{(2)}
=C_{\ref{prop:matching estimates}}^{(2)}(\kappa,\nu)\ge1$.
For $k=0$, $\gamma_0$ is unit mass on $\mathbb R^0$.
\end{proposition}
\begin{proof}
For the first assertion, center
$u_i=(\cos(2\pi x_i),\sin(2\pi x_i))$ to obtain rows $r_i$ of
norm at most two and energy
$E_r=(4/s)\sum_{i<j}\sin^2(\pi(x_i-x_j))
\ge8\|x_T\|_{\mathbb D}^2$.
Except with probability $2e^{-cE_r}$, the preceding lemma gives
$\sum_{\mathcal M_T}\sin^2(\pi(x_i-x_j))\ge E_r/16$.
On this event the characteristic product
\eqref{eq:matching characteristic product} is at most
$e^{-\nu E_r/8}$; on the complement it is at most one.
Averaging proves the Fourier bound.

For the second assertion assume $k\ge1$, and choose $b$ with
$\|F_b\|_{\mathrm{HS}}^2\ge\kappa k/2$.
The preceding lemma, applied to its centered rows, gives outside
probability $Ce^{-ck}$ a matching-difference matrix $B$ with
$\|B\|_{\mathrm{HS}}^2\ge\kappa k/8$ and
$\|B\|_{\mathrm{op}}\le2\sqrt2$.
For independent lazy signs $\epsilon$, put
$X=\|B^T\epsilon-z\|_2^2$. The moment bounds
$\mathbb E_\omega X=\nu\|B\|_{\mathrm{HS}}^2+\|z\|_2^2$
and $\mathbb E_\omega X^2\le C_\nu(\mathbb E_\omega X)^2$
give, by Paley--Zygmund,
$\mathbb P_\omega(\|B^T\epsilon-z\|_2\ge a_1\sqrt k)\ge a_2$
for fixed $a_1,a_2>0$, uniformly in $z$.
The function $f(\epsilon)=\|B^T\epsilon-z\|_2$ is convex and
$2\sqrt2$-Lipschitz. Talagrand's product concentration
\cite[Section~4.1]{Talagrand} gives, for a median $m_f$,
\begin{align*}
 \mathbb P_\omega(f\ge m_f+u)+\mathbb P_\omega(f\le m_f-u)
 \le Ce^{-cu^2},\qquad u\ge0.
\end{align*}
If $m_f<a_1\sqrt k/2$, the upper-tail bound contradicts the
Paley--Zygmund estimate for large $k$. Hence $m_f\ge a_1\sqrt k/2$,
and the lower-tail bound gives
$\mathbb P_\omega(f\le a_1\sqrt k/4)\le Ce^{-ck}$.
A larger prefactor covers bounded $k$.
Averaging over the matching and conditioning on the other walk
give the same small-ball bound for $F^T(\omega^{(0)},\omega^{(1)})$,
whose Gaussian moment is therefore at most $Ce^{-ck}$.
The matching lower bound \eqref{eq:matching Fourier lower} and
Gaussian inversion now give
\begin{align*}
 \gamma_k\{\|F\theta\|_{\mathbb D,T\oplus T}^2\le u\}
 &\le e^{C_{\ref{prop:matching Fourier comparison}}\nu u}
   \int\phi_{F^T(\omega^{(0)},\omega^{(1)})}(\theta)\,d\gamma_k(\theta)\\
 &=e^{C_{\ref{prop:matching Fourier comparison}}\nu u}
   \mathbb E_\omega e^{-\pi\|F^T(\omega^{(0)},\omega^{(1)})\|_2^2}.
\end{align*}
This proves the claim; $k=0$ is immediate.
\end{proof}

\section{Arithmetic exceptions}\label{app:product box estimates}
We prove Lemma~\ref{lem:box CLCD exceptions} under the original
product measure. We use one elementary interval count: if $\xi$ has
maximal atom at most $(c_I M)^{-1}$ and support in an interval of
diameter at most $C_I M$, then for $0<\eta<1$ and $\varphi M\ge c_*>0$,
\begin{align}
 \mathbb P\{\|\varphi(\xi-u)\|_{\mathbb R/\mathbb Z}\le\eta\}
 \le C(\eta+\varphi+M^{-1}),\label{eq:interval torus count}
\end{align}
uniformly in $u$, with $C=C(c_I,C_I,c_*)$.
Indeed, the image interval meets $O(\varphi M+1)$ integer
$\eta$-neighborhoods, each containing at most $2\eta/\varphi+2$
preimages in the original integer interval.

\begin{proof}[Proof of Lemma~\ref{lem:box CLCD exceptions}]
All constants in this proof depend only on the fixed data in the
statement until $A$ is chosen. Write $s=|T|$. Then $s\ge2\sigma D$,
and the difference identity \eqref{eq:difference-identities} gives
\begin{align*}
 \|D_Tu\|_2&\le B_*DM,\qquad
 \|P_{T,0}u\|_2\ge2a\sigma\sqrt D\,M.
\end{align*}
For $v=r_*u_T/\sqrt{N_*}$ and
$0<\vartheta<\sqrt{N_*}/(4br_*M)$, every difference between a
prescribed positive and a prescribed negative coordinate of
$\vartheta v$ has absolute value less than $1/2$. Consequently,
\begin{align*}
 \operatorname{dist}(\vartheta D_Tv,\mathbb Z^{\binom{s}{2}})
 &\ge \frac{2a\sigma r_*\vartheta DM}{\sqrt{N_*}}
 >\gamma_*\vartheta\|D_Tv\|_2.
\end{align*}
Thus
\begin{align*}
 \operatorname{CLCD}_{\mu s,\gamma_*}^{T}(v)
          \ge\frac{\sqrt{N_*}}{4br_*M}
 \qquad(0<\mu<1).
\end{align*}
In particular, the event in \eqref{eq:box CLCD exceptions} is empty
if $M<\sqrt{N_*}/(128br_*)$.

In the remaining range, if this CLCD is at most $16$, the definition
of its infimum supplies a witness $0<\vartheta\le17$.
Set $\varphi=r_*\vartheta/\sqrt{N_*}$. By this deterministic lower bound,
\begin{align*}
 \frac1{4bM}\le\varphi\le\frac{17r_*}{\sqrt{N_*}},\qquad
 \operatorname{dist}(\varphi D_Tu,\mathbb Z^{\binom{s}{2}})<\mu s.
\end{align*}
Take a deterministic mesh of spacing at most $c\mu/M$ in this
frequency interval, including its endpoints, where $c>0$ is a
sufficiently small fixed constant. It has at most
$C(1+M/(\mu\sqrt{N_*}))$ points. For the nearest mesh point
$\varphi_0$, the bound $\|D_Tu\|_2\le B_*DM$ and $s\ge2\sigma D$
imply
\begin{align*}
 \sum_{\{i,j\}\subseteq T}
   \|\varphi_0(u_i-u_j)\|_{\mathbb R/\mathbb Z}^2
 \le C\mu^2s^2.
\end{align*}
There is therefore an index $j\in T$ such that the sum over
$i\in T\setminus\{j\}$ is at most $C\mu^2s$.
Choose a fixed sufficiently large $K$, depending on $\sigma$.
At most $\sigma D/4$ of these indices have
$\|\varphi_0(u_i-u_j)\|_{\mathbb R/\mathbb Z}>K\mu$.
For large $D$, there is a subset of $P_{T,+}\setminus\{j\}$ of
size at least $\lfloor\sigma D/2\rfloor$ on which all these
distances are at most $K\mu$.

Fix $T$, the mesh point, the anchor $j$, and such a subset.
After conditioning only on $U_j$, its other coordinates remain
independent. The interval estimate \eqref{eq:interval torus count} bounds the
probability for each of them by
$C(\mu+N_*^{-1/2}+M^{-1})$. A union bound over all these choices gives
\begin{align*}
 &\mathbb P_Z\{\exists T\in\mathscr T:
       \operatorname{CLCD}_{\mu|T|,\gamma_*}^{T}
       (r_*U_T/\sqrt{N_*})\le16\} \\
 &\quad\le C N_*^2\left(1+\frac{M}{\mu\sqrt{N_*}}\right)2^D
       \left[C(\mu+N_*^{-1/2}+M^{-1})\right]
                  ^{\lfloor\sigma D/2\rfloor}.
\end{align*}
First choose $\mu$ sufficiently small in terms of $A$.
Since $D\ge\rho N_*$, the last two factors are then at most
$A^{-4N_*}$ for large $N_*$ in the present range
$M\ge\sqrt{N_*}/(128br_*)$.
Next choose ${c_{\ref{lem:box CLCD exceptions}}(A)}$ sufficiently
small and increase the lower dimension to absorb the prefactor
into $A^{2N_*}$. This proves the assertion. Decreasing $\mu$
only increases the CLCD and decreases the exceptional event.
\end{proof}

\printbibliography

@article{CJMS-pi,
  author  = {Campos, Marcelo and Jenssen, Matthew and Michelen, Marcus and Sahasrabudhe, Julian},
  title   = {The least singular value of a random symmetric matrix},
  journal      = {Forum of Mathematics, Pi},
  volume  = {12},
  year    = {2024},
  pages   = {e3},
  doi         = {10.1017/fmp.2023.29}
}

@article{CJMS-jams,
  author  = {Campos, Marcelo and Jenssen, Matthew and Michelen, Marcus and Sahasrabudhe, Julian},
  title   = {The singularity probability of a random symmetric matrix is exponentially small},
  journal      = {Journal of the American Mathematical Society},
  volume  = {38},
  number  = {1},
  year    = {2025},
  pages   = {179--224},
  doi     = {10.1090/jams/1042},
  eprint  = {2105.11384},
  eprinttype = {arxiv}
}

@article{C-prtf,
  author  = {Chatterjee, Sourav},
  title   = {Stein's method for concentration inequalities},
  journal      = {Probability Theory and Related Fields},
  volume  = {138},
  number  = {1--2},
  year    = {2007},
  pages   = {305--321},
  doi         = {10.1007/s00440-006-0029-y}
}

@article{Cook-ptrf,
  author  = {Cook, Nicholas A.},
  title   = {On the singularity of adjacency matrices for random regular digraphs},
  journal = {Probability Theory and Related Fields},
  volume  = {167},
  number  = {1--2},
  year    = {2017},
  pages   = {143--200},
  doi     = {10.1007/s00440-015-0679-8},
  eprint  = {1411.0243},
  eprinttype = {arxiv}
}

@article{CTV-duke,
  author  = {Costello, Kevin P. and Tao, Terence and Vu, Van},
  title   = {Random symmetric matrices are almost surely nonsingular},
  journal = {Duke Mathematical Journal},
  volume  = {135},
  number  = {2},
  year    = {2006},
  pages   = {395--413},
  doi     = {10.1215/S0012-7094-06-13527-5}
}

@article{Huang-duke,
  author  = {Huang, Jiaoyang},
  title   = {Invertibility of adjacency matrices for random $d$-regular graphs},
  journal      = {Duke Mathematical Journal},
  volume  = {170},
  number  = {18},
  year    = {2021},
  pages   = {3977--4032},
  doi     = {10.1215/00127094-2021-0006},
  eprint  = {1807.06465},
  eprinttype = {arxiv}
}

@article{JSS-imrn,
  author  = {Jain, Vishesh and Sah, Ashwin and Sawhney, Mehtaab},
  title   = {The smallest singular value of dense random regular digraphs},
  journal      = {International Mathematics Research Notices},
  volume  = {2022},
  number  = {24},
  year    = {2022},
  pages   = {19300--19334},
  doi     = {10.1093/imrn/rnab247},
  eprint  = {2008.04755},
  eprinttype = {arxiv}
}

@article{KKS-jams,
  author  = {Kahn, Jeff and Koml{\'o}s, J{\'a}nos and Szemer{\'e}di, Endre},
  title   = {On the probability that a random {$\pm1$}-matrix is singular},
  journal = {Journal of the American Mathematical Society},
  volume  = {8},
  number  = {1},
  year    = {1995},
  pages   = {223--240}
}

@article{regular-jmaa,
  author  = {Litvak, Alexander E. and Lytova, Anna and Tikhomirov, Konstantin and Tomczak-Jaegermann, Nicole and Youssef, Pierre},
  title   = {Adjacency matrices of random digraphs: singularity and anti-concentration},
  journal = {Journal of Mathematical Analysis and Applications},
  volume  = {445},
  number  = {2},
  year    = {2017},
  pages   = {1447--1491},
  doi     = {10.1016/j.jmaa.2016.08.020},
  eprint  = {1511.00113},
  eprinttype = {arxiv}
}

@article{regular-ptrf,
  author  = {Litvak, Alexander E. and Lytova, Anna and Tikhomirov, Konstantin and Tomczak-Jaegermann, Nicole and Youssef, Pierre},
  title   = {The smallest singular value of a shifted {$d$}-regular random square matrix},
  journal = {Probability Theory and Related Fields},
  volume  = {173},
  number  = {3--4},
  year    = {2019},
  pages   = {1301--1347},
  doi     = {10.1007/s00440-018-0852-y},
  eprint  = {1707.02635},
  eprinttype = {arxiv}
}

@article{Livshyts-jam,
  author  = {Livshyts, Galyna V.},
  title   = {The smallest singular value of heavy-tailed not necessarily i.i.d. random matrices via random rounding},
  journal      = {Journal d'Analyse Math\'ematique},
  volume  = {145},
  number  = {1},
  year    = {2021},
  pages   = {257--306},
  doi     = {10.1007/s11854-021-0183-2},
  eprint  = {1811.07038},
  eprinttype = {arxiv}
}

@article{McKay-cpc,
  author  = {McKay, Brendan D.},
  title   = {Subgraphs of dense random graphs with specified degrees},
  journal      = {Combinatorics, Probability and Computing},
  volume  = {20},
  number  = {3},
  year    = {2011},
  pages   = {413--433},
  doi     = {10.1017/S0963548311000034},
  eprint  = {1002.3018},
  eprinttype = {arxiv}
}

@article{McKay-ejc,
  author  = {McKay, Brendan D. and Wormald, Nicholas C.},
  title   = {Asymptotic enumeration by degree sequence of graphs of high degree},
  journal      = {European Journal of Combinatorics},
  volume  = {11},
  year    = {1990},
  pages   = {565--580},
  number      = {6},
  doi         = {10.1016/S0195-6698(13)80042-X}
}

@article{Meszaros-tams,
  author  = {M{\'e}sz{\'a}ros, Andr{\'a}s},
  title   = {The distribution of sandpile groups of random regular graphs},
  journal = {Transactions of the American Mathematical Society},
  volume  = {373},
  number  = {9},
  year    = {2020},
  pages   = {6529--6594},
  doi     = {10.1090/tran/8127},
  eprint  = {1806.03736},
  eprinttype = {arxiv}
}

@misc{NW-arxiv,
  author  = {Nguyen, Hoi H. and Wood, Melanie Matchett},
  title   = {Cokernels of adjacency matrices of random {$r$}-regular graphs},
  howpublished = {\href{https://arxiv.org/abs/1806.10068}{arXiv:\allowbreak 1806.10068}},
  year    = {2018},
  note    = {Preprint},
  eprint  = {1806.10068},
  eprinttype = {arxiv}
}

@article{Roos,
  author  = {Roos, Bero},
  title   = {New inequalities for permanents and hafnians and some generalizations},
  journal = {Linear and Multilinear Algebra},
  volume  = {73},
  number  = {8},
  year    = {2025},
  pages   = {1634--1667},
  doi     = {10.1080/03081087.2024.2436061},
  eprint  = {1906.06176},
  eprinttype = {arxiv}
}

@article{RV-adv,
  author  = {Rudelson, Mark and Vershynin, Roman},
  title   = {The {Littlewood--Offord} problem and invertibility of random matrices},
  journal = {Advances in Mathematics},
  volume  = {218},
  number  = {2},
  year    = {2008},
  pages   = {600--633},
  doi     = {10.1016/j.aim.2008.01.010}
}

@article{RV-imrn,
  author  = {Rudelson, Mark and Vershynin, Roman},
  title   = {Small ball probabilities for linear images of high-dimensional distributions},
  journal      = {International Mathematics Research Notices},
  volume  = {2015},
  number  = {19},
  year    = {2015},
  pages   = {9594--9617},
  doi     = {10.1093/imrn/rnu243}
}

@article{Talagrand,
  author  = {Talagrand, Michel},
  title   = {Concentration of measure and isoperimetric inequalities in product spaces},
  journal      = {Publications Math\'ematiques de l'IH\'ES},
  volume  = {81},
  year    = {1995},
  pages   = {73--205},
  doi         = {10.1007/BF02699376}
}

@article{TV-aom,
  author  = {Tao, Terence and Vu, Van H.},
  title   = {Inverse {Littlewood--Offord} theorems and the condition number of random discrete matrices},
  journal = {Annals of Mathematics},
  volume  = {169},
  number  = {2},
  year    = {2009},
  pages   = {595--632},
  doi     = {10.4007/annals.2009.169.595}
}

@article{Tikhomirov-aom,
  author  = {Tikhomirov, Konstantin},
  title   = {Singularity of random {Bernoulli} matrices},
  journal = {Annals of Mathematics},
  volume  = {191},
  number  = {2},
  year    = {2020},
  pages   = {593--634},
  doi     = {10.4007/annals.2020.191.2.6}
}

@misc{Tran-arxiv,
  author = {Tran, Tuan},
  title  = {The smallest singular value of random combinatorial matrices},
  year   = {2020},
  eprint = {2007.06318},
  eprinttype = {arxiv},
  howpublished= {\href{https://arxiv.org/abs/2007.06318}{arXiv:\allowbreak 2007.06318}},
  note        = {Preprint}
}

@article{Hoeffding-jasa,
  author  = {Hoeffding, Wassily},
  title   = {Probability inequalities for sums of bounded random variables},
  journal = {Journal of the American Statistical Association},
  volume  = {58},
  number  = {301},
  year    = {1963},
  pages   = {13--30},
  doi     = {10.1080/01621459.1963.10500830}
}

@misc{HSW-arxiv,
  author  = {Hu, Xuanang and Song, Zeyan and Wu, Xinglong},
  title   = {Sharp singularity asymptotics for discrete random matrices},
  year    = {2026},
  eprint  = {2609.39463},
  eprinttype = {arxiv},
  howpublished = {\href{https://arxiv.org/abs/2609.39463}{arXiv:\allowbreak 2609.39463}},
  note    = {Preprint}
}
\end{document}